\documentclass[12pt]{article}
\usepackage[margin=1in]{geometry}
\usepackage[latin9]{inputenc}
\usepackage{geometry}
\usepackage{setspace}
\usepackage{color}
\usepackage{booktabs}
\usepackage{amsmath}
\usepackage{amssymb}
\usepackage[authoryear,round]{natbib}
\usepackage{titletoc}

\usepackage{graphicx}
\usepackage[dvipsnames,svgnames,x11names,hyperref]{xcolor}
\usepackage{titletoc}
\definecolor{DarkBlue}{rgb}{0,0,0.9}
\usepackage{enumitem}
\usepackage{footmisc}
 \usepackage{setspace}
\usepackage{subcaption}
\usepackage{float}
\usepackage{etoolbox}
\AtBeginEnvironment{tabular}{\doublespacing}
\AtBeginEnvironment{figure}{\doublespacing}
\usepackage{footmisc}

\usepackage[unicode=true,
 bookmarks=false,
 breaklinks=false,pdfborder={0 0 1},colorlinks=true]
 {hyperref}
\hypersetup{citecolor=Maroon,urlcolor=DarkBlue,linkcolor=DarkBlue}
\makeatletter

\providecommand{\tabularnewline}{\\}

\let\mypdfximage\pdfximage
\def\pdfximage{\immediate\mypdfximage}

\usepackage{amsthm}\usepackage{array}
\newcolumntype{L}[1]{>{\raggedright\let\newline\\\arraybackslash\hspace{0pt}}m{#1}}
\newcolumntype{C}[1]{>{\centering\let\newline\\\arraybackslash\hspace{0pt}}m{#1}}
\newcolumntype{R}[1]{>{\raggedleft\let\newline\\\arraybackslash\hspace{0pt}}m{#1}}

\newtheorem{proposition}{Proposition}
\newtheorem{assumption}{Assumption}
\newtheorem{lemma}{Lemma}

\newtheorem{remark}{Remark}

\makeatother

\begin{document}
%\startcontents[sections]
%\printcontents[sections]{l}{1}{\setcounter{tocdepth}{2}}

\title{ Omitted variable bias of Lasso-based inference methods: A finite sample analysis\thanks{ This version: \today. First version: March 20, 2019. Alphabetical ordering; both authors contributed equally to this work. This paper was previously circulated as ``Behavior of Lasso and Lasso-based inference under limited variability'' and ``Omitted variable bias of Lasso-based inference methods under limited variability: A finite sample analysis''.  We would like to thank the editor (Xiaoxia Shi), anonymous referees, St\'ephane Bonhomme, Gordon Dahl, Graham Elliott, Michael Jansson, Michal Kolesar, Ulrich M\"uller, Andres Santos, Azeem Shaikh, Aman Ullah and seminar participants for their comments. We are especially grateful to Yixiao Sun and Jeffrey Wooldridge for providing extensive feedback. W\"uthrich is also affiliated with CESifo and the ifo Institute. Zhu acknowledges a start-up fund from the Department of Economics at UCSD and the Department of Statistics and the Department of Computer Science at Purdue University, West Lafayette.
}} 

\author{  Kaspar W\"uthrich\thanks{  Department of Economics, University of California, San Diego, 9500 Gilman Dr. La Jolla, CA 92093. Emails: \url{kwuthrich@ucsd.edu} and \url{yiz012@ucsd.edu}.} \quad{}\quad{}Ying Zhu\footnotemark[2]}

\date{Final author version, accepted at \emph{The Review of Economics and Statistics}\vspace{-1cm} }

\maketitle

\begin{abstract}
\normalsize
We study the finite sample behavior of Lasso-based inference methods such as post double Lasso and debiased Lasso. We show that these methods can exhibit substantial omitted variable biases (OVBs) due to Lasso not selecting relevant controls. This phenomenon can occur even when the coefficients are sparse and the sample size is large and larger than the number of controls. Therefore, relying on the existing asymptotic inference theory can be problematic in empirical applications. We compare the Lasso-based inference methods to modern high-dimensional OLS-based methods and provide practical guidance. 

\noindent\textbf{Keywords:} Lasso, post double Lasso, debiased Lasso, OLS, omitted variable bias, size distortions, finite sample analysis

\noindent\textbf{JEL codes:} C21, C52, C55
\end{abstract}

\section{Introduction}

Since their introduction, post double Lasso \citep{belloni2014inference} and debiased Lasso \citep{javanmard2014confidence,vandergeer2014asymptotically,Zhang_Zhang} have quickly become the most popular inference methods for problems with many control variables. Given the rapidly growing  theoretical\footnote{\begin{doublespace} \normalsize See, for example, \citet{farrell2015robust}, \citet{belloni2017program}, \citet{zhang2017simultaneous}, \citet{chernozhukov2018double}, \citet{caner2018asymptotically} among others. \end{doublespace} \singlespacing} and applied\footnote{\begin{doublespace}\normalsize  See, for example, \citet{chen2015can}, \citet{decker2016health}, \citet{schmitz2017informal}, \citet{breza2019social}, \citet{jones2019what}, \citet{cole2020mobilizing}, and \citet{enke2020moral} among others.\end{doublespace}} literature on these methods, it is crucial to take a step back and examine the performance of these procedures in empirically relevant settings and to better understand their merits and limitations relative to other alternatives.

In this paper, we study the performance of post double Lasso and debiased Lasso from two different angles. First, we develop a theory of the finite sample behavior of post double Lasso and debiased Lasso. We show that under-selection of the Lasso can lead to substantial omitted variable biases (OVBs) in finite samples. Our theoretical results on the OVBs are non-asymptotic; they complement but do not contradict the existing asymptotic theory. Second, we conduct extensive simulations and two empirical studies to investigate the performance of Lasso-based inference methods in practically relevant settings. We find that the OVBs can render the existing asymptotic approximations inaccurate and lead to size distortions.

Following the literature, we consider the following standard linear model 
\begin{eqnarray}
Y_{i} & = & D_{i}\alpha^{*}+X_{i}\beta^{*}+\eta_{i},\quad i=1,\dots,n,\label{eq:main-y}\\
D_{i} & = & X_{i}\gamma^{*}+v_{i},\quad i=1,\dots,n.\label{eq:main-d}
\end{eqnarray}
Here $Y_{i}$ is the outcome, $D_{i}$ is the scalar treatment variable
of interest, and $X_{i}$ is a $(1\times p)$-dimensional vector of control variables. 
In the main text, we focus on the performance of post double Lasso for estimating and making inferences (e.g., constructing confidence intervals) on the treatment effect $\alpha^{\ast}$ in settings where $p$ can be larger than or comparable to $n$. We present results for debiased Lasso in Appendix \ref{app: debiased lasso}.

Post double Lasso consists of two Lasso selection steps: a Lasso regression of $Y_{i}$ on $X_{i}$ and a Lasso regression of $D_{i}$ on $X_{i}$. In the
third step, the estimator of $\alpha^{\ast}$, $\tilde{\alpha}$,
is the OLS regression of $Y_{i}$ on $D_{i}$ and the
union of controls selected in the two Lasso steps. 
OVB arises in post double Lasso whenever the relevant controls (i.e., the controls with non-zero coefficients) are selected in neither Lasso steps, a situation we refer to as \emph{double under-selection}. Results that can explain when and why double under-selection occurs are scarce in the existing literature. This paper shows theoretically that this phenomenon can even occur in simple examples with classical assumptions (e.g., normal homoscedastic errors, orthogonal designs for the relevant controls), which are often viewed as favorable to the performance of the Lasso. 
We prove that if the products of the absolute values of the non-zero coefficients and the variances of the controls are no greater than half the regularization parameters derived based on standard Lasso theory\footnote{\begin{doublespace}\normalsize Note that the existing Lasso theory requires the regularization parameter to exceed a certain threshold, which depends on the standard deviations of the noise and the covariates.\end{doublespace}\singlespacing}, Lasso fails to select these controls in both steps with high probability.\footnote{\begin{doublespace}\normalsize The ``half the regularization parameters" type of condition on the magnitude of non-zero coefficients was independently discovered in \citet{lahiri2021necessary}. We are grateful to an anonymous referee for making us aware of this paper. Our proof strategies differ from the asymptotic ones in \citet{lahiri2021necessary} and allow us to derive an explicit lower bound with meaningful constants for the probability of under-selection for fixed $n$, which is needed for deriving an explicit formula for the OVB lower bound. While some of the arguments in \citet{lahiri2021necessary} can be made non-asymptotic, one of their core arguments for showing necessary conditions for variable selection consistency of the Lasso relies on $n$ tending to infinity. It is not clear that such an argument can lead to an explicit lower bound with meaningful constants for the probability of under-selection. On the other hand, our non-asympototic argument can easily lead to asymptotic conclusions.\end{doublespace}\singlespacing} This result allows us to derive the first non-asymptotic lower bound formula in the literature for the OVB of the post double Lasso estimator $\tilde{\alpha}$. Our lower bound provides explicit universal constants, which are essential for understanding the finite sample behavior of post double Lasso and its limitations.

The OVB lower bound is characterized by the interplay between the probability of double under-selection and the magnitude of the coefficients corresponding to the relevant controls in \eqref{eq:main-y}--\eqref{eq:main-d}. In particular, there are three regimes: (i) the omitted controls have coefficients whose magnitudes are large enough to cause substantial bias; (ii) the omitted controls have coefficients whose magnitudes are too small to cause substantial bias; (iii) the magnitude of the relevant coefficients is large enough such that the corresponding controls are selected with high-probability. We illustrate these three regimes in Figure \ref{fig:illu_intro}, which plots the bias of post double Lasso and the number of selected control variables as a function of the magnitude of the non-zero coefficients. \textbf{[FIGURE \ref{fig:illu_intro} HERE.]}

Our theoretical analysis of the OVB has important implications for inference procedures based on the post double Lasso. \citet{belloni2014inference} show that $\sqrt{n}(\tilde\alpha-\alpha^\ast)$ is asymptotically normal with zero mean. We show that in finite samples, the OVB lower bound can be more than twice as large as the standard deviation obtained from the asymptotic distribution in \citet{belloni2014inference}. This is true even when $n$ is much larger than $p$ and $\beta^\ast$ and $\gamma^\ast$ are sparse. To illustrate, assume that \eqref{eq:main-y} and \eqref{eq:main-d} share the same set of $k$ non-zero coefficients and set $(n,\,p)=(14238,\,384)$ as in \citet[][Section 3]{angrist2019machine}, who use post double Lasso to estimate the effect of elite colleges. The ratio of our OVB lower bound to the standard deviation in \citet{belloni2014inference} is $0.27$ if $k=1$ and $2.4$ if $k=10$. This example shows that the requirement on the sparsity parameter $k$ for the OVBs to be negligible is quite stringent. We emphasize that our findings do not contradict the existing results on the asymptotic distribution of post double Lasso in \citet{belloni2014inference}. Rather, our results suggest that the OVBs can make the asymptotic zero-mean approximation of $\sqrt{n}(\tilde\alpha-\alpha^\ast)$ inaccurate in finite samples.

To better understand the practical implications of the OVB of post double Lasso, we perform extensive simulations. Our simulation results can be summarized as follows. (i) Large OVBs are persistent across a range of empirically relevant settings and can occur even when $n$ is large and larger than $p$, and the sparsity parameter $k$ is small. (ii) The OVBs can lead to invalid inferences and under-coverage of confidence intervals. (iii) The performance of post double Lasso varies substantially across different popular choices of regularization parameters, and no single choice outperforms the others across all designs. While it may be tempting to choose a smaller regularization parameter than the standard recommendation in the literature to mitigate under-selection, we find that this idea does not work in general and can lead to rather poor performance.

In addition to the simulations, we consider two empirical applications: the analysis of the effect of 401(k)
plans on savings by \citet{belloni2017program} and \citet{chernozhukov2018double} and the study of the racial test score gap by \citet{fryerlevitt2013}. We draw samples of different sizes from the large original data and compare the subsample estimates to the estimates based on the original data.\footnote{\normalsize \begin{doublespace}For example, \citet{kolesar2018inference} use a similar of exercise to illustrate the issues with discrete running variables in regression discontinuity designs.\end{doublespace}} In both applications, we find substantial biases even when $n$ is considerably larger than $p$, and we document that the magnitude of the biases varies substantially depending on the regularization choice.

Given our theoretical results, simulations, and empirical evidence, a natural question is how to make statistical inferences in a reliable manner if one is concerned about OVBs. In many economic applications, $p$ is comparable to but still smaller than $n$. This motivates the recent development of high-dimensional OLS-based inference procedures \citep[e.g.,][]{cattaneo2018inference,dadamo2018cluster,jochmans2020heteroscedasticity,kline2020leave}. These methods are based on OLS regressions with all controls and rely on novel variance estimators that are robust to the inclusion of many controls (unlike conventional variance estimators). Based on extensive simulations, we find that OLS with standard errors proposed by \citet{cattaneo2018inference} demonstrates excellent coverage accuracy across all our simulation designs. Another advantage of OLS-based methods over Lasso-based inference methods is that the former do not rely on any sparsity assumptions. This is important because sparsity assumptions may not be satisfied in applications and, as this paper shows, the OVBs of Lasso-based inference procedures can be substantial even when $k$ is small and $n$ is large and larger than $p$. 
However, OLS yields somewhat wider confidence intervals than the Lasso-based inference methods, suggesting a trade-off between coverage accuracy and the length of the confidence intervals.

Our analyses suggest two main recommendations concerning the use of post double Lasso in empirical studies. First, if the estimates of $\alpha^{*}$ are robust to increasing the recommended regularization parameters in both Lasso steps, this suggests that either the OVBs are negligible (Regime (ii)) or under-selection is unlikely (Regime (iii)). In either case, post double Lasso is a reliable and efficient method. 
Otherwise, modern high-dimensional OLS-based inference methods constitute a possible alternative when $p$ is smaller than $n$. 
Second, our findings highlight the importance of augmenting the final OLS regression in post double Lasso with control variables motivated by economic theory and prior knowledge, as suggested by \citet{belloni2014inference}.

\section{Lasso and post double Lasso}
\label{sec:lasso_post_double_lassp}

\subsection{The Lasso}

\label{sec:lasso}

Consider the following linear regression model 
\begin{equation}
Y_{i}=X_{i}\theta^{*}+\varepsilon_{i},\qquad i=1,\dots,n,\label{eq:1}
\end{equation}
where $\left\{ Y_{i}\right\} _{i=1}^{n}=Y$ is an $n$-dimensional
response vector, $\left\{ X_{i}\right\} _{i=1}^{n}=X$ is an $n\times p$
matrix of covariates with $X_{i}$ denoting the $i$th row of $X$,
$\left\{ \varepsilon_{i}\right\} _{i=1}^{n}=\varepsilon$ is a zero-mean
error vector, and $\theta^{*}$ is a $p$-dimensional vector of unknown
coefficients.

The Lasso estimator of $\theta^{*}$, which was first proposed by \citet{tibsharini1996regression}, is given by 
\begin{equation}
\hat{\theta}\in\arg\min_{\theta\in\mathbb{R}^{p}}\frac{1}{2n}\sum_{i=1}^{n}\left(Y_{i}-X_{i}\theta\right)^{2}+\lambda\sum_{j=1}^{p}\left|\theta_{j}\right|,\label{eq:las}
\end{equation}
where $\lambda$ is the regularization parameter. Let
$\varepsilon\sim\mathcal{N}\left(0_{n},\,\sigma^{2}I_{n}\right)$
and $X$ be a fixed design matrix with normalized columns (i.e., $n^{-1}\sum_{i=1}^{n}X_{ij}^{2}=1$ for all $j=1,\dots,p$). In this example, \citet{bickel2009simultaneous} set $\lambda=2\sigma\sqrt{2n^{-1}\left(1+\tau\right)\log p}$
(where $\tau>0$) to establish upper bounds on $\sqrt{\sum_{j=1}^{p}(\hat{\theta}_{j}-\theta_{j}^{*})^{2}}$
with a high probability guarantee. To establish
perfect selection, \citet{wainwright2009sharp} sets
$\lambda$ proportional to $\sigma\phi^{-1}\sqrt{\left(\log p\right)/n}$,
where $\phi\in(0,\,1]$ is a measure of correlation between the covariates
with nonzero coefficients and those with zero coefficients. 

Besides the classical choices in \citet{bickel2009simultaneous} and
\citet{wainwright2009sharp}, other choices of $\lambda$
are available in the literature. For instance, \citet{belloni2012sparse} and
\citet{belloni2016cluster} propose choices that accommodate heteroscedastic and clustered errors. The regularization choice of \citet{belloni2012sparse}, which is recommended by \citet{belloni2014inference} for post double Lasso, is based on the following Lasso program: 
\begin{equation}
\hat{\theta}\in\arg\min_{\theta\in\mathbb{R}^{p}}\frac{1}{n}\sum_{i=1}^n\left(Y_i-X_i\theta \right)^2+\frac{\lambda}{n}\sum_{j=1}^p|\hat{l}_j\theta_j |,
\label{eq:general_lasso}
\end{equation}
where $(\hat{l}_1,\dots,\hat{l}_p)$ are penalty loadings obtained using the iterative algorithm developed in \citet{belloni2012sparse}.

Finally, a very popular practical
approach for choosing $\lambda$ is cross-validation; see, for example, \citet{homrighausen2013thelasso,homrighausen2014leaveoneout}
and \citet{chetverikov2020cross} for theoretical results on cross-validated Lasso. 

\subsection{Post double Lasso}
\label{sec:post_double_Lasso}
The model \eqref{eq:main-y}\textendash \eqref{eq:main-d} implies
the following reduced form model for $Y_{i}$: 
\begin{eqnarray}
Y_{i} & = & X_{i}\pi^{*}+u_{i},\label{eq:reduce-1}
\end{eqnarray}
where $\pi^{*}=\gamma^{*}\alpha^{*}+\beta^{*}$ and $u_{i}=\eta_{i}+\alpha^{*}v_{i}$.

The post double Lasso, introduced by \citet{belloni2014inference},
essentially exploits the Frisch-Waugh theorem, where the regressions
of $Y$ on $X$ and $D$ on $X$ are implemented with the Lasso: 
\begin{eqnarray}
\hat{\pi} & \in & \textrm{arg}\min_{\pi\in\mathbb{R}^{p}}\frac{1}{2n}\sum_{i=1}^{n}\left(Y_{i}-X_{i}\pi\right)^{2}+\lambda_{1}\sum_{j=1}^{p}\left|\pi_{j}\right|,\label{eq:las-1}\\
\hat{\gamma} & \in & \textrm{arg}\min_{\gamma\in\mathbb{R}^{p}}\frac{1}{2n}\sum_{i=1}^{n}\left(D_{i}-X_{i}\gamma\right)^{2}+\lambda_{2}\sum_{j=1}^{p}\left|\gamma_{j}\right|.\label{eq:las-2}
\end{eqnarray}
The final estimator $\tilde{\alpha}$ of $\alpha^{*}$ is then obtained
from an OLS regression of $Y$ on $D$ and the union of selected controls
\begin{equation}
\left(\tilde{\alpha},\,\tilde{\beta}\right)\in\textrm{arg}\min_{\alpha\in\mathbb{R},\beta\in\mathbb{R}^{p}}\frac{1}{2n}\sum_{i=1}^{n}\left(Y_{i}-D_{i}\alpha-X_{i}\beta\right)^{2}\quad\textrm{s.t. }\beta_{j}=0\;\forall j\notin\left\{ \hat{I}_{1}\cup\hat{I}_{2}\right\} ,\label{eq:double}
\end{equation}
where $\hat{I}_{1}=\textrm{supp}\left(\hat{\pi}\right)=\left\{ j:\,\hat{\pi}_{j}\neq0\right\} $
and $\hat{I}_{2}=\textrm{supp}\left(\hat{\gamma}\right)=\left\{ j:\,\hat{\gamma}_{j}\neq0\right\} $.

\section{Numerical example}
\label{sec:numerical_evidence}

This section presents a simple numerical example illustrating the OVB of post double Lasso. All computations were performed in \texttt{Matlab} \citep{MATLAB2020}. The Lasso is implemented using the built-in function \texttt{lasso}.  We consider a simple but classical setting that is often considered favorable to the performance of the Lasso. The data are simulated according to the structural model \eqref{eq:main-y}--\eqref{eq:main-d}, where $X_{i}\overset{iid}\sim \mathcal{N}\left(0_p,I_{p}\right)$,
$\eta_{i}\overset{iid}\sim\mathcal{N}(0,1)$, and $v_{i}\overset{iid}\sim \mathcal{N}(0,1)$
are independent of each other. Our object of interest is $\alpha^{*}$. We set $n=500$, $p=200$, $\alpha^{*}=0$,
and consider a sparse setting where $\beta_j^{*}=\gamma_j^{*}=c \cdot 1\left\{j\le k\right\}$ for $j=1,\dots,p$ and $k=5$. Following the simulation exercise in \citet{belloni2014inference}, we vary the population $R^2$s in \eqref{eq:main-d} and \eqref{eq:reduce-1} by varying the magnitude of the non-zero coefficients $c$. We employ the regularization parameter choice by \citet{bickel2009simultaneous}.

Figure \ref{fig:fsd_dml} displays the finite sample distribution of post double Lasso for different values of $R^2$. For comparison, we plot the distribution of the ``oracle estimator'' of $\alpha^\ast$, a regression of $Y_i-X_i\pi^\ast$ on $D_i-X_i\gamma^\ast$. The finite sample behavior of post double Lasso depends on how many of the $k=5$ relevant controls get selected in both Lasso steps. Figure \ref{fig:hist_sel_dml} shows histograms of the number of selected relevant controls. \textbf{[FIGURES \ref{fig:fsd_dml} AND \ref{fig:hist_sel_dml} HERE.]}

When $R^2=0.5$, post double Lasso exhibits an excellent performance. The finite sample distribution is well-approximated by the normal distribution of the oracle estimator and centered at $\alpha^\ast=0$. The reason for the excellent performance is that all $k=5$ controls are selected with high probability such that post double Lasso essentially coincides with an OLS regression of $Y_i$ on $D_i$ and the five relevant controls.

Let us now consider what happens if we decrease the magnitude of the coefficients and the implied $R^2$. For $R^2=0.3$, post double Lasso exhibits a large finite sample bias. Moreover, the distribution of post double Lasso differs substantially from the distribution of the oracle estimator: it has a larger standard deviation and is slightly skewed. This distribution is a mixture of the distributions of OLS conditional on the two Lasso steps selecting different combinations of relevant controls (Panel (b) in Figure \ref{fig:hist_sel_dml}). For $R^2=0.1$, post double Lasso again exhibits a significant bias. At the same time, the shape of the distribution is similar to that of the oracle estimator. This is because, with high probability, none of the controls gets selected, while the coefficients are large enough to cause an OVB. Decreasing the $R^2$ to $0.01$ reduces the bias. However, it does not change the shape of the finite sample distribution because the selection performance remains unchanged.

The simple numerical example in this section shows that post double Lasso can suffer from OVBs when the two Lasso steps do not select all relevant controls. The magnitude of the coefficients corresponding to the omitted controls can be large enough such that the OVB shifts the location of the finite sample distribution far away from the true value $\alpha^\ast=0$. The issue documented here is not a ``small sample'' phenomenon but persists even in large sample settings; see Appendix \ref{app: large sample}.

\section{Theoretical analysis}
\label{sec:inference}

This section provides a theoretical analysis of the OVB of post double Lasso. Our goal here is to demonstrate that,
even in simple examples with classical assumptions (e.g.,
normal homoscedastic errors, orthogonal designs for the relevant controls),
which are often viewed favorable to the performance of Lasso,
the finite sample OVBs of post double Lasso can be substantial relative
to the standard deviation provided in the existing literature. We
first establish a new necessary result for the Lasso's inclusion and
then derive lower and upper bounds on the OVBs of post double Lasso.
These results are derived for fixed $\left(n,\,p,\,k\right)$ and
are also valid when $\left(k\log p\right)/n\rightarrow0$ or $\left(k\log p\right)/n\rightarrow\infty$.
As it will become clear in the following, $p$ needs to be large enough
for our results to be informative. Without loss of generality, we
normalize the matrix $X$ such that $\left(X_{j}^{T}X_{j}\right)/n=1$
for all $j=1,\dots,p$. We focus on fixed designs (of $X$) to highlight
the essence of the problem; see Appendix \ref{sec:Random-design}
for an extension to random designs. 

For the convenience of the reader, here we collect the notation to
be used in the theoretical analysis. Let $1_{m}$ denote the $m-$dimensional
(column) vector of ``1''s and $0_{m}$ is defined similarly. The
$\ell_{1}-$norm of a vector $v\in\mathbb{R}^{m}$ is denoted by $\left|v\right|_{1}:=\sum_{i=1}^{m}\left|v_{i}\right|$
and the $\ell_{\infty}-$norm of a vector $v\in\mathbb{R}^{m}$ is
denoted by $\left|v\right|_{\infty}:=\max_{i=1,\dots,m}\left|v_{i}\right|$.
The $\ell_{\infty}$ matrix norm (maximum absolute row sum) of a matrix
$A$ is denoted by $\left\Vert A\right\Vert _{\infty}:=\max_{i}\sum_{j}\left|a_{ij}\right|$.
For a vector $v\in\mathbb{R}^{m}$ and a set of indices $T\subseteq\left\{ 1,\dots,m\right\} $,
let $v_{T}$ denote the sub-vector (with indices in $T$) of $v$.
For a matrix $A\in\mathbb{R}^{n\times m}$, let $A_{T}$ denote the
submatrix consisting of the columns with indices in $T$. For a vector
$v\in\mathbb{R}^{m}$, let $\textrm{sgn}(v):=\left\{ \textrm{sgn}(v_{j})\right\} _{j=1,\dots,m}$
denote the sign vector such that $\textrm{sgn}(v_{j})=1$ if $v_{j}>0$,
$\textrm{sgn}(v_{j})=-1$ if $v_{j}<0$, and $\textrm{sgn}(v_{j})=0$
if $v_{j}=0$. Given a set $K$, let $\textrm{card}\left(K\right)$
denote the cardinality of $K$. We denote $\max\left\{ a,\,b\right\} $
by $a\vee b$ and $\min\left\{ a,\,b\right\} $ by $a\wedge b$.

\subsection{Stronger necessary results on the Lasso's inclusion}
\label{sec:stronger_necessary}
Post double Lasso
exhibits OVBs whenever the relevant controls are selected in neither
\eqref{eq:las-1} nor \eqref{eq:las-2}. To the best of our knowledge, there are no formal results strong enough to show that, \textit{with high probability},
Lasso can fail to select the relevant controls in both steps. Therefore, we first establish a new necessary
result for the single Lasso's inclusion in Lemma \ref{prop:fixed_design}.
%To the best of our knowledge,
%there are no formal results strong enough to show that %Lasso can fail to select the relevant controls in both
%steps with high
%probability. Therefore, 
To derive this result, we consider the following classical assumptions, which are often viewed favorable to the performance of the Lasso.
\begin{assumption}\label{ass: incoherence} In terms of model \eqref{eq:1},
suppose: (i) $K=\left\{ j:\,\theta_{j}^{*}\neq0\right\} \neq\emptyset$
and $\textrm{card}\left(K\right)= k\leq\left(n\wedge p\right)$;
(ii) $X_{K}^{T}X_{K}$ is a diagonal matrix; (iii) $\left\Vert \left(X_{K^{c}}^{T}X_{K}\right)\left(X_{K}^{T}X_{K}\right)^{-1}\right\Vert _{\infty}=1-\phi$
for some $\phi\in(0,\,1]$, where $K^{c}$ is the complement of $K$.\end{assumption}

Known as the
incoherence condition due to \citet{wainwright2009sharp}, part (iii)
in Assumption \ref{ass: incoherence} is needed for the exclusion
of the irrelevant controls. Note that if the columns in $X_{K^{c}}$
are orthogonal to the columns in $X_{K}$ (but within $X_{K^{c}}$,
the columns need not be orthogonal to each other), then $\phi=1$.
Obviously a special case of this is when the entire $X$ consists
of mutually orthogonal columns (which is possible if $n\geq p$).
To provide some intuition for Assumption \ref{ass: incoherence}(iii),
let us consider the simple case where $k=1$ and $K=\left\{ 1\right\} $,
$X$ is centered (such that $\left\{n^{-1} \sum_{i=1}^{n}X_{ij}\right\} _{j=1}^{p}=0_{p}$),
and the columns in $X$ are normalized such that the standard deviations
of $X_{1}$ and $X_{j}$ (for any $j\in\left\{ 2,3,\dots,p\right\} $)
are identical. Then, $1-\phi$ is simply the maximum of the absolute
(sample) correlations between $X_{1}$ and each of the $X_{j}$s with
$j\in\left\{ 2,3,\dots,p\right\} $. 

\begin{lemma}[Necessary result on the Lasso's inclusion]\label{prop:fixed_design}
In model \eqref{eq:1}, suppose the $\varepsilon_{i}$s are independent
over $i=1,\dots,n$ and $\varepsilon_{i}\sim\mathcal{N}\left(0,\sigma^{2}\right)$,
where $\sigma\in\left(0,\,\infty\right)$. Let Assumption \ref{ass: incoherence}
hold. We solve the Lasso (\ref{eq:las}) with \begin{equation}
    \lambda\geq\frac{2\sigma}{\phi}\sqrt{\frac{2\left(1+\tau\right)\log p}{n}} \label{eq:Lam}
\end{equation}
where $\tau>0$. Let $E_{1}$ denote the event that $\textrm{sgn}\left(\hat{\theta}_{j}\right)=-\textrm{sgn}\left(\theta_{j}^{*}\right)$
for at least one $j\in K$, and $E_{2}$ denote the event that $\textrm{sgn}\left(\hat{\theta}_{l}\right)=\textrm{sgn}\left(\theta_{l}^{*}\right)$
for at least one $l\in K$ with 
\begin{equation}
\left|\theta_{l}^{*}\right|\leq\frac{\lambda}{2}.\label{eq:min}
\end{equation}
Then, we have 
\begin{equation}
\mathbb{P}\left(E_{1}\cap\mathcal{\mathcal{E}}\right)=\mathbb{P}\left(E_{2}\cap\mathcal{\mathcal{E}}\right)=0\label{eq:zero}
\end{equation}
where $\mathcal{E}=\left\{ \left|n^{-1}X^{T}\varepsilon\right|_{\infty}\leq\sigma\phi^{-1}\sqrt{2n^{-1}\left(1+\tau\right)\log p}\right\} $
and $\mathbb{P}\left(\mathcal{E}\right)\geq1-1/p^{\tau}$.

If (\ref{eq:min}) holds for all $l\in K$, we have 
\begin{equation}
\mathbb{P}\left(\hat{\theta}=0_{p}\right)\geq1-\frac{1}{p^{\tau}}.\label{eq:nec}
\end{equation}

\end{lemma}

Lemma \ref{prop:fixed_design} shows that for large enough $p$, Lasso
fails to select any of the relevant covariates with high probability
if (\ref{eq:min}) holds for all $l\in K$. If such conditions hold
with respect to both \eqref{eq:main-y} and \eqref{eq:main-d}, then
Lemma \ref{prop:fixed_design} implies that the relevant controls
are selected in neither \eqref{eq:las-1} nor \eqref{eq:las-2} with
probability at least $1-2/p^{\tau}$; see Panels
(c) and (d) of Figure \ref{fig:hist_sel_dml} for an illustration. 

Let us rewrite (\ref{eq:min}) as $\left|\theta_{l}^{*}\right|=2^{-1}\left|a\right|\lambda$
with $\left|a\right|\in(0,\,1]$. Assume that $\left|a\right|\phi^{-1}\sigma$
is bounded from above and away from zero; moreover, $\lambda$ satisfies \eqref{eq:Lam} and scales as $\sqrt{\left(\log p\right)/n}$. These conditions imply that
$\left|\theta_{l}^{*}\right|\asymp\sqrt{1/n}$ in the classical
asymptotic framework where $n\rightarrow \infty$ and $p$ is fixed. This regime of $\theta_{l}^{*}$
is exactly where classical model selection procedures struggle to distinguish
a coefficient from zero in low-dimensional settings. 

In the introduction, we have discussed the relationship of our result to that in \citet{lahiri2021necessary}. It is also interesting to compare Lemma \ref{prop:fixed_design} with the results in \citet{wainwright2009sharp}. Note that (\ref{eq:zero}) implies $\mathbb{P}\left(\hat{\theta}_{l}\neq0\right)\leq1/p^{\tau}$
for any $l\in K$ subject to (\ref{eq:min}). In comparison, \citet{wainwright2009sharp}
shows that whenever $\theta_{l}^{*}\in\left(\lambda\textrm{sgn}\left(\theta_{l}^{*}\right),\,0\right)$
or $\theta_{l}^{*}\in\left(0,\,\lambda\textrm{sgn}\left(\theta_{l}^{*}\right)\right)$
for some $l\in K$, 
\begin{equation}
\mathbb{P}\left[\textrm{sgn}\left(\hat{\theta}_{K}\right)=\textrm{sgn}\left(\theta_{K}^{*}\right)\right]\leq\frac{1}{2}.\label{eq:wainwright}
\end{equation}
Constant bounds in the form of (\ref{eq:wainwright}) cannot explain
that, with high probability, Lasso fails to select the relevant covariates
in both \eqref{eq:las-1} and \eqref{eq:las-2} when $p$ is sufficiently
large.

\begin{remark}\label{more choices}As the choices of regularization
parameters used in the vast majority of literature \citep[e.g.,][]{bickel2009simultaneous,wainwright2009sharp,belloni2012sparse,belloni2013least,belloni2014inference},
the choice of $\lambda$ in Lemma \ref{prop:fixed_design} is derived
from the principle that $\lambda$ should be no smaller than $2\max_{j=1,\dots,p}\left|n^{-1}X_{j}^{T}\varepsilon\right|$
with high probability. In particular, our choice for $\lambda$ takes
the form of that in \citet{wainwright2009sharp}, but ours involves
a sharper universal constant. Choosing regularization parameters in
this form ensures the exclusion of irrelevant controls. In addition,
our choice for $\lambda$ has a scaling that can be achieved by the regularization parameters in \citet{belloni2012sparse}, \citet{belloni2013least}, \citet{belloni2014inference},
and coincides with that in \citet{bickel2009simultaneous} when the
columns in $X_{K^{c}}$ are orthogonal to the columns in $X_{K}$
(but within $X_{K^{c}}$, the columns need not be orthogonal to each
other). \end{remark}

\subsection{Lower bounds on the OVBs\label{sec:bias_post_double_lasso}}
In this section, we apply Lemma \ref{prop:fixed_design} to derive lower bounds on the OVB of post double Lasso.
We consider the structural model \eqref{eq:main-y}--\eqref{eq:main-d},
which can be written in matrix notation as 
\begin{eqnarray}
Y & = & D\alpha^{*}+X\beta^{*}+\eta,\label{eq:20}\\
D & = & X\gamma^{*}+v.\label{eq:21-1}
\end{eqnarray}
In matrix notation, the reduced form \eqref{eq:reduce-1} becomes
\begin{eqnarray}
Y & = & X\pi^{*}+u,\label{eq:reduce}
\end{eqnarray}
where $\pi^{*}=\gamma^{*}\alpha^{*}+\beta^{*}$ and $u=\eta+\alpha^{*}v$.
We make the following assumptions about model (\ref{eq:20})--(\ref{eq:21-1}).

\begin{assumption}\label{ass:inference_normality} (i) The error
terms $\eta$ and $v$ consist of independent entries drawn from $\mathcal{N}\left(0,\,\sigma_{\eta}^{2}\right)$
and $\mathcal{N}\left(0,\,\sigma_{v}^{2}\right)$, respectively, where
$\eta$ and $v$ are independent of each other; (ii) the data are centered: $\bar{D}=n^{-1}\sum_{i=1}^{n}D_{i}=0$,
$\bar{X}=\left\{n^{-1}\sum_{i=1}^{n}X_{ij}\right\} _{j=1}^{p}=0_{p}$,
and $\bar{Y}=n^{-1}\sum_{i=1}^{n}Y_{i}=0$; (iii) $K=\left\{ j:\,\beta_{j}^{*}\neq0\right\} =\left\{ j:\,\gamma_{j}^{*}\neq0\right\} \neq\emptyset$
and $\textrm{card}\left(K\right)= k\leq\left(n\wedge p\right)$.

\end{assumption}

Proposition \ref{prop:bias_post_double_formula} derives a lower bound
formula for the OVB of post double Lasso concerning the case where
$\alpha^{*}=0$.

\begin{proposition}[OVB lower bound]\label{prop:bias_post_double_formula}
Suppose $\alpha^{*}=0$. Let Assumption \ref{ass: incoherence}(ii)-(iii)
and Assumption \ref{ass:inference_normality} hold; $\lambda_{1}=2\phi^{-1}\sigma_{\eta}\sqrt{2n^{-1}\left(1+\tau\right)\log p}$
and $\lambda_{2}=2\phi^{-1}\sigma_{v}\sqrt{2n^{-1}\left(1+\tau\right)\log p}$;
for all $j\in K$ and $\left|a\right|,\left|b\right|\in(0,\,1]$,
\begin{equation}
\text{both}\quad\beta_{j}^{*}=a\phi^{-1}\sigma_{\eta}\sqrt{\frac{2\left(1+\tau\right)\log p}{n}}\quad\text{ and}\quad\gamma_{j}^{*}=b\phi^{-1}\sigma_{v}\sqrt{\frac{2\left(1+\tau\right)\log p}{n}}.\label{eq:16}
\end{equation}
In terms of $\tilde{\alpha}$ obtained from (\ref{eq:double}), we
have 
\[
\left|\mathbb{E}\left(\tilde{\alpha}-\alpha^{*}\vert\mathcal{M}\right)\right|\geq\underset{:=\underline{\text{OVB}}}{\underbrace{\max_{r\in(0,1]}T_{1}\left(r\right)T_{2}\left(r\right)}}
\]
where 
\begin{eqnarray*}
T_{1}\left(r\right) & = & \frac{\left(1+\tau\right)\left|ab\right|\phi^{-2}\sigma_{\eta}\frac{k\log p}{n}}{4\left(1+\tau\right)\phi^{-2}b^{2}\sigma_{v}\frac{k\log p}{n}+\left(1+r\right)\sigma_{v}},\\
T_{2}\left(r\right) & = & 1-k\exp\left(\frac{-b^{2}\left(1+\tau\right)\log p}{4\phi^{2}}\right)-\frac{1}{p^{\tau}}-\exp\left(\frac{-nr^{2}}{8}\right),
\end{eqnarray*}
for any $r\in(0,\,1]$, and $\mathcal{M}$ is an event with $\mathbb{P}\left(\mathcal{M}\right)\geq1-k\exp\left(-\left(4\phi^{2}\right)^{-1}b^{2}\left(1+\tau\right)\log p\right)-2/p^{\tau}$.

\end{proposition}

\begin{remark}In our theoretical results, we implicitly assume $p$
is sufficiently large such that $1-k\exp\left(-\left(4\phi^{2}\right)^{-1}b^{2}\left(1+\tau\right)\log p\right)-2/p^{\tau}>0$.
Indeed, probabilities in such a form are often referred to as the
``high-probability'' guarantees in the literature of (non-asymptotic)
high-dimensional statistics concerning large $p$ and small enough
$k$. Recalling the definitions of $\hat{I}_{1}$ and $\hat{I}_{2}$ 
in Section \ref{sec:post_double_Lasso}, the event $\mathcal{M}$ is the intersection
of $\left\{ \hat{I}_{1}=\hat{I}_{2}=\emptyset\right\} $ and an additional
event $\mathcal{E}_{t^{*}}=\left\{ \left|n^{-1}X_{K}^{T}v\right|_{\infty}\leq t^{*},\,t^{*}=4^{-1}\left|b\right|\lambda_{2}\right\} $. The event
$\left\{ \hat{I}_{1}=\hat{I}_{2}=\emptyset\right\} $ occurs with
probability at least $1-2/p^{\tau}$, and the event $\mathcal{E}_{t^{*}}$
occurs with probability at least $1-k\exp\left(-\left(4\phi^{2}\right)^{-1}b^{2}\left(1+\tau\right)\log p\right)$.
 The additional event $\mathcal{E}_{t^{*}}$ is needed for us to derive
a non-trivial lower bound. In particular, with probability at most
$k\exp\left(-\left(4\phi^{2}\right)^{-1}b^{2}\left(1+\tau\right)\log p\right)$,
we have $\left|n^{-1}X_{K}^{T}v\right|_{\infty}\geq t^{*}$, and
on this event, the lower bound in Proposition \ref{prop:bias_post_double_formula}
can be negative, which is uninformative for the absolute value of
OVBs. \end{remark}

Let us compare the non-asymptotic lower bound in Proposition \ref{prop:bias_post_double_formula} to the implications of the existing asymptotic results for the bias of post double Lasso. If $\sigma_{v}$ is bounded away from zero and $\sigma_{\eta}$ is
bounded from above, the existing theory would imply that the biases
of post double Lasso are bounded from above by $\texttt{constant}\cdot\left(k\log p\right)/n$,
irrespective of whether Lasso fails to select the relevant controls
or not, and how small $\left|a\right|$ and $\left|b\right|$ are.
The (positive) $\texttt{constant}$ does not depend on $\left(n,\,p,\,k,\,\beta_{K}^{*},\,\gamma_{K}^{*},\,\alpha^{*}\right)$,
and bears little meaning in the asymptotic framework which simply
assumes $\left(k\log p\right)/\sqrt{n}\rightarrow0$ among other sufficient
conditions. [The existing theoretical framework makes it difficult
to derive an informative $\texttt{constant}$, and to our knowledge,
the literature provides no such derivation.] The asymptotic upper
bound $\texttt{constant}\cdot\left(k\log p\right)/n$ does not distinguish cases that vary in
$\left(\beta_{K}^{*},\,\gamma_{K}^{*},\,\alpha^{*}\right)$. By contrast,
our lower bound analyses are informative about whether the upper bound
can be attained by the magnitude of the OVBs and provide explicit
constants. These features of our analysis are crucial for understanding
the finite sample limitations of post double Lasso. In view of Proposition
\ref{prop:bias_post_double_formula}, $\underline{OVB}$ is not a
simple linear function of $\left(k\log p\right)/n$ in general, but roughly
linear in $\left(k\log p\right)/n$ when $\phi^{-2}n^{-1}k\log p=o\left(1\right)$
and $k\exp\left(-\left(4\phi^{2}\right)^{-1}b^{2}\left(1+\tau\right)\log p\right)+2/p^{\tau}=o\left(1\right)$.

\subsection{Key takeaways of our theoretical results}
\label{sec:takeaways_theory}
The finite sample behavior of post double Lasso can be characterized by three regimes: (i) \emph{non-negligible} OVBs, (ii) negligible OVBs, and (iii) absence of OVBs. 

\noindent \textbf{Regime (i) (non-negligible OVB)\textbf{.}} When double under-selection occurs with high probability and $\left(k\log p\right)/\sqrt{n}$
is not small enough, according to Proposition \ref{prop:bias_post_double_formula},
the OVB lower bound can be substantial compared to the standard deviation
obtained from the asymptotic distribution in \citet{belloni2014inference}.
To gauge the magnitude of the OVB and explain why the confidence intervals
proposed in the literature can exhibit under-coverage, it is instructive
to compare $\underline{OVB}$ with $\ensuremath{\sigma_{\tilde{\alpha}}=n^{-1/2}\left(\sigma_{\eta}/\sigma_{v}\right)}$,
the standard deviation (of $\tilde{\alpha}$) obtained from the asymptotic
distribution in \citet{belloni2014inference}.\footnote{\normalsize \begin{doublespace}We thank Ulrich M\"uller for suggesting this comparison.\end{doublespace}}
Let us consider an example with $n=14238$, $p=384$ \citep[as in][]{angrist2019machine}, $a=b=1$, $\sigma_{\eta}=\sigma_{v}=1$, $\tau=0.5$,
and $\phi=0.5$. If $\left|\beta_{j}^{*}\right|=0.07$ and $\left|\gamma_{j}^{*}\right|=0.07$
for all $j\in K$, then $\underline{OVB}/\sigma_{\tilde{\alpha}}\approx0.27$
when $k=1$ and $\underline{OVB}/\sigma_{\tilde{\alpha}}\approx2.4$
when $k=10$; see Panel (c) of Figures \ref{fig:fsd_dml} and \ref{fig:hist_sel_dml} for an illustration of Regime (i). The OVBs can also be non-negligible when some but not all relevant controls are selected; see Panel (b) of Figures \ref{fig:fsd_dml} and \ref{fig:hist_sel_dml}.
These results suggest that, non-asymptotically, post
double Lasso cannot avoid the ``post-selection inference issues''
raised in a series of papers by Leeb and P\"otscher \citep[e.g.,][]{leeb2005model,leeb2008can,leeb2017testing}.

\noindent \textbf{Regime (ii) (negligible OVB).} 
By
Lemma \ref{prop:fixed_design} and (\ref{eq:16}), $\left|\hat{\beta}-\beta^{*}\right|_{1}=2^{-1}\left|a\right|k\lambda_{1}$
and $\left|\hat{\gamma}-\gamma^{*}\right|_{1}=2^{-1}\left|b\right|k\lambda_{2}$
with probability at least $1-2/p^{\tau}$. If $\sigma_{v}$
is bounded away from zero, $\sigma_{\eta}$ is bounded from above,
and $\left|a\right|=\left|b\right|=o\left(1\right)$, by a similar argument
as in \citet{belloni2014inference}, we can show that $\sqrt{n}\left(\tilde{\alpha}-\alpha^{*}\right)$
is approximately normal and centered at zero, even if $\left(k\log p\right)/\sqrt{n}$
is bounded away from zero and scales as a constant. Holding other factors constant, the magnitude
of OVBs decreases as $\left|\beta_{K}^{*}\right|$ and $\left|\gamma_{K}^{*}\right|$
decrease (i.e., as $\left|a\right|$ and $\left|b\right|$ decrease).
As $\left|a\right|$ and $\left|b\right|$ become very small, the
relevant controls become essentially irrelevant. Panel (d) of Figures \ref{fig:fsd_dml} and \ref{fig:hist_sel_dml} provides an illustration of Regime (ii).

\noindent \textbf{Regime (iii) (absence of OVB).} All relevant controls will be selected when the magnitude of their coefficients is large enough. Specifically, for all $j\in K,\;a>3,\,b>3$, if 
\begin{equation}
\text{either}\quad\left|\beta_{j}^{*}\right|=a\phi^{-1}\sigma_{\eta}\sqrt{\frac{2\left(1+\tau\right)\log p}{n}}\quad\text{or}\quad\left|\gamma_{j}^{*}\right|=b\phi^{-1}\sigma_{v}\sqrt{\frac{2\left(1+\tau\right)\log p}{n}},\label{eq:15}
\end{equation}
then $\mathbb{P}\left[\textrm{supp}\left(\hat{\pi}\right)=\textrm{supp}\left(\pi^{*}\right)\right]\geq1-1/p^{\tau}$
or $\mathbb{P}\left[\textrm{supp}\left(\hat{\gamma}\right)=\textrm{supp}\left(\gamma^{*}\right)\right]\geq1-1/p^{\tau}$
by standard arguments \citep[e.g.,][]{wainwright_2019}. As a result,
$\mathbb{P}\left(\left\{ \hat{I}_{1}\cup\hat{I}_{2}\right\} =K\right)\geq1-1/p^{\tau}$
(where $\hat{I}_{1}$ and $\hat{I}_{2}$ are defined in Section \ref{sec:post_double_Lasso});
i.e., the final OLS step (\ref{eq:double}) includes all the relevant
controls with high probability. By similar argument as in Appendix
\ref{sec:appendix_post_double_lasso}, on the high probability event
$\left\{ \left\{ \hat{I}_{1}\cup\hat{I}_{2}\right\} =K\right\} $,
the OVB of $\tilde{\alpha}$ is zero. Panel (a) of Figures \ref{fig:fsd_dml} and \ref{fig:hist_sel_dml} provides an illustration of Regime (iii).

\medskip

This ``triple-regime'' characterization suggests that one may assess the robustness of post double Lasso by increasing the penalty level. If increasing $\lambda_{1}$ and $\lambda_{2}$  yields similar estimates
$\tilde{\alpha}$, then the underlying model could be in the regime
where either the OVBs are negligible (Regime (ii)) or under-selection in both Lasso steps is unlikely (Regime (iii)).
By contrast, under Regime (i), the performance of post double
Lasso can be quite sensitive to an increase of $\lambda_{1}$ and
$\lambda_{2}$. The rationale behind this heuristic lies in that the
final step of post double Lasso, (\ref{eq:double}), is simply an
OLS regression of $Y$ on $D$ and the union of selected controls
from (\ref{eq:las-1})--(\ref{eq:las-2}). A natural question is by how much
$\lambda_{1}$ and $\lambda_{2}$ should be increased for the robustness
checks. For the regularization choice proposed in \citet{belloni2014inference}, we will show in the simulations of Section \ref{sec:simulation_evidence} that an increase by $50\%$ works well in practice.  

Finally, our theoretical results have interesting implications for the comparison between post double Lasso and post single Lasso, where the latter only relies on one Lasso step to select the relevant controls. Note that the magnitude of the OVB of the post single Lasso estimator of
$\alpha^{*}$ also falls into three regimes: (i) non-negligible
OVB, (ii) negligible OVB, and (iii) absence of OVB. Thus, qualitatively, the OVBs of post
single Lasso and post double Lasso have a similar behavior in finite samples. However, quantitatively,
the magnitude of the OVBs can be much larger for the
post single Lasso than for the post double Lasso, as illustrated by
the following example. If $\beta_{j}^{*}=a\phi^{-1}\sigma_{\eta}\sqrt{2n^{-1}\left(1+\tau\right)\log p}$
with $\left|a\right|\in(0,\,1]$, and $\left|\gamma_{j}^{*}\right|\geq1$,
then the argument for showing Proposition \ref{prop:bias_post_double_formula}
in Appendix \ref{sec:appendix_post_double_lasso} implies that the
OVB lower bound scales roughly as $\left|a\right|k\sqrt{\left(\log p\right)/n}$
for the post single Lasso estimator of $\alpha^{*}$ (and note that
$\left(\left|a\right|k\sqrt{\left(\log p\right)/n}\right)/\left(1/\sqrt{n}\right)=\left|a\right|k\sqrt{\log p}$). 

\subsection{Additional theoretical results}
In this section, we briefly summarize the additional theoretical results that are provided in the appendix.
First, we also consider cases where $\alpha^{*}\neq0$. The conditions required
to derive the explicit formula are difficult to interpret when $\alpha^{*}\ne0$.
However, it is possible to provide easy-to-interpret scaling results
(without explicit constants) for cases where $\alpha^{*}\ne0$. Roughly, the scaling of our OVB lower bound can be as large as 
\begin{equation*}
\left(\frac{\sigma_{\eta}}{\sigma_{v}}\vee\left|\alpha^{*}\right|\right)\frac{k\log p}{n}
\quad \text{when} \quad \frac{k\log p}{n}\rightarrow0 
\end{equation*}
and 
\begin{equation}
\frac{\sigma_{\eta}}{\sigma_{v}}\vee\left|\alpha^{*}\right| \quad \text{when} \quad  \frac{k\log p}{n}\rightarrow\infty.\label{eq:upper_bound2}
\end{equation}
These results reveal an interesting feature of the post
double Lasso. The scaling of the OVB lower bound depends
on $\left|\alpha^{*}\right|$ when the relevant controls are not selected. This is because the error $u$ in the reduced form equation \eqref{eq:reduce} involves $\alpha^{*}$ such that the
choice of $\lambda_{1}$ in \eqref{eq:las-1} depends on $\left|\alpha^{*}\right|$ via the variance of $u$. By contrast, it is well-known that the
OVB of OLS does not depend on $\left|\alpha^{*}\right|$ when relevant
controls are omitted. Interested readers are referred to Propositions \ref{prop:bias_post_double_selection}
and \ref{prop:bias_post_double_selection-1} in Appendix \ref{subsec:Lower-bounds-alpha not zero}
for details.

Second, we also provide upper bounds on the OVB. We have seen in equation \eqref{eq:upper_bound2} that the lower bound on the OVBs scales as
$
\left(\sigma_{\eta}/\sigma_{v}\right)\vee\left|\alpha^{*}\right|$ when 
$\left(k\log p\right)/n\rightarrow\infty$.
Interestingly enough, we can
also show that the upper bound on the OVB scales as in equation \eqref{eq:upper_bound2}
despite $\left(k\log p\right)/n\rightarrow\infty$
and the Lasso being inconsistent in the sense $\sqrt{n^{-1}\sum_{i=1}^{n}\left(X_{i}\hat{\pi}-X_{i}\pi^{*}\right)^{2}}\rightarrow\infty$,
$\sqrt{n^{-1}\sum_{i=1}^{n}\left(X_{i}\hat{\gamma}-X_{i}\gamma^{*}\right)^{2}}\rightarrow\infty$
with high probability. Interested readers are referred to Propositions
\ref{prop:bias_post_double_selection-upper} and \ref{prop:bias_post_double_selection-upper-1}
in Appendix \ref{subsec:Upper-bounds-OVB} for details.

\section{Simulations and empirical evidence}
\label{sec:simulation_evidence}

To better understand the practical implications of the OVB of post double Lasso, in this section, we present the results from simulations and two empirical applications with widely-used regularization choices available in standard software packages. The analyses were carried out using \texttt{Matlab} \citep{MATLAB2020}, \texttt{R} \citep{R2021}, and \texttt{Stata} \citep{stata2021}.

\subsection{Simulations}
\label{sec:mc_simulations}
In this section, we present simulation evidence on the performance of post double Lasso with three choices of the regularization parameter: (i) the heteroscedasticity-robust proposal of \citet{belloni2012sparse,belloni2014inference} ($\lambda_{\text{BCCH}}$) implemented using the \texttt{R}-package \texttt{hdm} with the \texttt{double selection} option \citep{hdm2016}, (ii) the regularization parameter with the minimum cross-validated error ($\lambda_{\text{min}}$) implemented using the \texttt{R}-package \texttt{glmnet} \citep{glmnet},
and (iii) the regularization parameter corresponding to the minimum plus one standard deviation cross-validated error ($\lambda_{\text{1se}}$) also implemented using \texttt{glmnet}. We use the same type of regularization parameter choice in both Lasso steps.

We simulate data according to the DGP of Section \ref{sec:numerical_evidence}. To illustrate the role of the sample size $n$ and the sparsity parameter $k$, we consider (i) $(n,p,k)=(500,200,5)$, (ii) $(n,p,k)=(1000,200,5)$, and (iii) $(n,p,k)=(500,200,10)$. We show results for $R^2\in \{0.01,0.05,0.1,0.2,0.3,0.4,0.5\}$ based on 1,000 simulation repetitions. Appendix \ref{app:additional_simulations} presents additional simulation evidence, where we vary $n$, the distribution of $(X_i,\eta_i,v_i)$, the true value $\alpha^\ast$, and also consider a heteroscedastic DGP. 

We start by investigating the selection performance of the two Lasso steps of post double Lasso. Panel (a) of Figure \ref{fig:sel} displays the average number of selected controls (i.e., the cardinality of $\hat{I}_{1}\cup\hat{I}_{2}$ in \eqref{eq:double}) as a function of $R^2$. Lasso with $\lambda_{\text{BCCH}}$ selects the lowest number of controls. Choosing $\lambda_{\text{1se}}$ leads to a somewhat higher number of selected controls and results in moderate over-selection for larger values of $R^2$. Lasso with $\lambda_{\text{min}}$  selects the highest number of controls and exhibits substantial over-selection. Panel (b) shows the corresponding average numbers of selected relevant controls. \textbf{[FIGURE \ref{fig:sel} HERE.]}

Figure \ref{fig:dml_bias_std} presents evidence on the bias of post double Lasso. To make the results easier to interpret, we report the ratio of the bias to the empirical standard deviation. Post double Lasso with $\lambda_{\text{BCCH}}$ can exhibit biases that are more than two times larger than the standard deviation when $(n,p,k)=(500,200,10)$. The bias can still be comparable to the standard deviation when $(n,p,k)=(1000,200,5)$. Consistent with our theoretical discussions in Sections \ref{sec:bias_post_double_lasso} and \ref{sec:takeaways_theory}, the relationship between $R^2$ and the ratio of bias to standard deviation is non-monotonic: it is increasing for small $R^2$ and decreasing for larger $R^2$. The bias is somewhat smaller for $\lambda=\lambda_{\text{1se}}$. Setting $\lambda=\lambda_{\text{min}}$ yields the smallest ratio of bias to standard deviation. Finally, we note that when $R^2$ is large enough such that there is no under-selection, post double Lasso performs well and is approximately unbiased for all regularization parameters. \textbf{[FIGURE \ref{fig:dml_bias_std} HERE.]}

The additional simulation evidence reported in Appendix \ref{app:additional_simulations} confirms these results but further shows that $\alpha^\ast$ is an important determinant of the performance of post double Lasso because of its direct effect on the magnitude of the coefficients and the error variance in the reduced form equation \eqref{eq:reduce-1}. Moreover, we show that, while choosing $\lambda=\lambda_{\text{min}}$ works well when $\alpha^\ast=0$, this choice can yield poor performances when $\alpha^\ast\ne 0$ (see Figure \ref{fig:dml_bias_std_app}). A similar phenomenon arises when using $0.5\lambda_{\text{BCCH}}$ instead of $\lambda_{\text{BCCH}}$: this choice works well when $\alpha^\ast=0$ (see Figure \ref{fig:dml_bias_std_sens} below), but yields biases when $\alpha^\ast\ne 0$. We found that, under our DGPs, this is related to the fact that when $\alpha^\ast\ne 0$, \eqref{eq:main-d} and \eqref{eq:reduce-1} differ with respect to the underlying coefficients and noise variances, which leads to differences in the (over-)selection behavior of the Lasso. Thus, there is no simple recommendation for how to choose the regularization parameters in practice.

The substantive performance differences between the three regularization choices suggest that post double Lasso is sensitive to the penalty levels in the intermediate case where $R^2$ is small enough so that under-selection occurs but large enough to cause substantial OVBs. To further investigate this issue, we compare the results for $\lambda_{\text{BCCH}}$, $0.5 \lambda_{\text{BCCH}}$, and $1.5 \lambda_{\text{BCCH}}$. Figure \ref{fig:sel_sensitivity} displays the average numbers of all selected controls (relevant or not) and selected relevant controls in both Lasso steps. The differences in the selection performance are substantial. Lasso with $0.5 \lambda_{\text{BCCH}}$ over-selects for all $n$ and $R^2$, while Lasso with $1.5 \lambda_{\text{BCCH}}$ under-selects unless $R^2$ and $n$ are large and $k=5$. The differences get smaller as $n$ increases and larger as $k$ increases. \textbf{[FIGURE \ref{fig:sel_sensitivity} HERE.]}

Figure \ref{fig:dml_bias_std_sens} displays the ratio of bias to standard deviation. Choosing $0.5 \lambda_{\text{BCCH}}$ yields small biases relative to the standard deviations for all $R^2$. By contrast, choosing $1.5 \lambda_{\text{BCCH}}$ yields biases that can be more than six times larger than the standard deviations when $(n,p,k)=(500,200,10)$ and still substantial when $(n,p,k)=(1000,200,5)$. For very small and large values of $R^2$, post double Lasso is less sensitive to the penalty level. In Section \ref{sec:recommendation}, we discuss how to interpret and use robustness checks with respect to the regularization parameters in empirical applications. \textbf{[FIGURE \ref{fig:dml_bias_std_sens} HERE.]}

In sum, our simulation evidence shows that (i) under-selection can lead to large biases relative to the standard deviations, (ii) the performance of post double Lasso can be very sensitive to the choice of regularization parameters, and (iii) there is no simple recommendation for how to choose the regularization parameters in practice.

\subsection{Empirical evidence}
\label{sec:empirical_evidence}
\subsubsection{The effect of 401k plans on total wealth}
\label{sec:401k}
We revisit the analysis of the causal effect of eligibility for 401(k) plans ($D$) on total wealth ($Y$).\footnote{\normalsize \begin{doublespace} The effect of 401(k) plans is well-studied. We estimate intention to treat effect of 401(k) eligibility on assets as, e.g., in \citet{Poterbaetal1994,Poterbaetal1995,Poterbaetal1998} and \citet{Benjamin2003}. Other studies have used 401(k) eligibility to instrument for the actual 401(k) participation status \citep[e.g.,][]{abadie2003semiparametric,CH2004,belloni2017program,wuthrich2019closed}.\end{doublespace}} We use the data on $n=9915$ households from the 1991 SIPP \citep{belloni2017data} analyzed by \citet{belloni2017program} and \citet{chernozhukov2018double} with high-dimensional methods. We consider two different specifications of the control variables ($X$). 

\begin{enumerate}[itemsep=0pt]
\item \textbf{Two-way interactions (TWI) specification.} We use the same set of low-dimensional control variables as in \citet{Benjamin2003} and \citet{CH2004}: seven income dummies, five age dummies, family size, four education dummies, and dummies for marital status, two-earner status, defined benefit pension status, individual retirement account (IRA) participation status, and homeownership. Following common empirical practice, we augment this baseline specification with all two-way interactions. After removing collinear columns there are $p=167$ control variables. 
\item \textbf{Quadratic spline \& interactions (QSI) specification.} This is the ``Quadratic Spline Plus Interactions specification'' of \citet[][p.265]{belloni2017program}. It contains dummies for marital status, two-earner status, defined benefit pension status, IRA participation status, and homeownership, second-order polynomials in family size and education, a third-order polynomial in age, a quadratic spline in income with six breakpoints, as well as interactions of all the non-income variables with each term in the income spline. After removing collinear columns there are $p=272$ control variables. 
\end{enumerate}

Table \ref{tab:results_401k} presents post double Lasso estimates based on the whole sample with $\lambda_{\text{BCCH}}$, $0.5\lambda_{\text{BCCH}}$, and $1.5\lambda_{\text{BCCH}}$. For comparison, we also report OLS estimates with and without controls. For both specifications, the results are qualitatively similar across the different regularization choices and similar to OLS with all controls. This is possible as $n$ is much larger than $p$. Nevertheless, there are some non-negligible quantitative differences between the point estimates. A comparison to OLS without control variables shows that omitting controls can yield substantial OVBs in this application.\textbf{[TABLE \ref{tab:results_401k} HERE.]}

To investigate the impact of under-selection, we perform the following exercise. We draw random subsamples of size $n_s\in \{200,400,800,1600\}$ with replacement from the original dataset. Based on each subsample, we estimate $\alpha^\ast$ using post double Lasso with $\lambda_{\text{BCCH}}$, $0.5\lambda_{\text{BCCH}}$, and $1.5\lambda_{\text{BCCH}}$ and compute the bias as the difference between the average subsample estimate and the point estimate based on the original data with the same type of regularization choice in Table \ref{tab:results_401k}. The results are based on 1,000 simulation repetitions.

Figures \ref{fig:app_401k} displays the bias and the ratio of bias to standard deviation for both specifications. We find that post double Lasso can exhibit large finite sample biases. The biases under the QSI specification tend to be smaller (in absolute value) than the biases under the TWI specification. Interestingly, the ratio of bias to standard deviation may not be monotonically decreasing in $n_s$ (in absolute value) due to the standard deviation decaying faster than the bias. Finally, we find that post double Lasso can be very sensitive to the penalty level. \textbf{[FIGURE \ref{fig:app_401k} HERE.]}

\subsubsection{Racial differences in the mental ability of children}

We revisit \citet{fryerlevitt2013}'s analysis of the racial differences in the mental ability of young children based on data from the US Collaborative Perinatal Project \citep{fryer2013data}. As in the reanalysis of \citet{chernozhukov2020generic}, we restrict the sample to Black and White children so that our final sample includes $n=30002$ observations. We focus on the standardized test score in the Wechsler Intelligence Test at the age of seven as our outcome variable ($Y$). 
The variable of interest ($D$) is an indicator for Black children. We use the same specification as in \citet{fryerlevitt2013}, excluding interviewer fixed effects. The control variables ($X$) include extensive information on socio-demographic characteristics, the home environment, and the prenatal environment; see their Table 1B for descriptive statistics. After removing collinear terms there are $p=78$ controls. 
 
Table \ref{tab:results_testscores} shows the results for post double Lasso with $\lambda_{\text{BCCH}}$, $0.5\lambda_{\text{BCCH}}$, and $1.5\lambda_{\text{BCCH}}$, as well as OLS with and without controls based on the whole sample. Since $n=30002$ is much larger than $p=78$, all methods except for OLS without controls yield similar results. \textbf{[TABLE \ref{tab:results_testscores} HERE.]}

To investigate the impact of under-selection, we draw random subsamples of size $n_s\in \{200,400,800,1600\}$ with replacement from the original dataset. In each sample, we estimate $\alpha^\ast$ using post double Lasso with $\lambda_{\text{BCCH}}$, $0.5\lambda_{\text{BCCH}}$, and $1.5\lambda_{\text{BCCH}}$ and compute the bias as the difference between the average estimate based on the subsamples and the estimate based on the original data with the same type of regularization choice. The results are based on 1,000 simulation repetitions.

Figure \ref{fig:app_test} displays the bias and the ratio of bias to standard deviation. While the magnitude of the bias is decreasing in $n_s$, it can be substantial and larger than the standard deviation when $n_s$ is small. Moreover, the performance of post double Lasso is very sensitive to the choice of the regularization parameters. With $0.5\lambda_{\text{BCCH}}$, post double Lasso is approximately unbiased for all $n_s$, whereas, with $1.5\lambda_{\text{BCCH}}$, the bias is comparable to the standard deviation even when $n_s=1600$. \textbf{[FIGURE \ref{fig:app_test} HERE.]}

\section{Implications for inference and comparison to high-dimensional OLS-based methods}

The OVBs have important consequences for making inferences based on post double Lasso. Figure \ref{fig:dml_cov} displays the coverage rates of 90\% confidence intervals based on the DGPs in Section \ref{sec:mc_simulations} and shows that the OVB of post double Lasso can cause substantial under-coverage even when $n=1000$ and $k=5$. The most important determinant of the under-coverage is the sparsity parameter $k$. Our results show that the requirement on $k$ for guaranteeing a good finite sample coverage accuracy for all $R^2$ can be quite stringent. \textbf{[FIGURE \ref{fig:dml_cov} HERE.]}

These results prompt the question of how to make inference in a reliable manner when one is concerned about OVBs. In many economic applications, $p$ is comparable to but still smaller than $n$. In such settings, OLS-based inference procedures provide a natural alternative to Lasso-based methods. 
Under classical conditions, OLS is the best linear unbiased estimator and admits exact finite sample inference as long as $p+1\leq n$ (recalling that the number of regression coefficients is $p+1$ in \eqref{eq:main-y}). Unlike the Lasso-based inference methods, OLS does not rely on any sparsity assumptions. This is important because sparsity assumptions may not be satisfied in applications and, as we show in this paper, the OVBs of Lasso-based inference procedures can be substantial even when $k$ is small and $n$ is large and larger than $p$. Indeed, OLS-based inference exhibits desirable optimality properties absent sparsity (or other restrictions) on $\beta^\ast$.\footnote{\normalsize\begin{doublespace} For one-sided testing problems, the one-sided $t$-test based on OLS with all controls is the uniformly most powerful test; for two-sided problems, the two-sided $t$-test is the uniformly most powerful unbiased test \citep{vandervaart1998book}. We refer to Section 4 in \citet{armstrong2016}, Section 5.5 in \citet{elliott2015nearly}, and Section 2.1 in \citet{li2021linear} for further discussions.\end{doublespace}}

While OLS is unbiased, constructing standard errors is challenging when $p$ is large. For instance, \citet{cattaneo2018inference} show that conventional Eicker-White robust standard errors are inconsistent under asymptotics where $p$ grows as fast as $n$. 
This result motivates a recent literature to develop high-dimensional OLS-based inference procedures that are valid in settings with many controls \citep[e.g.,][]{cattaneo2018inference,jochmans2020heteroscedasticity,kline2020leave}. 

Figures \ref{fig:comparison_ols} compares the finite sample performance of post double Lasso and OLS with the heteroscedasticity robust HCK standard errors proposed by \citet{cattaneo2018inference}. Panel (a) shows that OLS exhibits close-to-exact empirical coverage rates irrespective of the magnitude of the coefficients and the implied $R^2$. The additional simulation evidence in Appendix \ref{app:additional_simulations} confirms the excellent performance of OLS with HCK standard errors. Panel (b) displays the average length of 90\% confidence intervals and shows that OLS yields somewhat wider confidence intervals than post double Lasso. \textbf{[FIGURE \ref{fig:comparison_ols} HERE.]}

In sum, our simulation results suggest that modern OLS-based inference methods that accommodate many controls may constitute a viable alternative to Lasso-based inference methods. These methods are unbiased and demonstrate an excellent size accuracy, irrespective of the magnitude of the coefficients corresponding to the relevant controls. However, there is a trade-off because OLS yields somewhat wider confidence intervals than post double Lasso. 

Finally, it is worth noting that we consider settings where one can easily invert $X^T X$ and the OLS and HCK variance estimators are numerically stable. In the case of singular or nearly singular $X^T X$, regularization is often unavoidable; see Section \ref{sec:recommendation} for a discussion of alternatives to OLS and Lasso-based inference methods.

\section{Recommendations for empirical practice}
\label{sec:recommendation}

Here we summarize the practical implications of our results and provide guidance for empirical researchers.

First, the simulation evidence in Section \ref{sec:simulation_evidence} and Appendix \ref{app:additional_simulations} along with the theoretical results (see Section \ref{sec:takeaways_theory}) suggest the following heuristic: if the estimates of $\alpha^{*}$ are robust to increasing the theoretically recommended regularization parameters in the two Lasso steps, post double Lasso could be a reliable and efficient method. Therefore, we recommend to always check whether empirical results are robust to increasing the regularization parameters. Based on our simulations, a simple rule of thumb is to increase by $50\%$ the regularization parameters proposed in \citet{belloni2014inference}. Robustness checks are standard in other contexts (e.g., bandwidth choices in regression discontinuity designs), and our results highlight the importance of such checks in the context of Lasso-based inference methods.

Second, following \citet{belloni2014inference}, we recommend to always augment the union of selected controls with an ``amelioration'' set of controls motivated by economic theory and prior knowledge to mitigate the OVBs. 

Third, our simulations show that in moderately high-dimensional settings where $p$ is comparable to but smaller than $n$, recently developed OLS-based inference methods that are robust to the inclusion of many controls exhibit better size properties. These simulation results suggest that high-dimensional OLS-based procedures constitute a possible alternative to Lasso-based inference methods.

Forth, OLS-based methods are not applicable when $p>n$, and the OLS and variance estimators can be numerically unstable under severe multi-collinearity even if $p<n$. In such cases, regularization is often needed. Ridge regressions, which impose restrictions on the Euclidean norm of $\beta^\ast$, avoid variable selection and may be a useful alternative to the Lasso; see also \citet{armstrong2020bias} for related restrictions on $\beta^\ast$. 

Finally, in many economic applications, researchers start with a small number of raw controls and want to use a flexible non-parametric model to capture the dependence of outcomes on controls while maintaining a simple parametric form for modeling the variables of interest. Such a specification leads to the classical partially linear models. In fact, \citet{belloni2014inference} motivate post double Lasso with these models. If one is concerned about OVBs, inference methods that do not rely on variable selection are natural alternatives to post double Lasso. Under suitable smoothness restrictions on the  non-parametric component, inference on the parameter of interest in partially linear models is a well-studied problem \citep[e.g.,][]{robinson1988root,newey1994asymptotic}. The frameworks proposed in these papers can be built upon procedures such as sieves \citep[e.g.,][]{chen2007chapter16}, local non-parametric methods \citep[e.g.,][]{fan1996local}, and kernel ridge regressions \citep[e.g.,][]{scholkopf2002learning}.

\section{Conclusion}

Given the rapidly increasing popularity of Lasso and Lasso-based inference methods in empirical economic research, it is crucial to better understand the merits and limitations of these new tools, and how they compare to other alternatives such as the high-dimensional OLS-based procedures. 

This paper presents theoretical results as well as simulation and empirical evidence on the finite sample behavior of post double Lasso and the debiased Lasso (in the appendix). Specifically, we analyze the finite sample OVBs arising from the Lasso not selecting all the relevant control variables. Our results have important practical implications, and we provide guidance for empirical researchers.

We focus on the implications of under-selection for post double Lasso and the debiased Lasso in linear regression models. However, our results on the under-selection of the Lasso also have important implications for other inference methods that rely on Lasso as a first-step estimator. Towards this end, an interesting avenue for future research would be to investigate the impact of under-selection on the performance of the Lasso-based approaches proposed by \citet{belloni2014inference}, \citet{farrell2015robust}, \citet{belloni2017program}, and \citet{chernozhukov2018double} for non-linear models. In moderately high-dimensional settings where $p$ is smaller than but comparable to $n$, it would also be interesting to compare the treatment effects estimators in \citet{belloni2014inference} to the robust finite sample methods proposed by \citet{rothe2017robust}.

Finally, this paper motivates further examinations of the practical usefulness of Lasso-based inference procedures and other modern high-dimensional methods. For example,  \citet{angrist2019machine} present interesting simulation evidence on the finite sample behavior of Lasso-based IV methods \citep[e.g.,][]{belloni2012sparse}. It would be interesting to explore the implications of our theoretical results on the under-selection of the Lasso in problems with weak instruments.

\bibliographystyle{plainnat}
\bibliography{bibliography.bib}
%\printbibliography

\newpage
\section*{Figures and tables}

\begin{figure}[H]
\caption{Illustration of bias and double under-selection}
\begin{center}
\includegraphics[width=0.6\textwidth, trim = {0 0 0 0}]{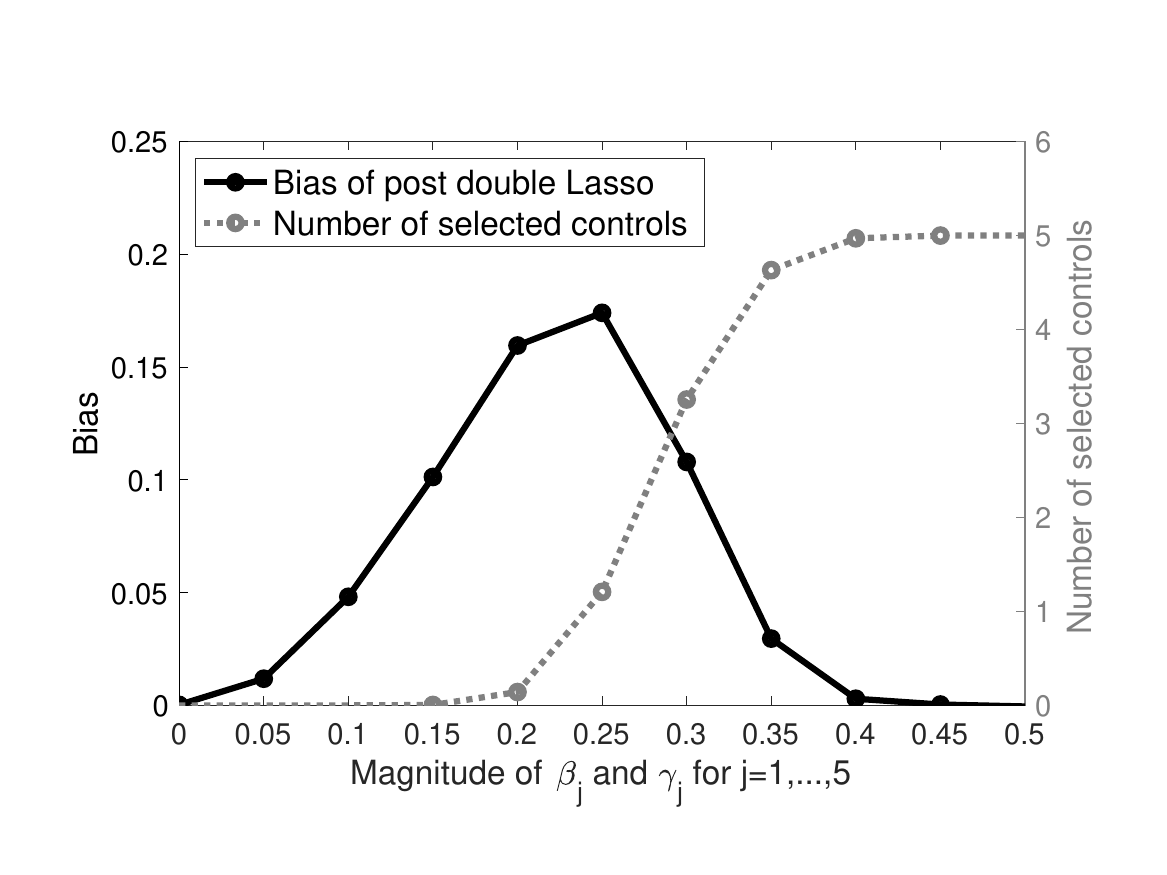}
\end{center}
\begin{doublespace}\textit{Notes:} The figure shows the bias (solid line) of the post double Lasso estimator $\tilde{\alpha}$ and the average number of selected controls (dotted line) as a function of the magnitude of the coefficients associated with the relevant controls. We generate the data based on \eqref{eq:main-y}\textendash \eqref{eq:main-d}, where $X_{i}\overset{iid}\sim\mathcal{N}\left(0,I_{p}\right)$,
$\eta_{i}\overset{iid}\sim \mathcal{N}(0,1)$, and $v_{i}\overset{iid}\sim \mathcal{N}(0,1)$ are independent of each other. We set $(n,p)=(500,200)$, $\alpha^\ast=0$, and $\beta_j^{*}=\gamma_j^{*}$. We vary the magnitude of $\beta_j^{*}$ and $\gamma_j^{*}$ for $j=1,\dots,k$,  where $k=5$, and set $\beta_j^{*}=\gamma_j^{*}=0$ for $j>k$. We use the Lasso regularization choice by \citet{bickel2009simultaneous}.\end{doublespace}
\label{fig:illu_intro}
\end{figure}

\newpage

\begin{figure}[H]
\caption{Finite sample distribution}
\begin{center}
\includegraphics[width=0.9\textwidth, trim = {0 1cm 0 0.5cm}]{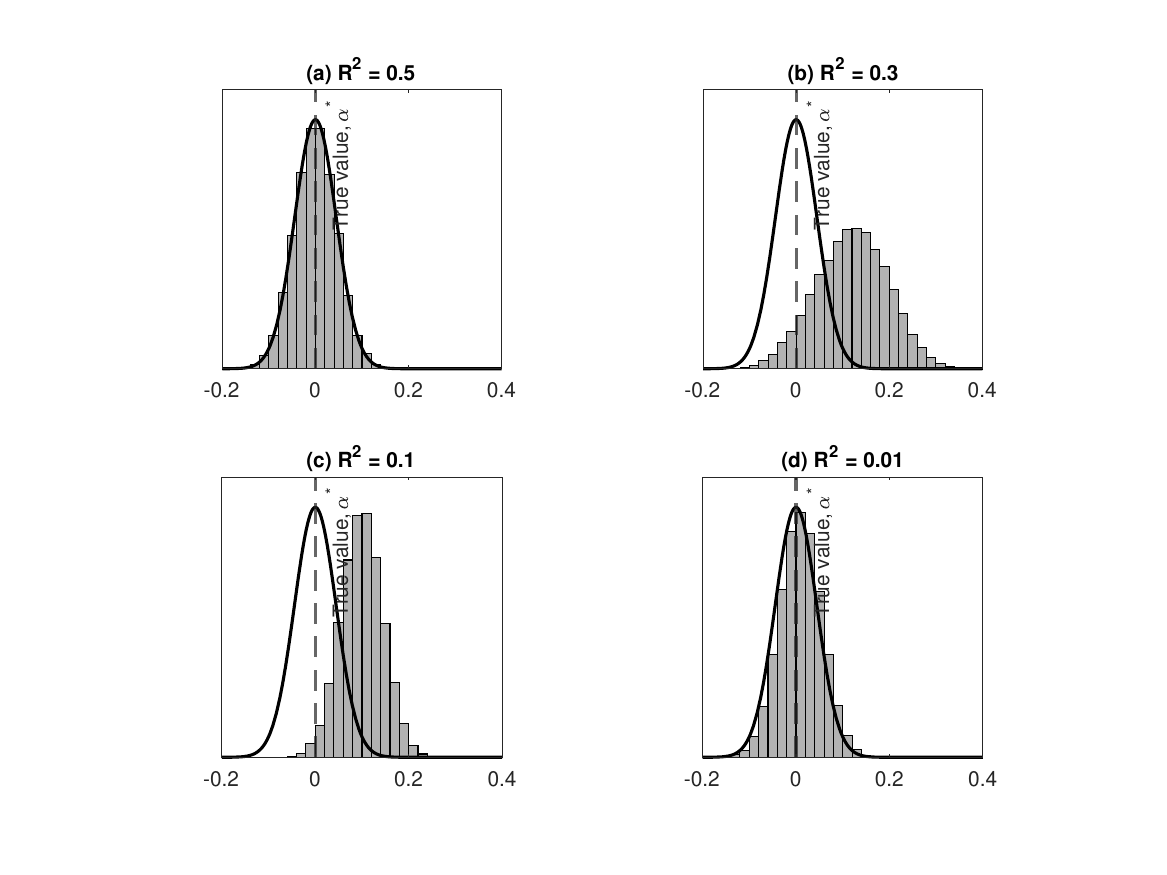}
\end{center}
\doublespacing{\textit{Notes:} The grey histograms show the finite sample distributions and the black curves show the densities of the oracle estimators.}
\label{fig:fsd_dml}
\end{figure}

\newpage

\begin{figure}[H]
\caption{Number of selected relevant controls}
\begin{center}
\includegraphics[width=0.9\textwidth, trim = {0 1.5cm 0 0.5cm}]{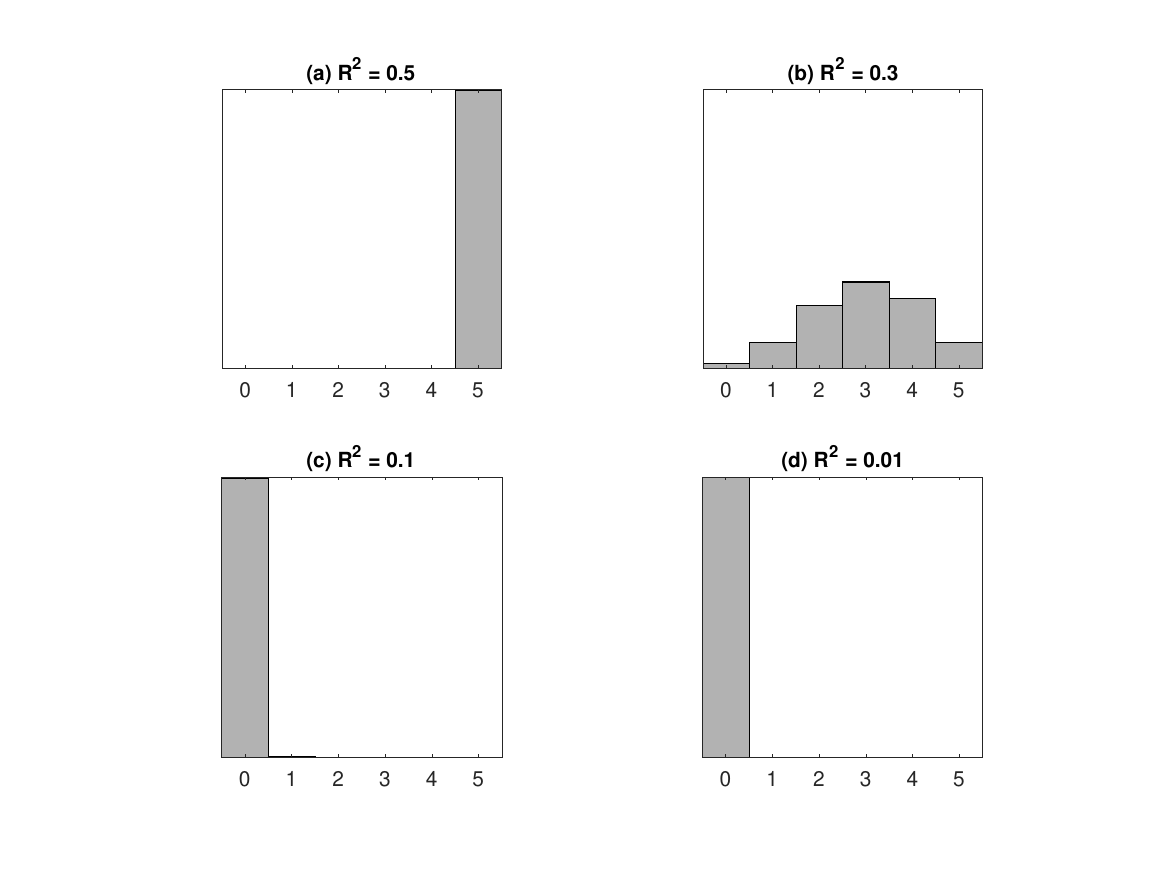}
\end{center}
\label{fig:hist_sel_dml}
\end{figure}

\newpage

\begin{figure}[H]
\caption{Average number of selected controls}

    \begin{subfigure}[b]{\textwidth}
    \caption{All controls}
    \centering
        \includegraphics[width=0.325\textwidth, trim = {0 0cm 0 0cm}]{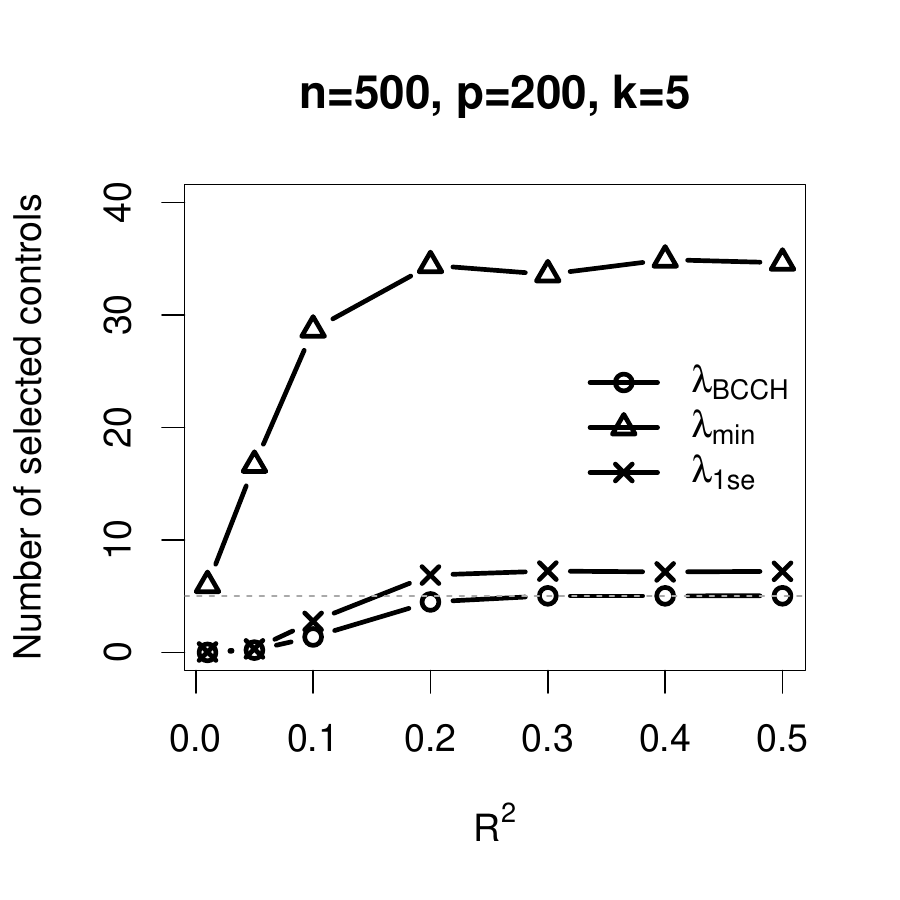}
        \includegraphics[width=0.325\textwidth, trim = {0 0cm 0 0cm}]{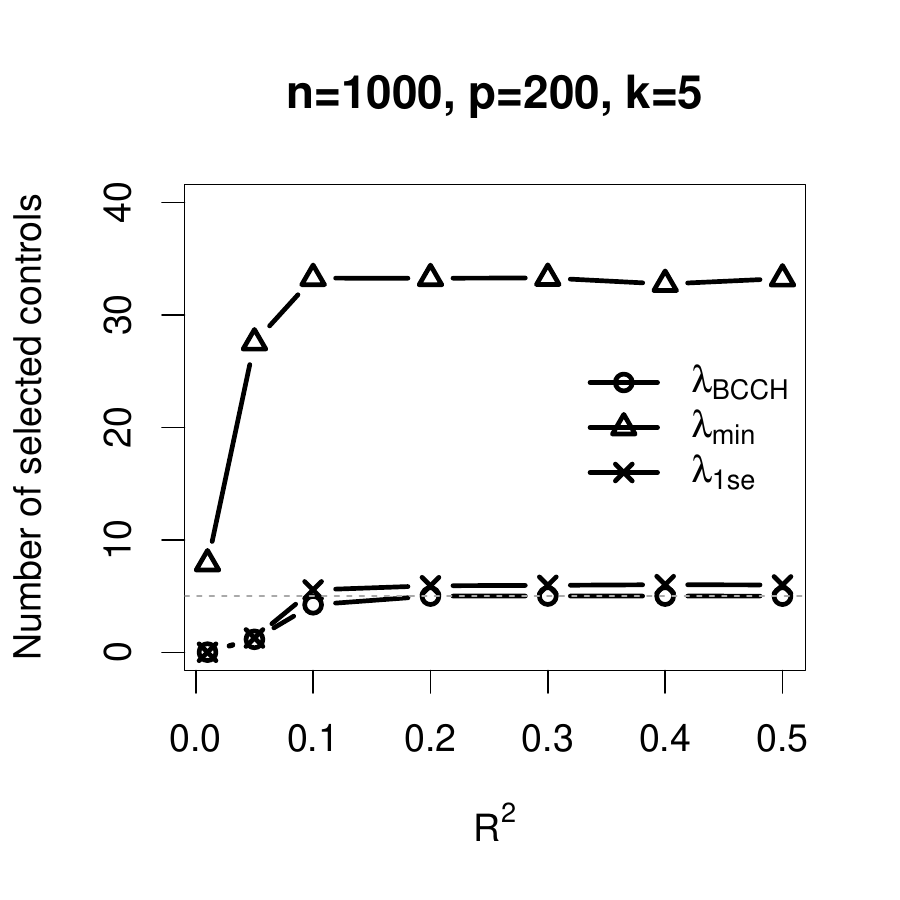}
        \includegraphics[width=0.325\textwidth, trim = {0 0cm 0 0cm}]{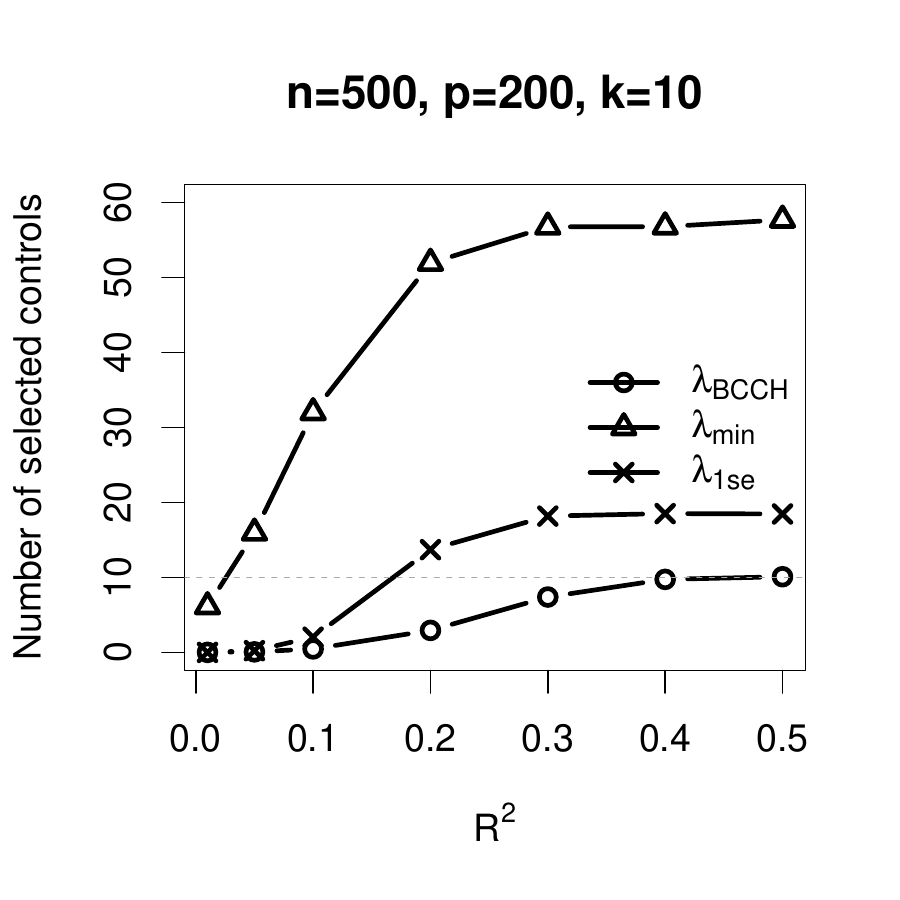}
    \end{subfigure}  
    
    \begin{subfigure}[b]{\textwidth}
    \caption{Relevant controls}
    \centering
        \includegraphics[width=0.325\textwidth, trim = {0 1cm 0 0cm}]{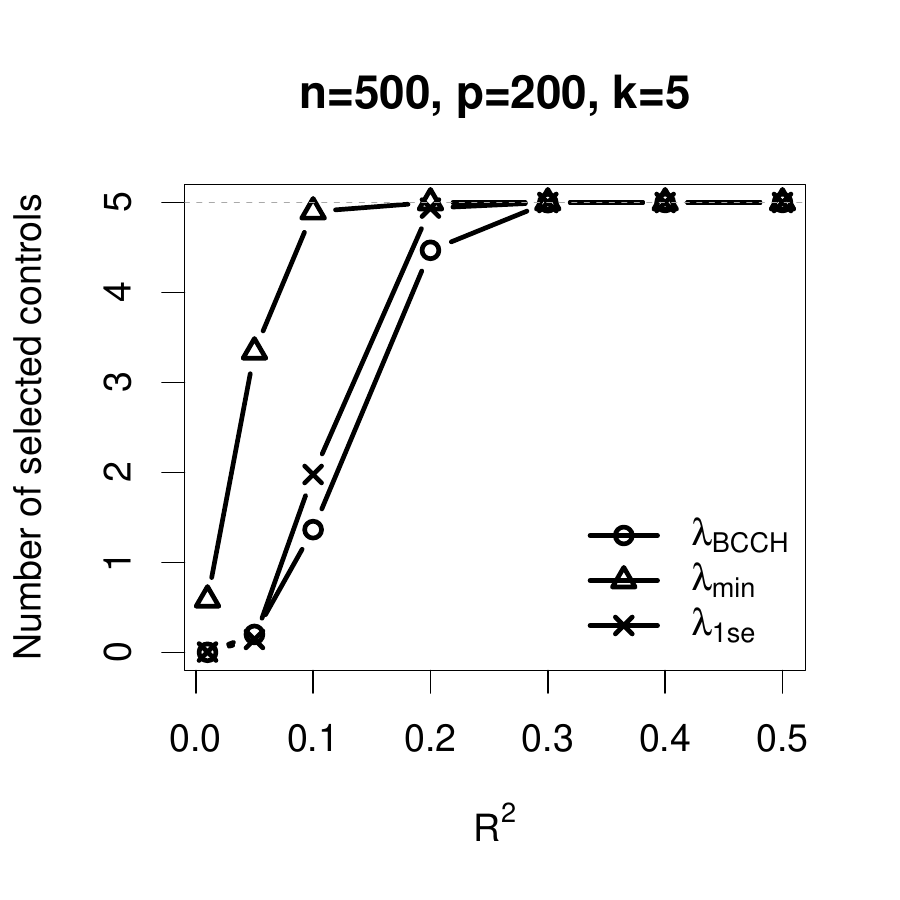}
        \includegraphics[width=0.325\textwidth, trim = {0 1cm 0 0cm}]{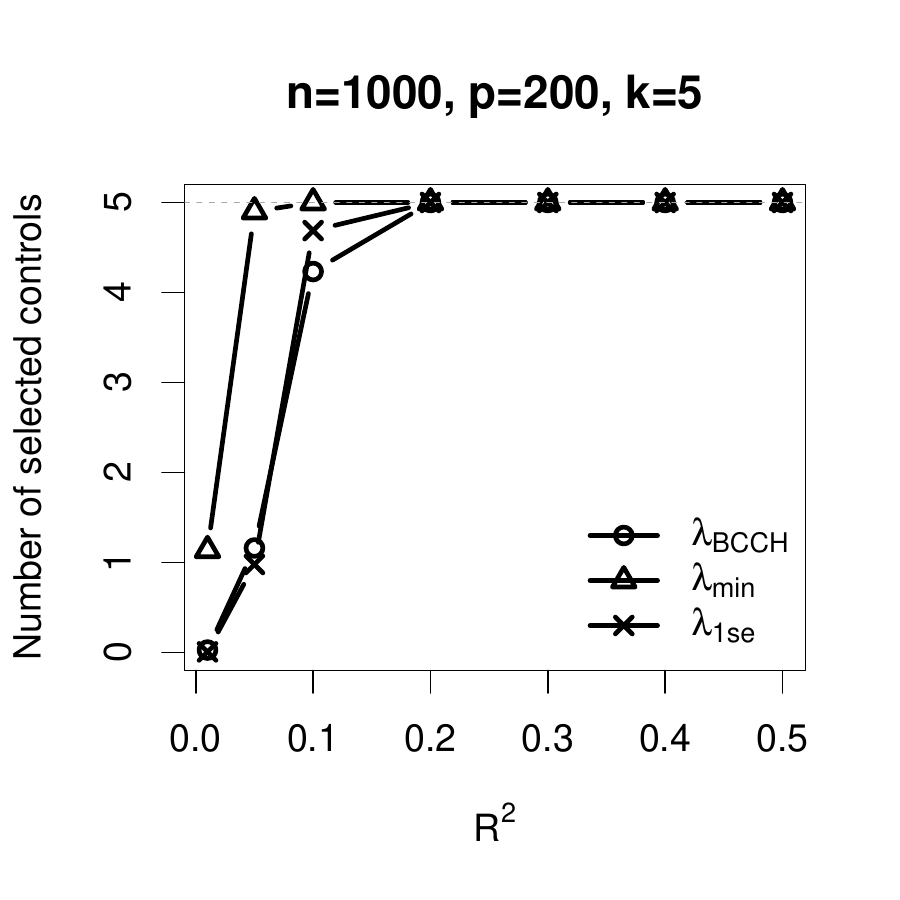}
        \includegraphics[width=0.325\textwidth, trim = {0 1cm 0 0cm}]{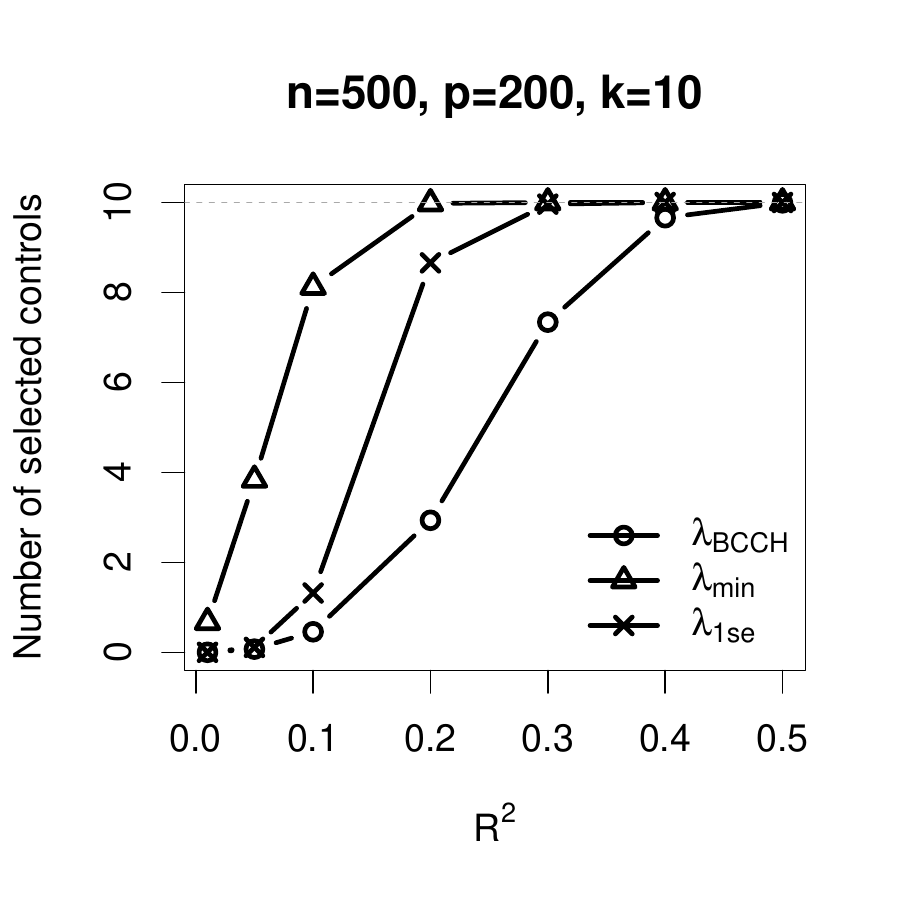}
    \end{subfigure}  
\label{fig:sel}

\end{figure}

\newpage
\begin{figure}[H]
\caption{Ratio of bias to standard deviation}
\begin{center}
\includegraphics[width=0.325\textwidth, trim = {0 1.5cm 0 0.6cm}]{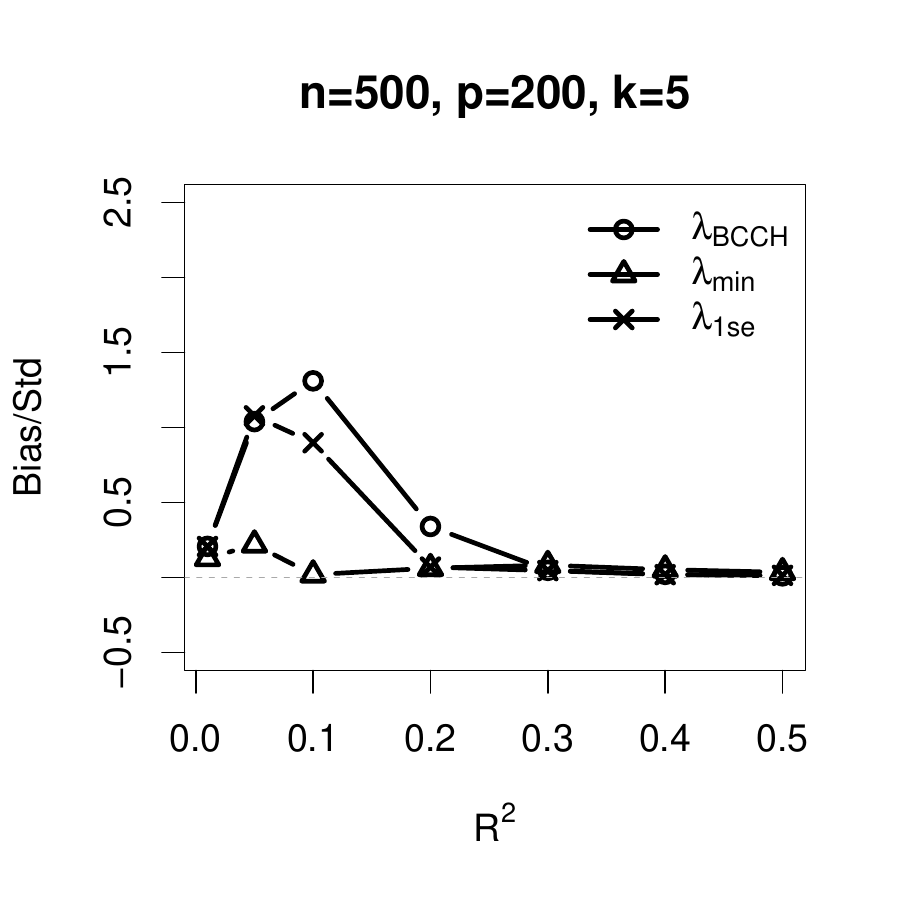}
\includegraphics[width=0.325\textwidth, trim = {0 1.5cm 0 0.6cm}]{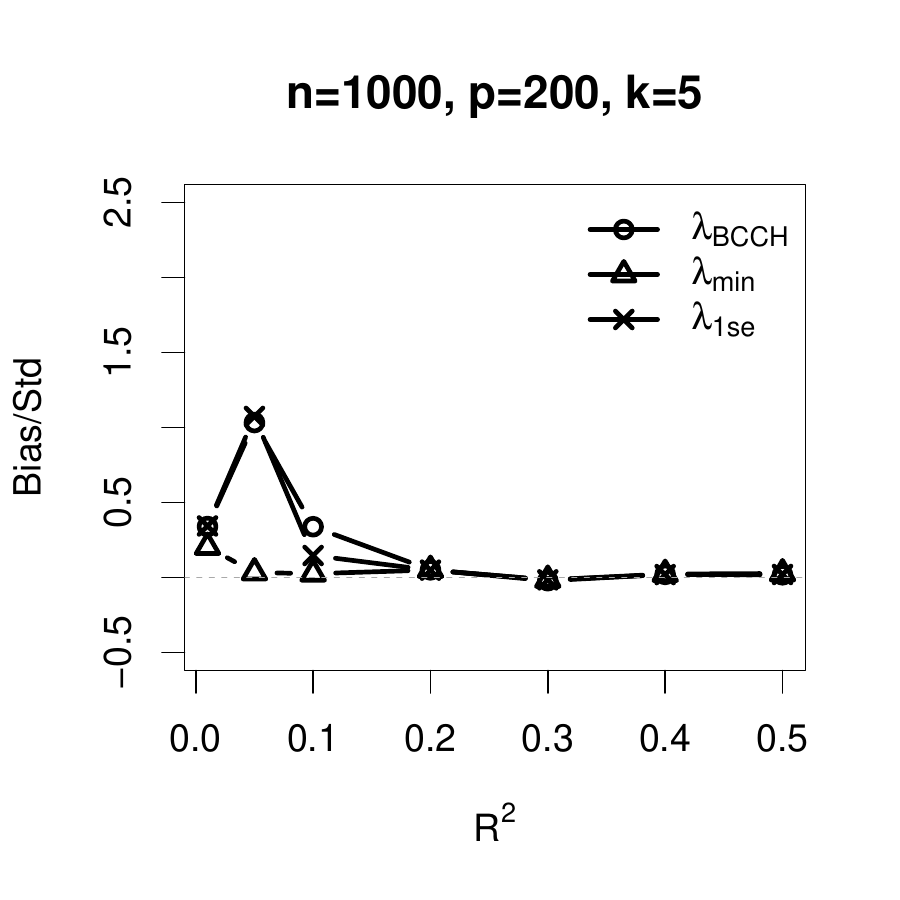}
\includegraphics[width=0.325\textwidth, trim = {0 1.5cm 0 0.6cm}]{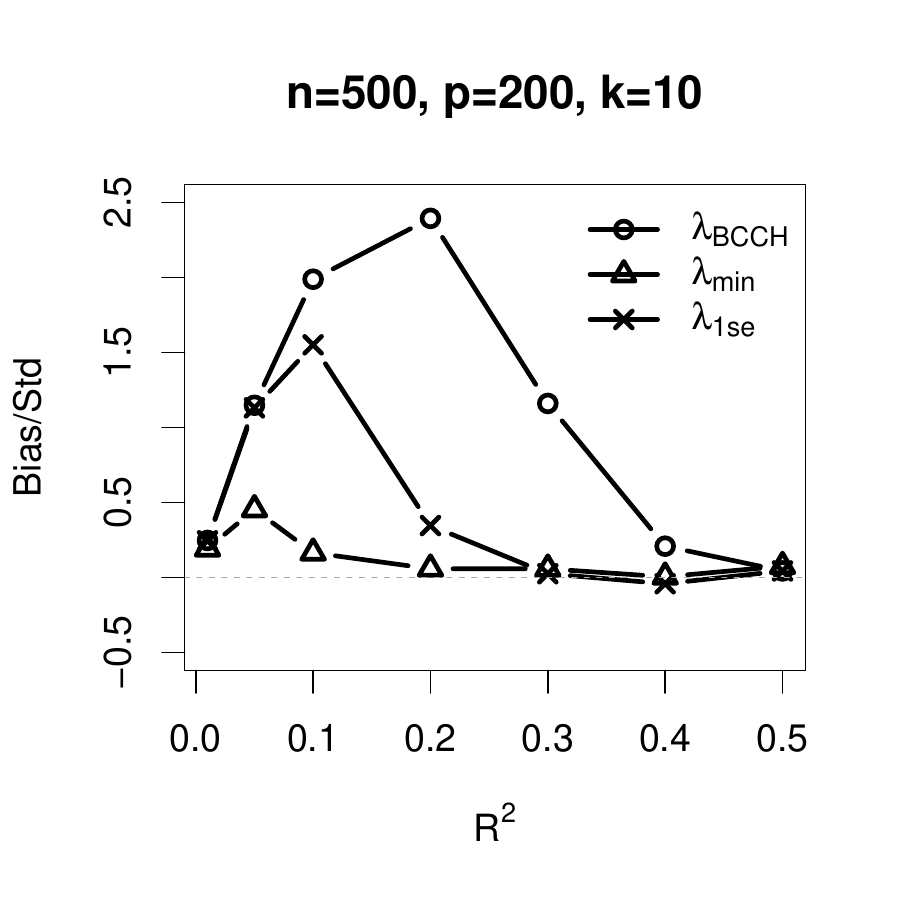}
\end{center}
\label{fig:dml_bias_std}
\end{figure}

\newpage

\begin{figure}[H]
\caption{Average number of selected controls: sensitivity to penalty level}

    \begin{subfigure}[b]{\textwidth}
    \caption{All controls}
    \centering
        \includegraphics[width=0.325\textwidth, trim = {0 0cm 0 0cm}]{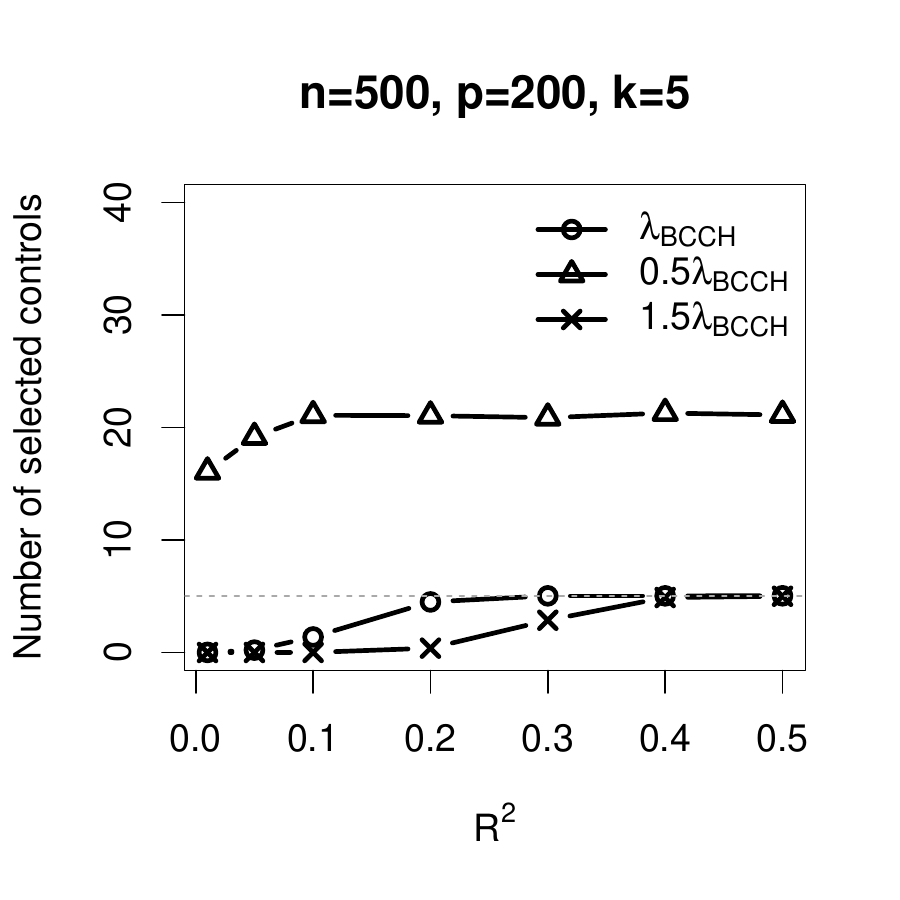}
        \includegraphics[width=0.325\textwidth, trim = {0 0cm 0 0cm}]{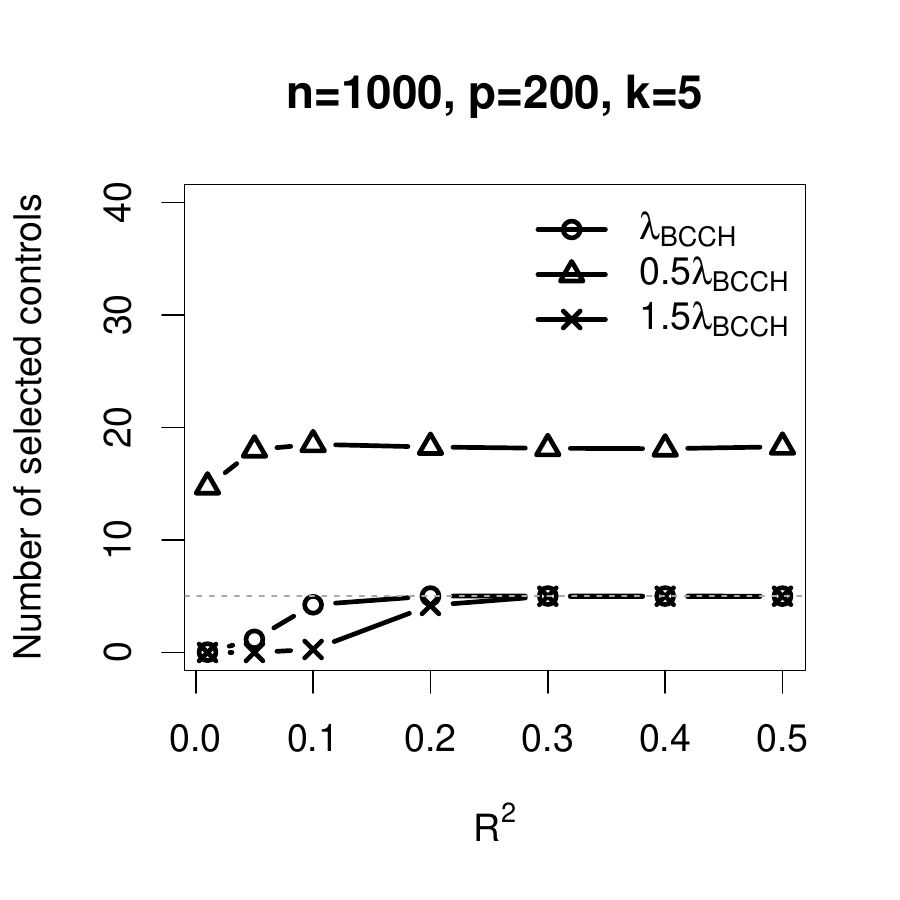}
        \includegraphics[width=0.325\textwidth, trim = {0 0cm 0 0cm}]{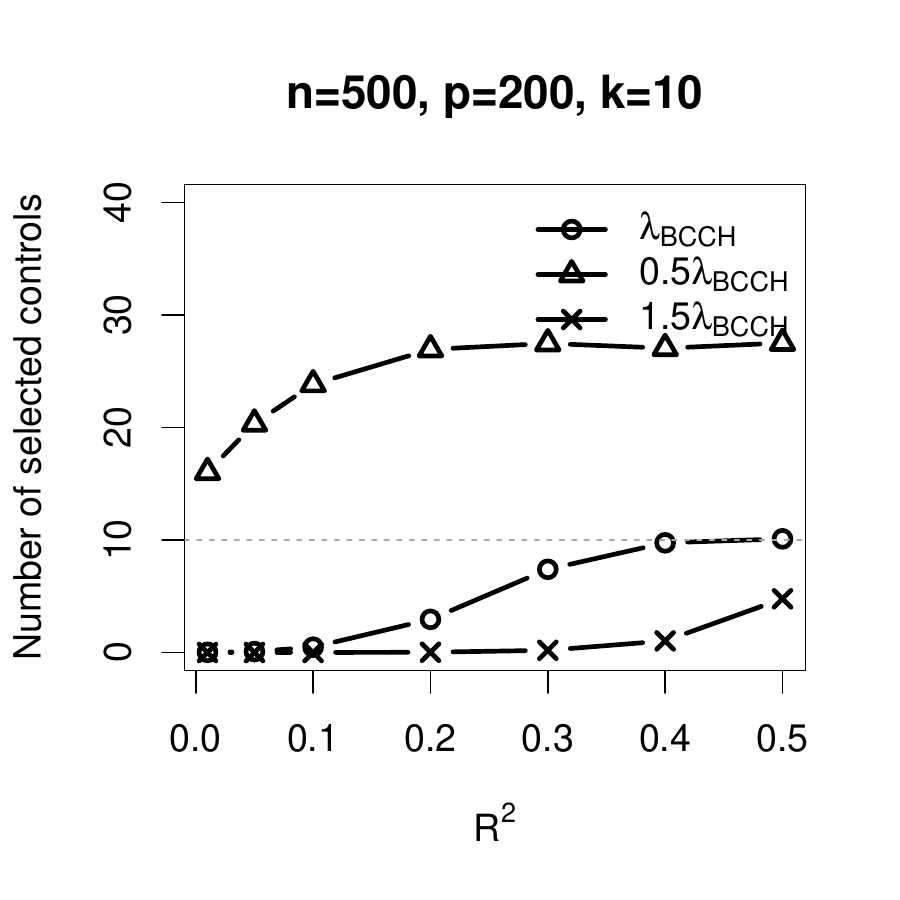}
    \end{subfigure}  
    
    \begin{subfigure}[b]{\textwidth}
    \caption{Relevant controls}
    \centering
        \includegraphics[width=0.325\textwidth, trim = {0 1cm 0 0cm}]{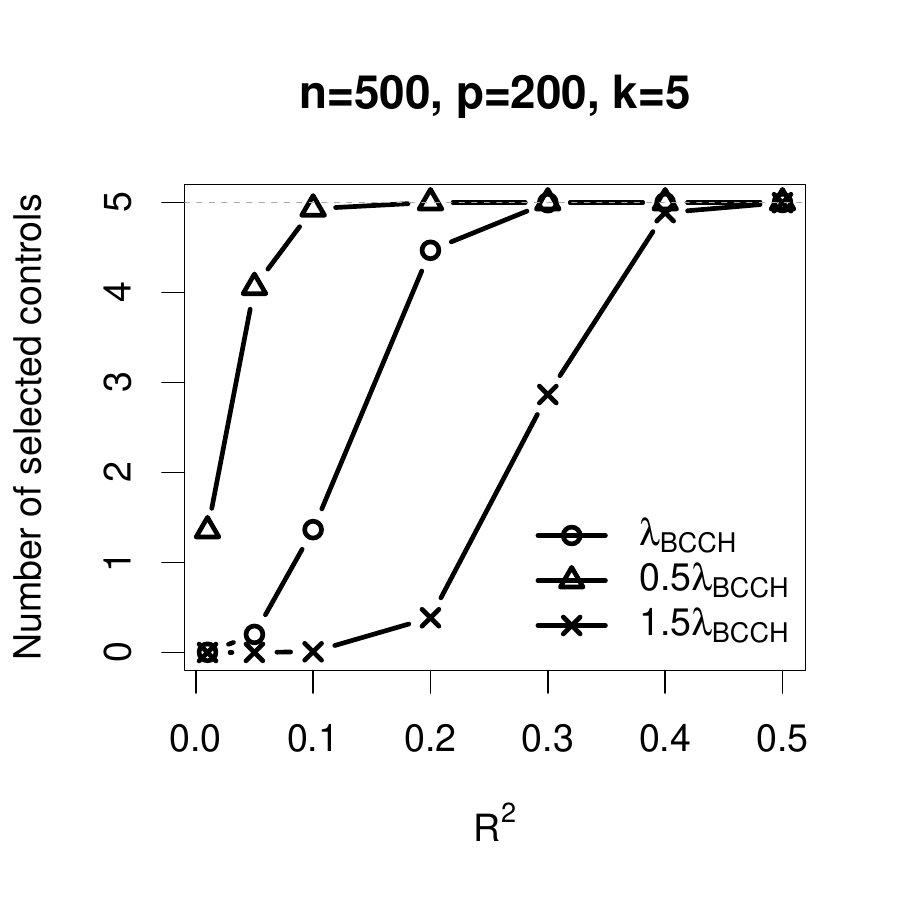}
        \includegraphics[width=0.325\textwidth, trim = {0 1cm 0 0cm}]{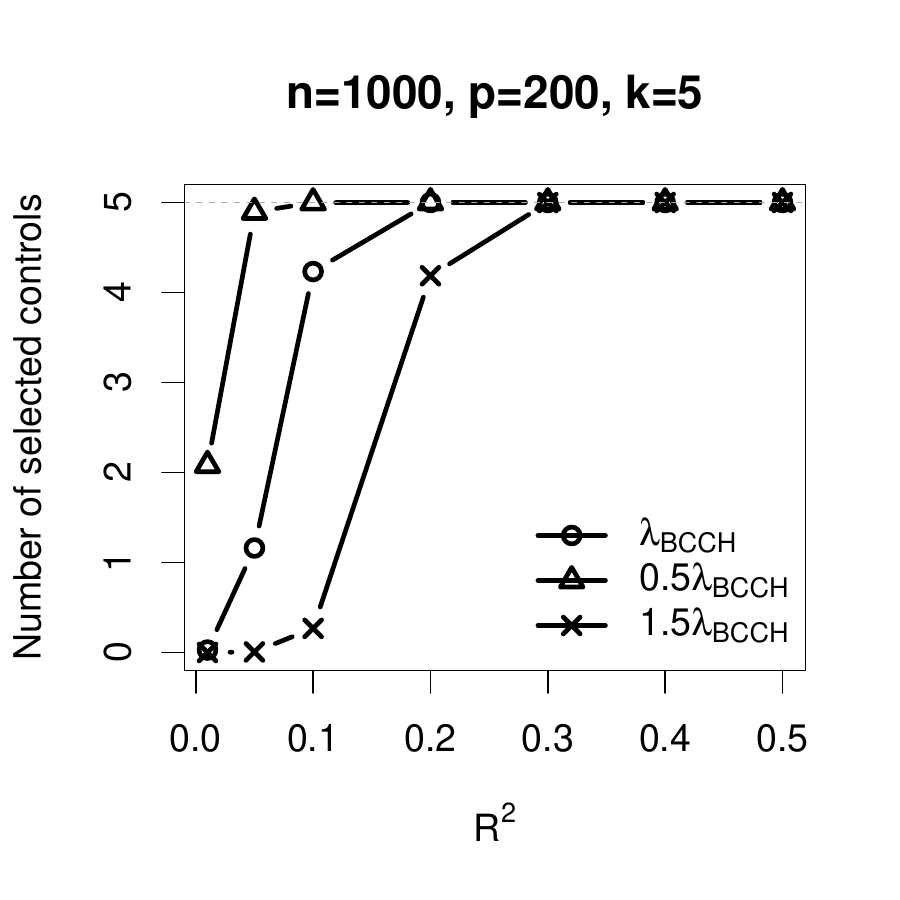}
        \includegraphics[width=0.325\textwidth, trim = {0 1cm 0 0cm}]{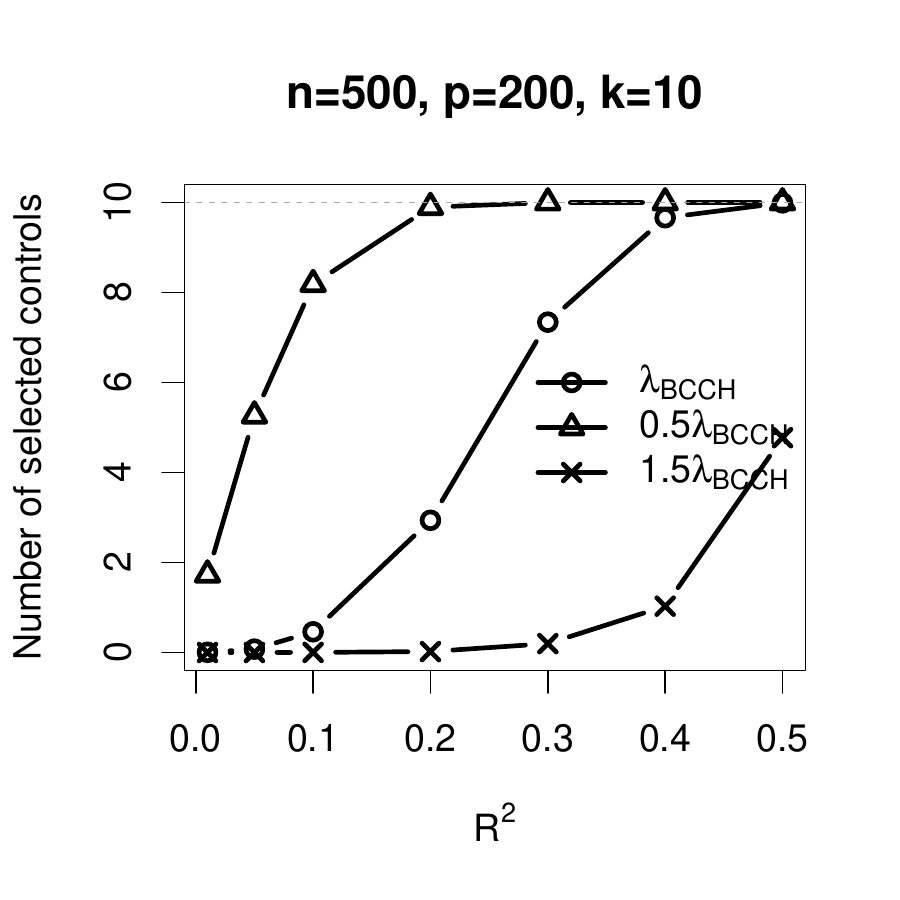}
    \end{subfigure}  
\label{fig:sel_sensitivity}

\end{figure}

\newpage

\begin{figure}[H]
\caption{Ratio of bias to standard deviation: sensitivity to penalty level}
\begin{center}
\includegraphics[width=0.325\textwidth, trim = {0 1.5cm 0 0.6cm}]{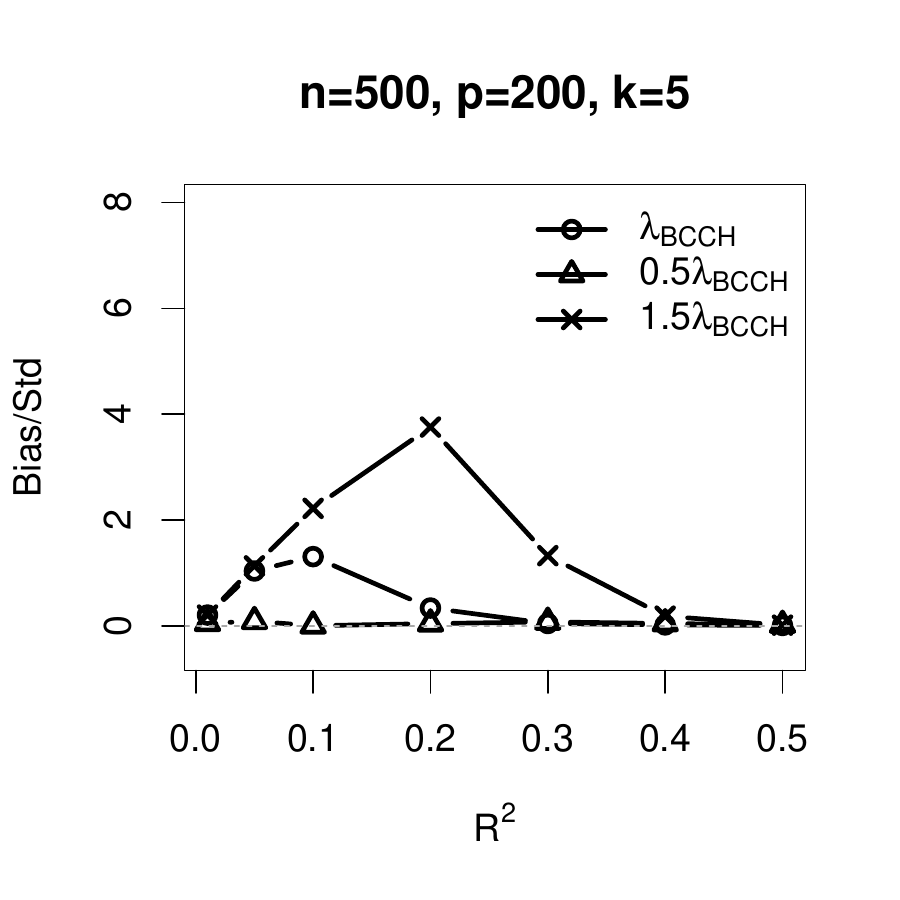}
\includegraphics[width=0.325\textwidth, trim = {0 1.5cm 0 0.6cm}]{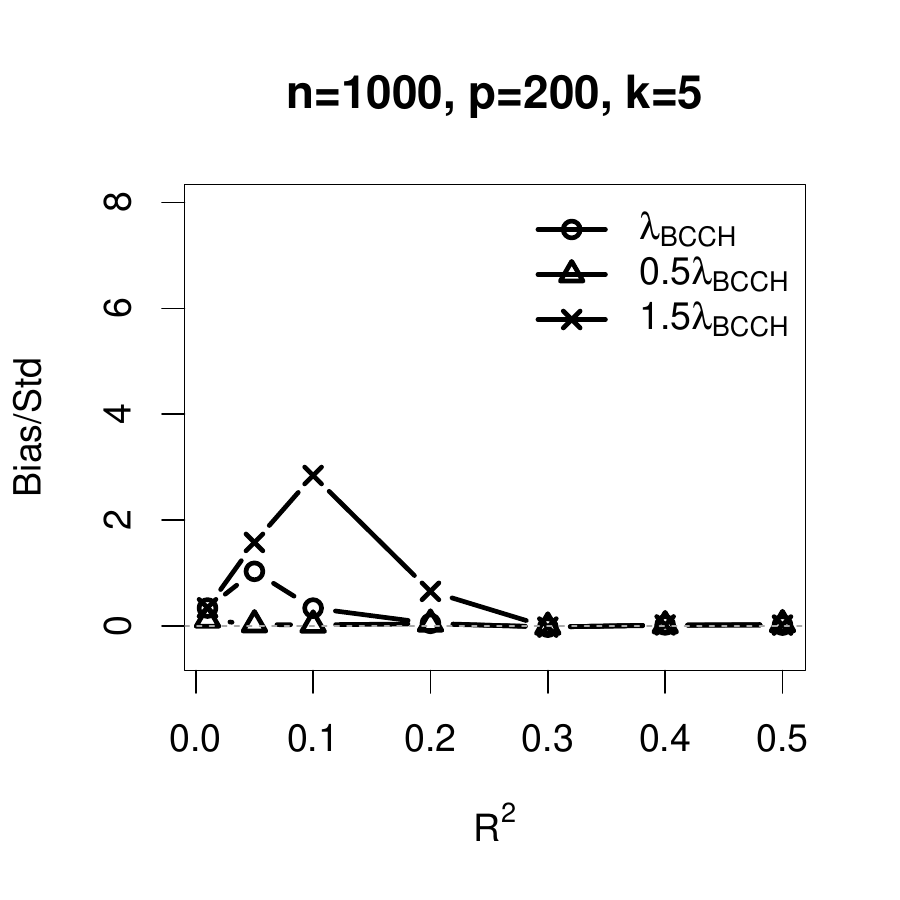}
\includegraphics[width=0.325\textwidth, trim = {0 1.5cm 0 0.6cm}]{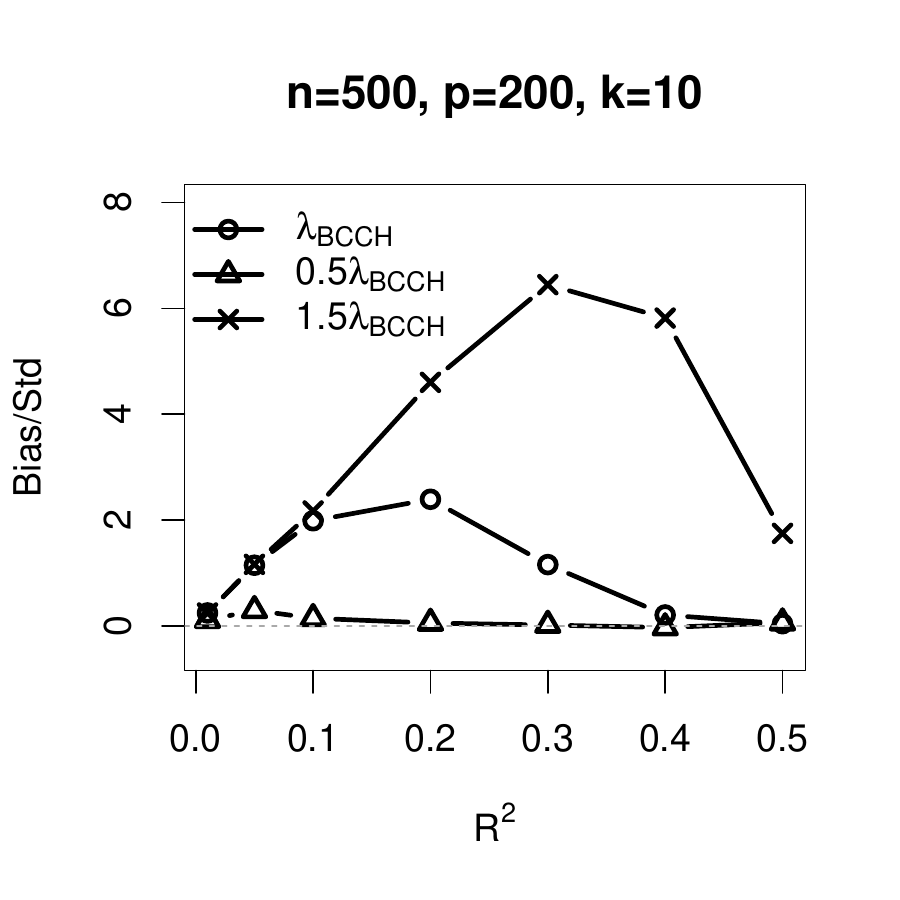}
\end{center}
\label{fig:dml_bias_std_sens}
\end{figure}

\newpage

\begin{figure}[H]
\caption{Bias and ratio of bias to standard deviation}

\vspace{-0.5cm}

\begin{center}
    \begin{subfigure}[b]{\textwidth}
    \caption{TWI specification}
    \centering
    \includegraphics[width=0.4\textwidth, trim = {0 1cm 0 1cm}]{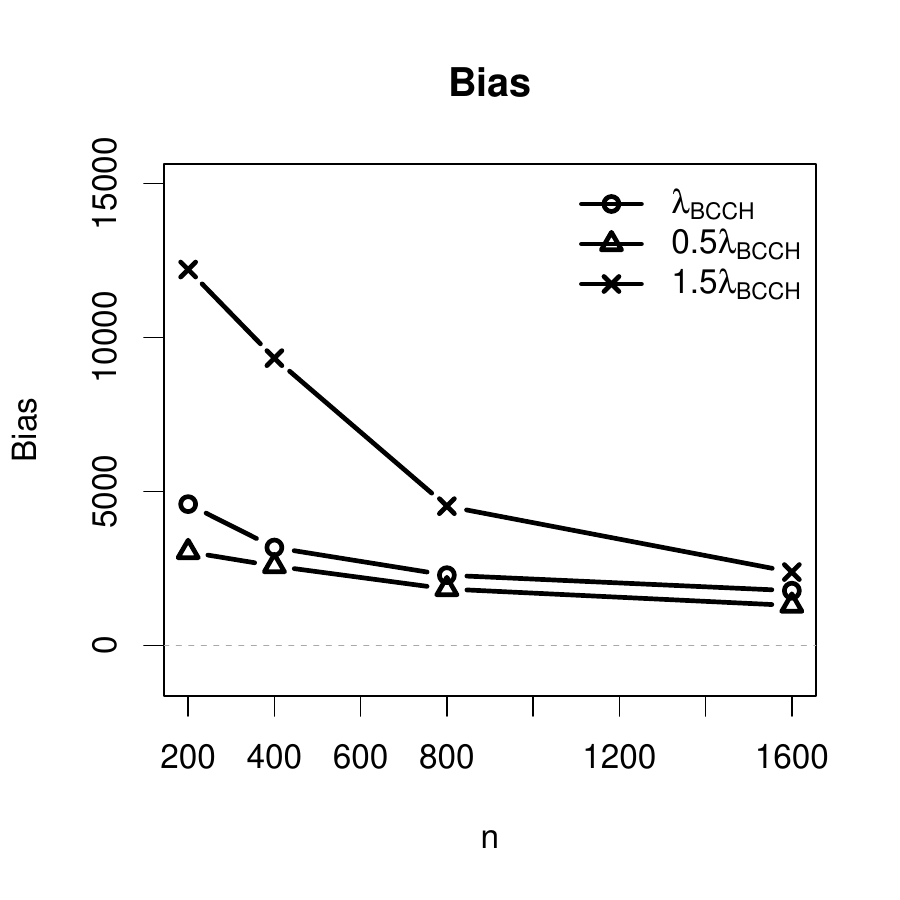}
    \includegraphics[width=0.4\textwidth, trim = {0 1cm 0 1cm}]{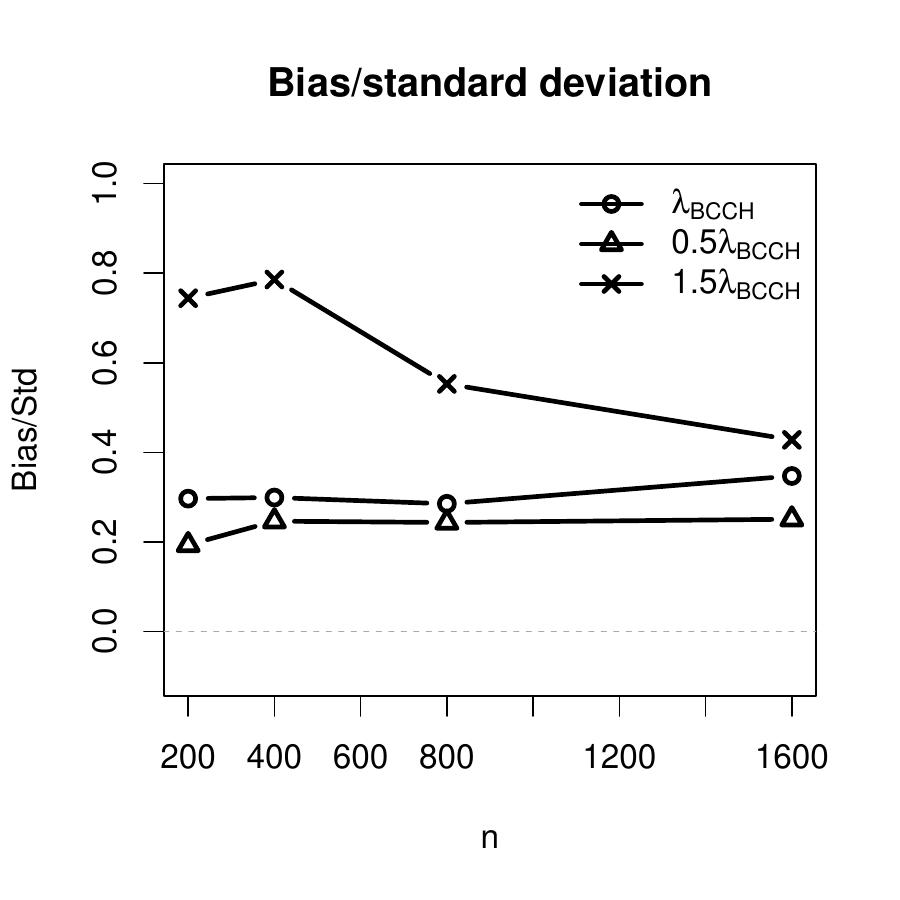}
    \end{subfigure}  
    
    \begin{subfigure}[b]{\textwidth}
    \caption{QSI specification}
    \centering
    \includegraphics[width=0.4\textwidth, trim = {0 1.5cm 0 1cm}]{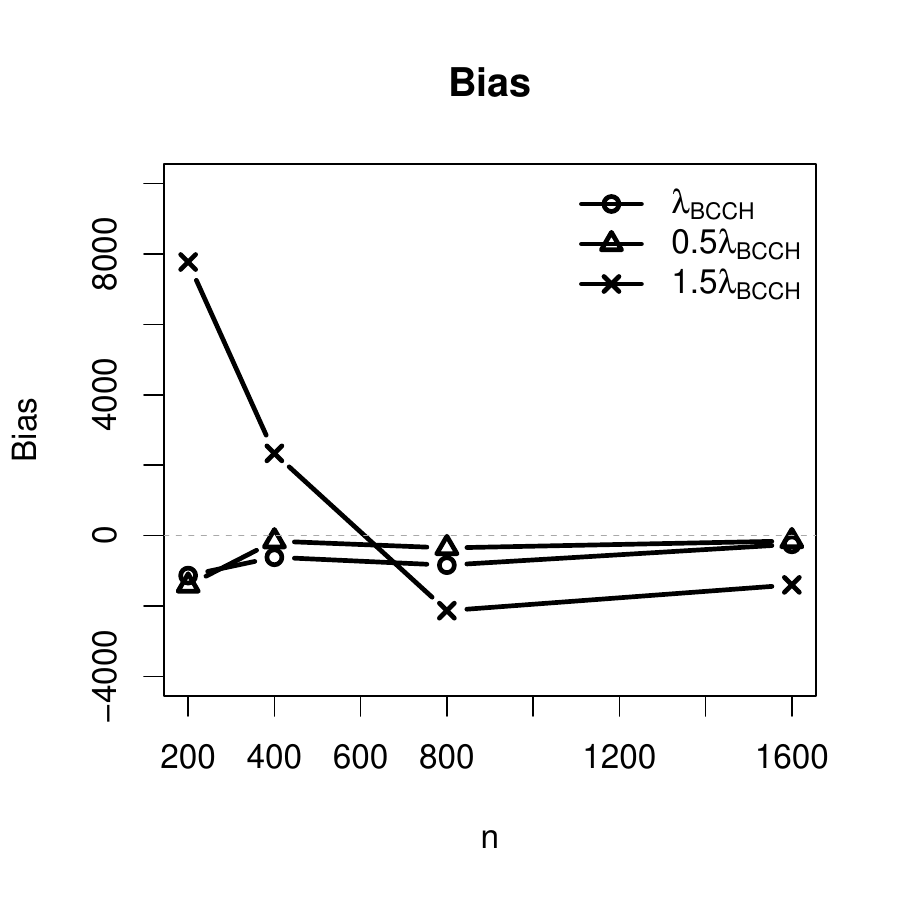}
    \includegraphics[width=0.4\textwidth, trim = {0 1.5cm 0 1cm}]{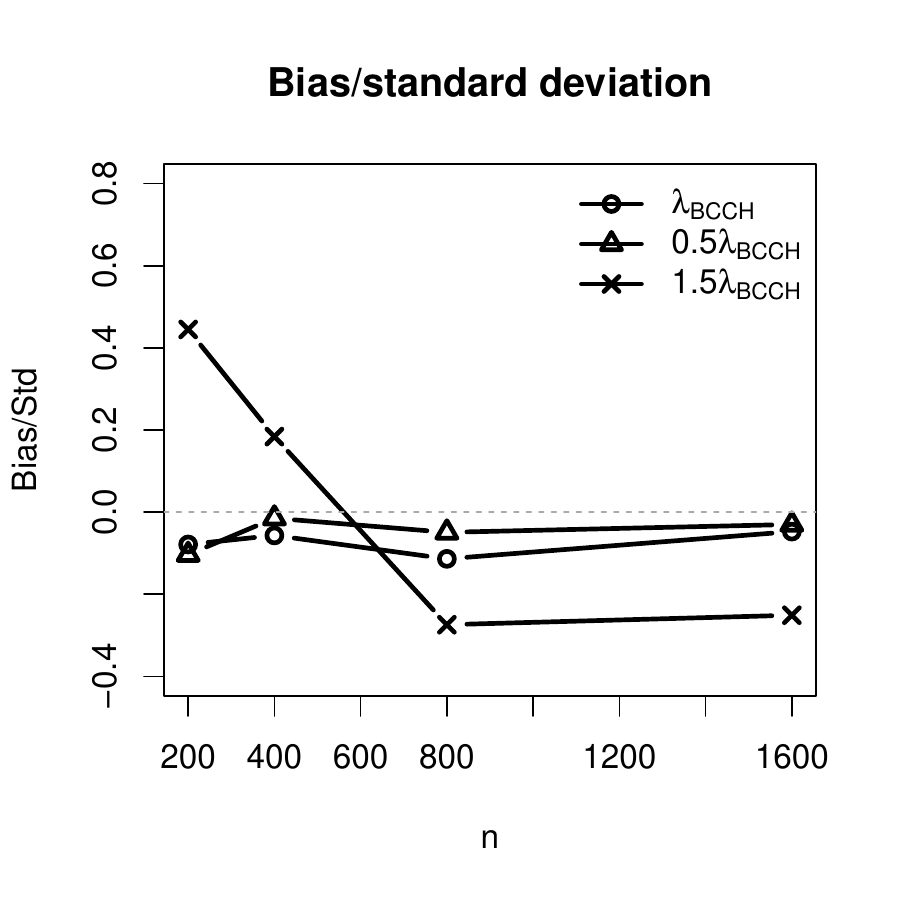}
    \end{subfigure}  
\end{center}
\label{fig:app_401k}
\end{figure}

\newpage

\begin{figure}[H]
\caption{Bias and ratio of bias to standard deviation}
\begin{center}
\includegraphics[width=0.4\textwidth, trim = {0 1.5cm 0 1.5cm}]{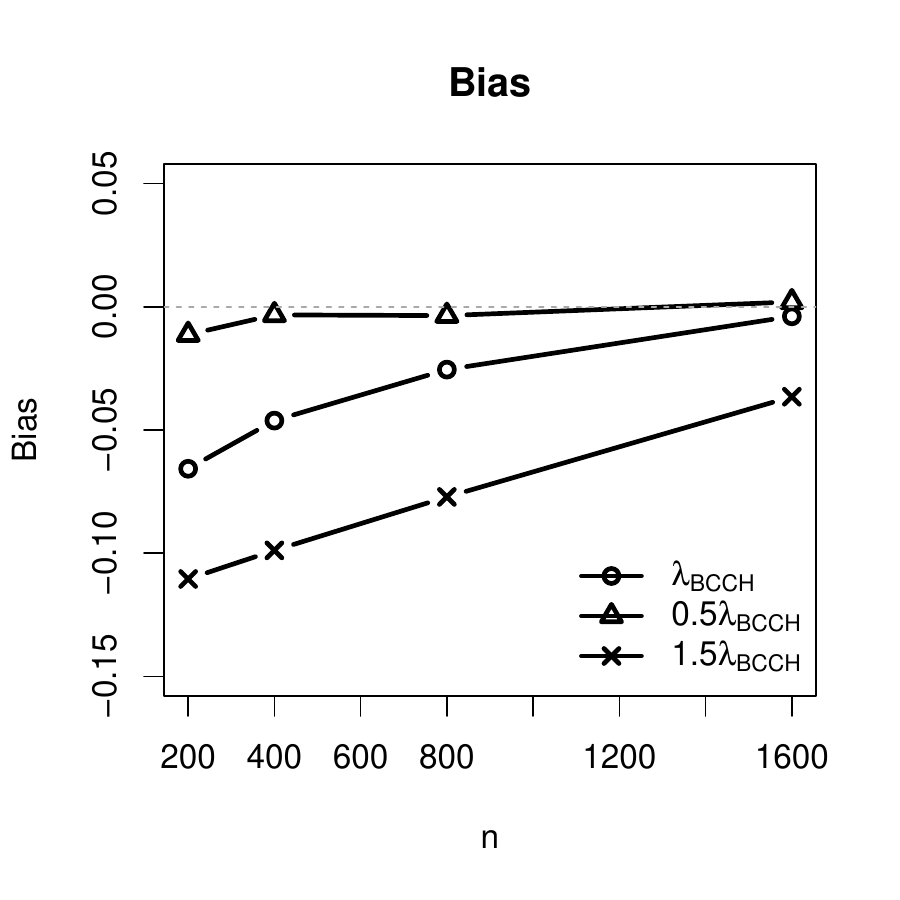}
\includegraphics[width=0.4\textwidth, trim = {0 1.5cm 0 1.5cm}]{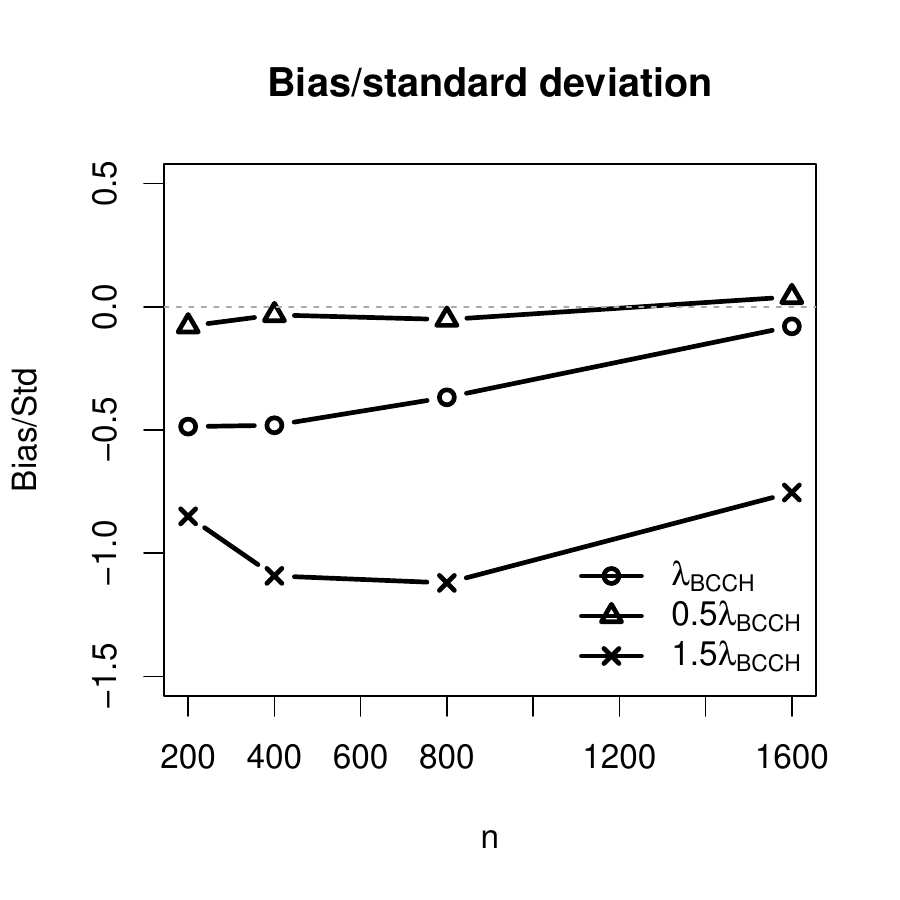}
\end{center}
\label{fig:app_test}
\end{figure}

\newpage

\begin{figure}[H]
\caption{Coverage 90\% confidence intervals}

\vspace{-0.5cm}

\begin{center}
\includegraphics[width=0.325\textwidth, trim = {0 1.5cm 0 0.6cm}]{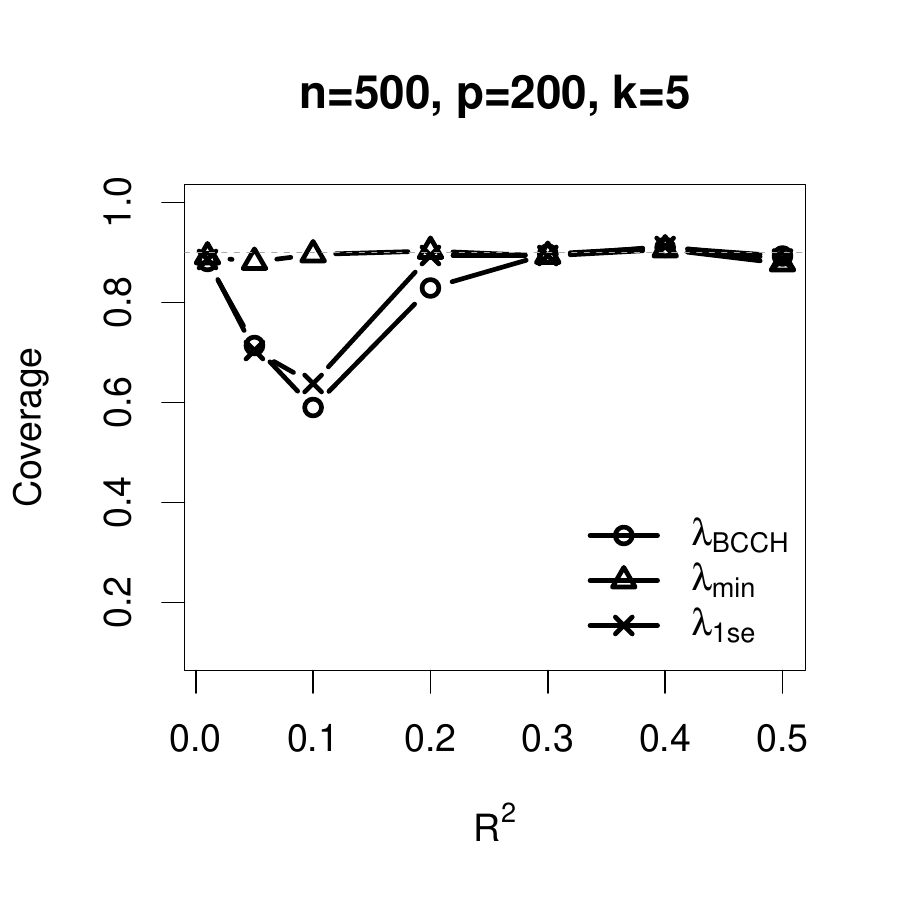}
\includegraphics[width=0.325\textwidth, trim = {0 1.5cm 0 0.6cm}]{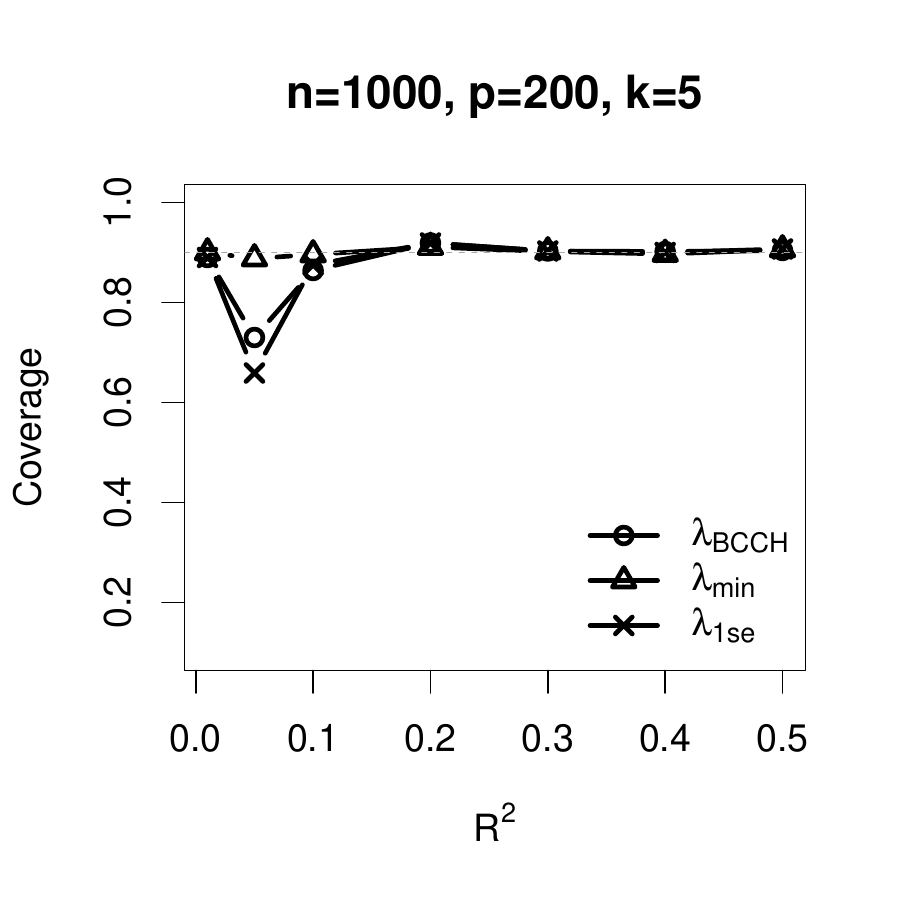}
\includegraphics[width=0.325\textwidth, trim = {0 1.5cm 0 0.6cm}]{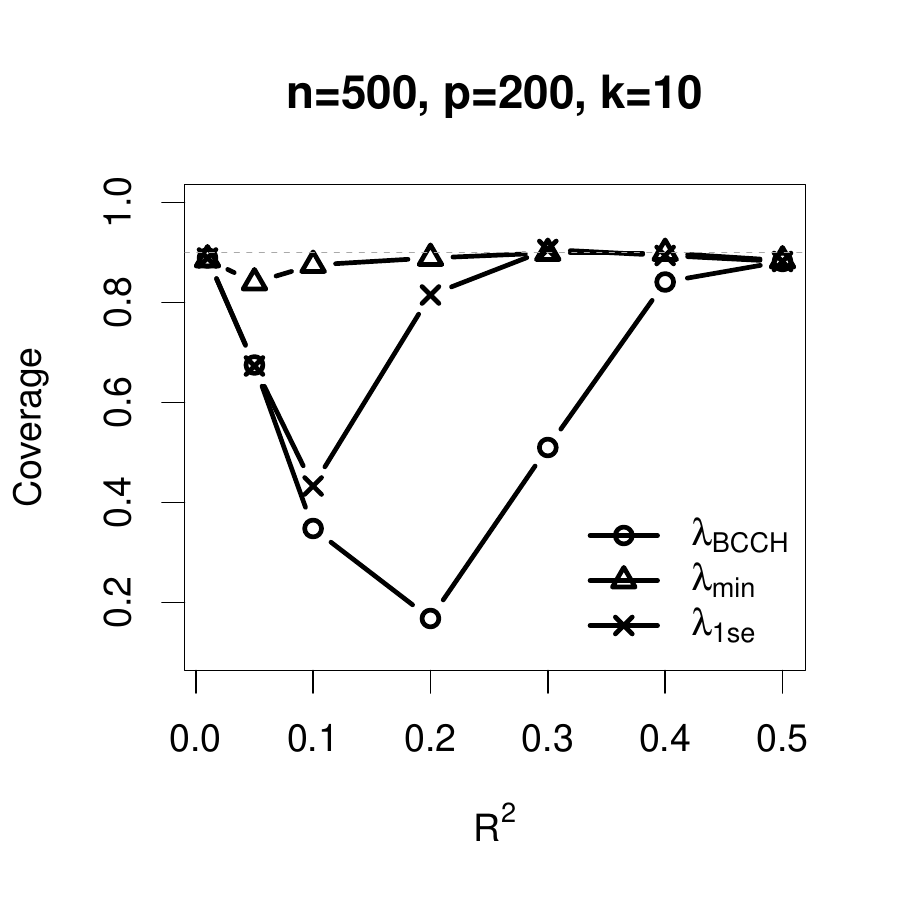}
\end{center}
\label{fig:dml_cov}
\end{figure}

\newpage

\begin{figure}[H]
\caption{Comparison to OLS with HCK standard errors}

    \begin{subfigure}[b]{\textwidth}
    \caption{Coverage 90\% confidence intervals}
    \centering
    \includegraphics[width=0.325\textwidth, trim = {0 0 0 0cm}]{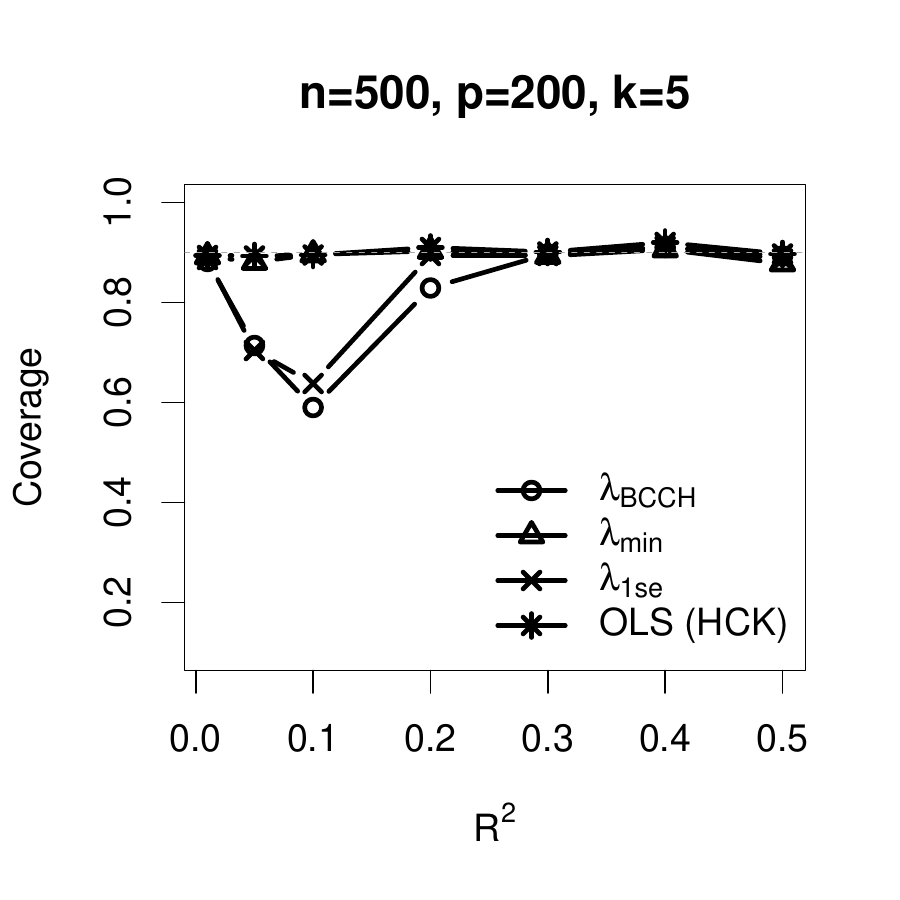}
    \includegraphics[width=0.325\textwidth, trim = {0 0 0 0cm}]{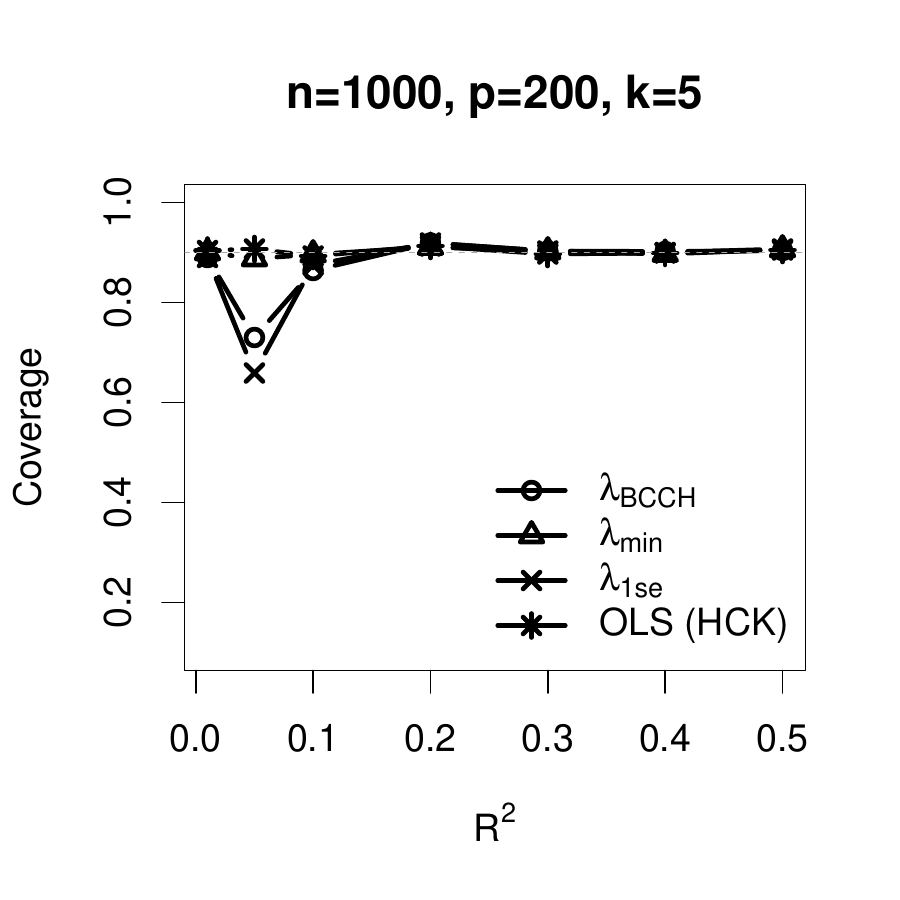}
    \includegraphics[width=0.325\textwidth, trim = {0 0 0 0cm}]{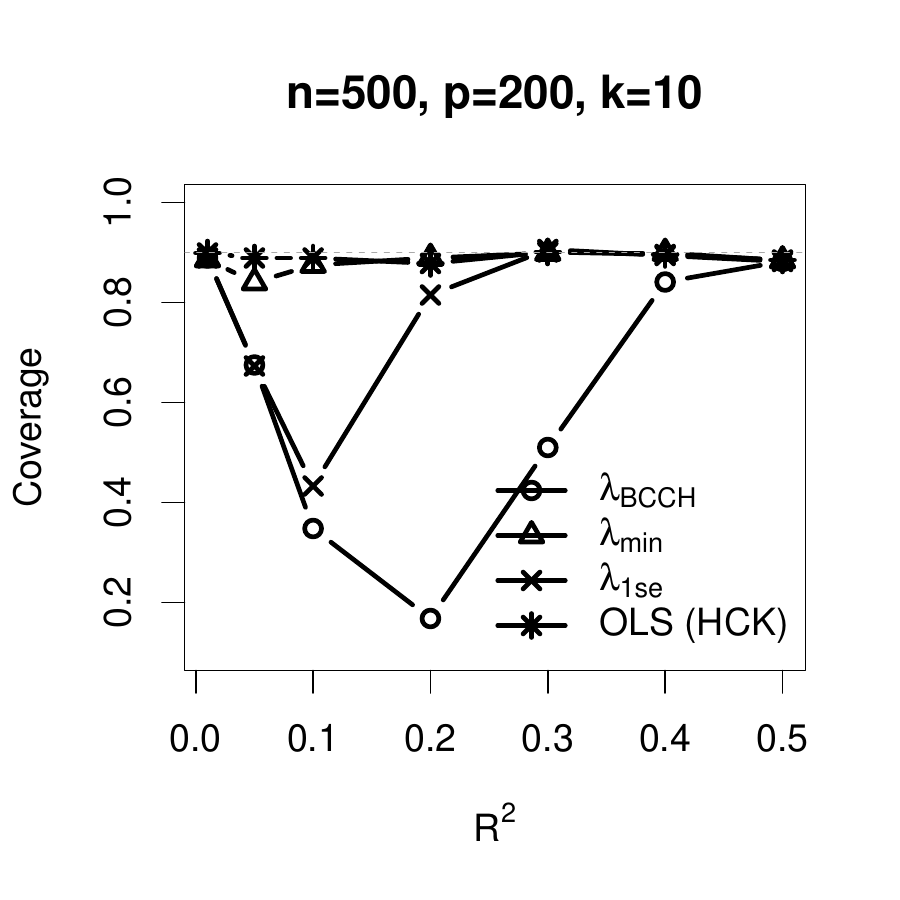}
    \end{subfigure}  
    
    \begin{subfigure}[b]{\textwidth}
    \caption{Average length 90\% confidence intervals}
    \centering
    \includegraphics[width=0.325\textwidth, trim = {0 1cm 0 0cm}]{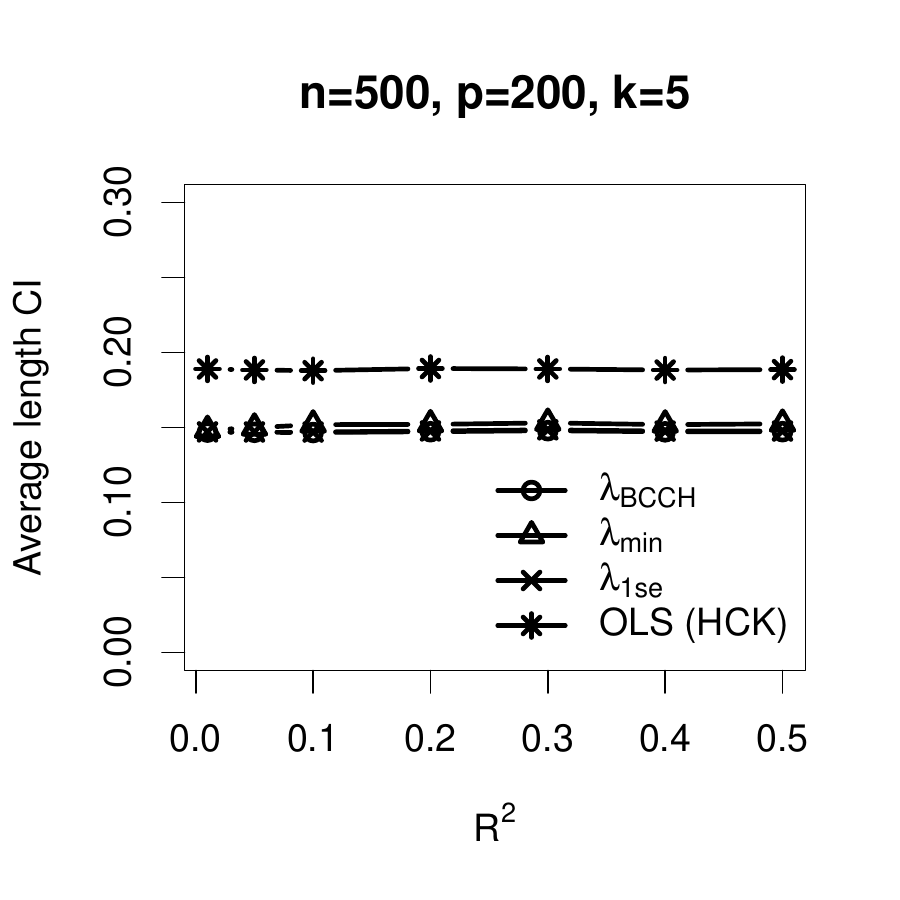}
    \includegraphics[width=0.325\textwidth, trim = {0 1cm 0 0cm}]{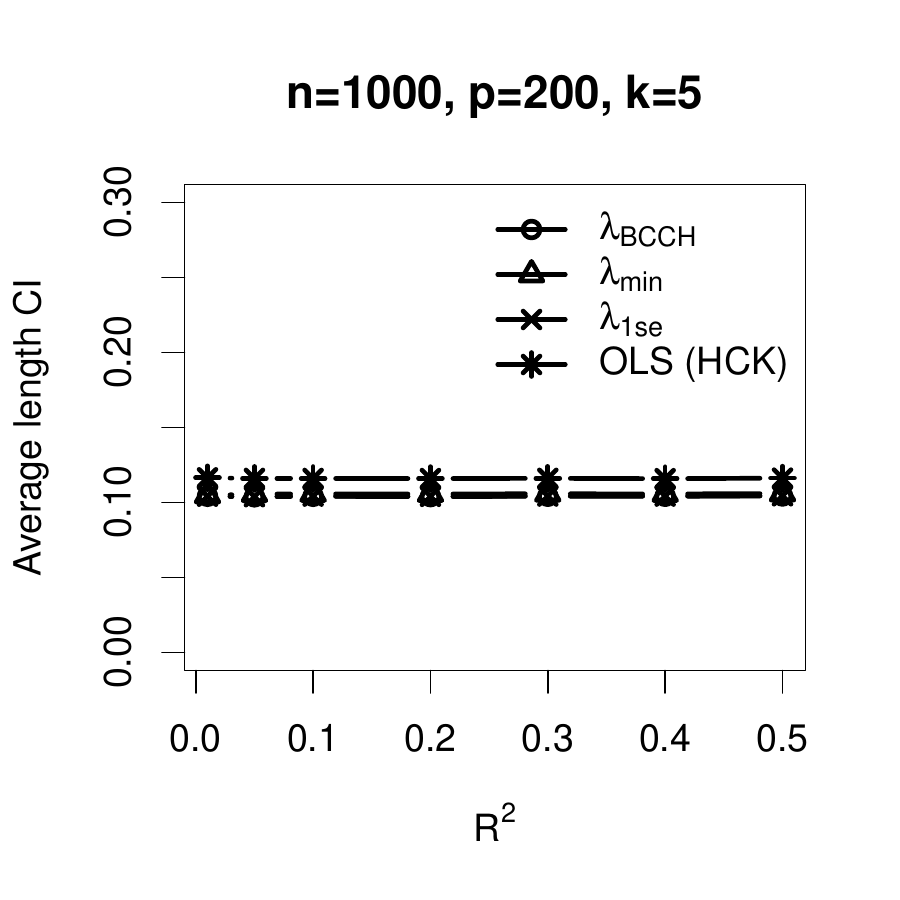}
    \includegraphics[width=0.325\textwidth, trim = {0 1cm 0 0cm}]{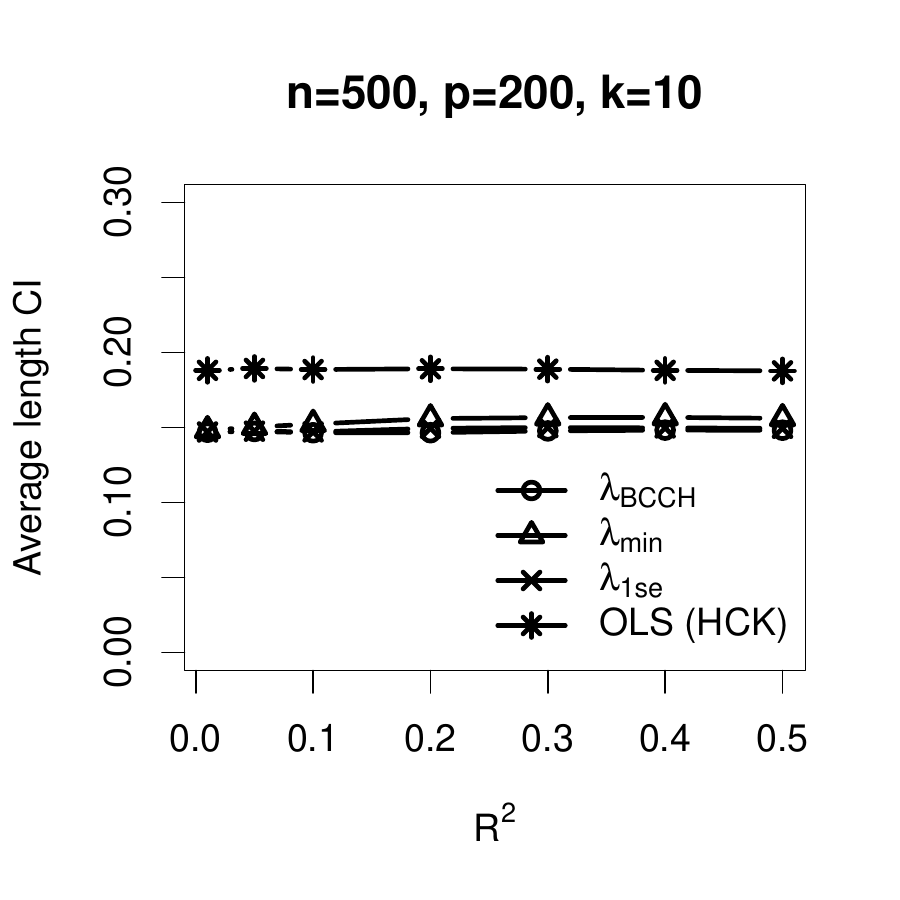}
    \end{subfigure}  
\label{fig:comparison_ols}

\end{figure}

\newpage

 \begin{table}[H]
\centering \caption{Results based on original data}
\begin{footnotesize}
\begin{tabular}{lcc}\label{tab:results_401k}
 &  &  \tabularnewline
 \toprule
\midrule
\multicolumn{3}{c}{TWI specification ($p=167$)}\tabularnewline
\midrule 
Method & Point estimate  & Robust std.\ error  \tabularnewline
\midrule 
Post double Lasso ($\lambda_{\text{BCCH}}$)& 6624.47 & 2069.62 \\ 
Post double Lasso ($0.5\lambda_{\text{BCCH}}$) & 6432.36 & 2073.25 \\ 
Post double Lasso (1.5$\lambda_{\text{BCCH}}$) & 7474.51 & 2052.89 \\ 
OLS with all controls & 6751.91 & 2067.86 \\ 
OLS without controls & 35669.52 & 2412.02 \\
\midrule 
\multicolumn{3}{c}{QSI specification ($p=272$)}\tabularnewline
\midrule 
Method & Point estimate  & Robust std.\ error  \tabularnewline
\midrule 
Post double Lasso ($\lambda_{\text{BCCH}}$)  & 4558.97 & 2015.90 \\ 
Post double Lasso ($0.5\lambda_{\text{BCCH}}$) & 5638.04 & 1992.68 \\ 
Post double Lasso (1.5$\lambda_{\text{BCCH}}$) & 4406.55 & 2028.20 \\ 
OLS with all controls & 5988.41 & 2033.02 \\ 
OLS without controls & 35669.52 & 2412.02 \\ 
\midrule
\bottomrule
\end{tabular}
\end{footnotesize}
\end{table}

\newpage

 \begin{table}[H]
\centering \caption{Results based on original data} \begin{footnotesize} 
\begin{tabular}{lcc}\label{tab:results_testscores}
 &  &  \tabularnewline
 \toprule
\midrule 
Method & Point estimate  & Robust std.\ error  \tabularnewline
\midrule 
Post double Lasso ($\lambda_{\text{BCCH}}$) & -0.6770 & 0.0114 \\ 
Post double Lasso ($0.5\lambda_{\text{BCCH}}$) & -0.6762 & 0.0114 \\ 
Post double Lasso (1.5$\lambda_{\text{BCCH}}$) & -0.6778 & 0.0114 \\ 
OLS with all controls & -0.6694 & 0.0115 \\ 
OLS without controls  & -0.8538 & 0.0105 \\ 
\midrule 
\bottomrule
\end{tabular}
\end{footnotesize}
\end{table}

\newpage

\appendix 
\pagenumbering{arabic}

\begin{center}
{\LARGE Online Appendix to ``Omitted variable bias of Lasso-based inference methods: A finite sample analysis''}
\end{center}
\startcontents[sections]
\printcontents[sections]{l}{1}{\setcounter{tocdepth}{1}}

\section*{Notation}

Here we collect additional notation that is not provided in the main
text. The $\ell_{q}-$norm of a vector $v\in\mathbb{R}^{m}$ is denoted
by $\left|v\right|_{q}$, $1\leq q<\infty$, where $\left|v\right|_{q}:=\left(\sum_{i=1}^{m}\left|v_{i}\right|^{q}\right)^{1/q}$. %The support of $v$ is denoted by $\textrm{supp}(v):=\left\{ j:\,v_{j}\neq0\right\} $.
For a vector $v\in\mathbb{R}^{m}$, $|v|$ denotes the $m-$dimensional
vector $(|v_{1}|,\dots,|v_{m}|)^{T}$. For a matrix $A\in\mathbb{R}^{n\times m}$,
the $\ell_{2}-$operator norm of $A$ is defined as $\left\Vert A\right\Vert _{2}:=\sup_{v\in S^{m-1}}\left|Av\right|_{2}$,
where $S^{m-1}=\left\{ v\in\mathbb{R}^{m}:\,\left|v\right|_{2}=1\right\} $.
For a square matrix $A\in\mathbb{R}^{m\times m}$, let $\lambda_{\min}(A)$
and $\lambda_{\max}(A)$ denote its minimum eigenvalue and maximum
eigenvalue, respectively. %We denote $\max\left\{ a,\,b\right\} $
%by $a\vee b$ and $\min\left\{ a,\,b\right\} $ by $a\wedge b$.

%Let
%$1_{m}$ denote the $m-$dimensional (column) vector of ``1''s and
%$0_{m}$ is defined similarly. 

%The $\ell_{\infty}$ matrix norm (maximum absolute row sum) of a matrix
%$A$ is denoted by $\left\Vert A\right\Vert _{\infty}:=\max_{i}\sum_{j}\left|a_{ij}\right|$.
%For a vector $v\in \mathbb{R}^{m}$ and a set of indices $T\subseteq\left\{ 1,\dots,m\right\} $,
%let $v_{T}$ denote the sub-vector (with indices in $T$) of $v$.
%For a matrix $A\in \mathbb{R}^{n \times m}$, let $A_{T}$ denote the submatrix consisting of
%the columns with indices in $T$. 
%For functions $f(n)$ and $g(n)$, we write $f(n)\succsim g(n)$ to mean that $f(n)\geq cg(n)$ for a universal constant $c\in(0,\,\infty)$ and similarly, $f(n)\precsim g(n)$
%to mean that $f(n)\leq c^{'}g(n)$ for a universal constant $c^{'}\in(0,\,\infty)$;
%$f(n)\asymp g(n)$ when $f(n)\succsim g(n)$ and $f(n)\precsim g(n)$
%hold simultaneously. As a general rule, $c$ constants
%denote positive universal constants that are independent of $n$,
%$p$, $k$, $\sigma_{\eta}$, $\sigma_{v}$, $s$, and may change
%from place to place.

\section{Proofs for the main results}

\subsection{Lemma \ref{prop:fixed_design}}

\label{appendix a1}

\subsubsection*{Preliminary}

We will exploit the following Gaussian tail bound: 
\[
\mathbb{P}\left(\mathcal{Z}\geq t\right)\leq\frac{1}{2}\exp\left(\frac{-t^{2}}{2\sigma^{2}}\right)
\]
for all $t\geq0$, where $\mathcal{Z}\sim\mathcal{N}\left(0,\,\sigma^{2}\right)$.
Note that the constant ``$\frac{1}{2}$'' cannot be improved uniformly.

Given $\lambda\geq\frac{2\sigma}{\phi}\sqrt{\frac{2\left(1+\tau\right)\log p}{n}}$
where $\tau>0$ and the tail bound 
\[
\mathbb{P}\left(\left|\frac{X^{T}\varepsilon}{n}\right|_{\infty}\geq t\right)\leq\exp\left(\frac{-nt^{2}}{2\sigma^{2}}+\log p\right)\leq\frac{1}{p^{\tau}}
\]
for $t=\frac{\sigma}{\phi}\sqrt{\frac{2\left(1+\tau\right)\log p}{n}}$,
we have 
\begin{equation}
\lambda\geq2\left|\frac{X^{T}\varepsilon}{n}\right|_{\infty}\label{eq:lambda}
\end{equation}
with probability at least $1-\frac{1}{p^{\tau}}$. Let the event 
\begin{equation}
\mathcal{E}=\left\{ \left|\frac{X^{T}\varepsilon}{n}\right|_{\infty}\leq\frac{\sigma}{\phi}\sqrt{\frac{2\left(1+\tau\right)\log p}{n}}\right\} .\label{eq:event}
\end{equation}
Note that $\mathbb{P}\left(\mathcal{E}\right)\geq1-\frac{1}{p^{\tau}}$.

Lemma \ref{prop:fixed_design} relies on the following intermediate
results.

\textit{(i) On the event $\mathcal{E}$, (\ref{eq:las}) has a unique
optimal solution $\hat{\theta}$ such that $\hat{\theta}_{j}=0$ for
$j\notin K$.}

\textit{(ii) If $\mathbb{P}\left(\left\{ \hat{\theta}_{j}\neq0,\,j\in K\right\} \cap\mathcal{E}\right)>0$,
conditioning on $\left\{ \hat{\theta}_{j}\neq0,\,j\in K\right\} \cap\mathcal{E}$,
we must have 
\begin{equation}
\left|\hat{\theta}_{j}-\theta_{j}^{*}\right|\geq\frac{\lambda}{2}.\label{eq:37-1}
\end{equation}
}

Claim (i) above follows from the argument in \citet{wainwright_2019}.
To show claim (ii), we develop our own proof.

The proof for claim (i) above is based on a construction called Primal-Dual
Witness (PDW) method developed by \citet{wainwright2009sharp}. The
procedure is described as follows. 
\begin{enumerate}
\item Set $\hat{\theta}_{K^{c}}=0_{p-k}$. 
\item Obtain $(\hat{\theta}_{K},\,\hat{\delta}_{K})$ by solving 
\begin{equation}
\hat{\theta}_{K}\in\arg\min_{\theta_{K}\in\mathbb{R}^{k}}\left\{ \underset{:=g(\theta_{K})}{\underbrace{\frac{1}{2n}\left|Y-X_{K}\theta_{K}\right|_{2}^{2}}}+\lambda\left|\theta_{K}\right|_{1}\right\} ,\label{eq:sub}
\end{equation}
and choosing $\hat{\delta}_{K}\in\partial\left|\theta_{K}\right|_{1}$
such that $\nabla g(\theta_{K})\vert_{\theta_{K}=\hat{\theta}_{K}}+\lambda\hat{\delta}_{K}=0$.\footnote{For a convex function $f:\,\mathbb{R}^{p}\mapsto\mathbb{R}$, $\delta\in\mathbb{R}^{p}$
is a subgradient at $\theta$, namely $\delta\in\partial f(\theta)$,
if $f(\theta+\triangle)\geq f(\theta)+\left\langle \delta,\,\triangle\right\rangle $
for all $\triangle\in\mathbb{R}^{p}$. %		When $f(\theta)=\left|\theta\right|_{1}$,
 %		note that $\delta\in\partial\left|\theta\right|_{1}$ if and only
 %		if $\delta_{j}$ is $1$ when $\theta_{j}>0$, $-1$ when $\theta_{j}<0$, and any number in $[-1,\,1]$ when $\theta_{j}=0$. 
 %		$\delta_{j}=\textrm{sgn}(\theta_{j})$ for all $j=1,\dots,p$. Only here $\textrm{sgn}(0)$ is allowed to be any number in $[-1,\,1]$ but in other places we define $\textrm{sgn}(0)=0$.
 } 
\item Obtain $\hat{\delta}_{K^{c}}$ by solving 
\begin{equation}
\frac{1}{n}X^{T}(X\hat{\theta}-Y)+\lambda\hat{\delta}=0,\label{eq:42}
\end{equation}
and check whether or not $\left|\hat{\delta}_{K^{c}}\right|_{\infty}<1$
(the \textit{strict dual feasibility} condition) holds. 
\end{enumerate}
Lemma 7.23 from Chapter 7 of \citet{wainwright_2019} shows that,
if the PDW construction succeeds, then $\hat{\theta}=(\hat{\theta}_{K},\,0_{p-k})$
is the unique optimal solution of program (\ref{eq:las}). To show
that the PDW construction succeeds on the event $\mathcal{E}$, it
suffices to show that $\left|\hat{\delta}_{K^{c}}\right|_{\infty}<1$.
The details can be found in Chapter 7.5 of \citet{wainwright_2019}.
In particular, under the choice of $\lambda$ stated in Lemma \ref{prop:fixed_design},
we obtain that $\left|\hat{\delta}_{K^{c}}\right|_{\infty}<1$ and
hence the PDW construction succeeds conditioning on $\mathcal{E}$
where $\mathbb{P}\left(\mathcal{E}\right)\geq1-\frac{1}{p^{\tau}}$.

In summary, conditioning on $\mathcal{E}$, under the choice of $\lambda$
stated in Lemma \ref{prop:fixed_design}, program (\ref{eq:las})
has a unique optimal solution $\hat{\theta}$ such that $\hat{\theta}_{j}=0$
for $j\notin K$.

We now show (\ref{eq:37-1}). By construction, $\hat{\theta}=(\hat{\theta}_{K},\,0_{p-k})$,
$\hat{\delta}_{K}$, and $\hat{\delta}_{K^{c}}$ satisfy (\ref{eq:42})
and therefore we obtain 
\begin{eqnarray}
\frac{1}{n}X_{K}^{T}X_{K}\left(\hat{\theta}_{K}-\theta_{K}^{*}\right)-\frac{1}{n}X_{K}^{T}\varepsilon+\lambda\hat{\delta}_{K} & = & 0_{k},\label{eq:7-1}\\
\frac{1}{n}X_{K^{c}}^{T}X_{K}\left(\hat{\theta}_{K}-\theta_{K}^{*}\right)-\frac{1}{n}X_{K^{c}}^{T}\varepsilon+\lambda\hat{\delta}_{K^{c}} & = & 0_{p-k}.\label{eq:8}
\end{eqnarray}
Solving the equations above yields 
\begin{equation}
\hat{\theta}_{K}-\theta_{K}^{*}=\left(\frac{X_{K}^{T}X_{K}}{n}\right)^{-1}\frac{X_{K}^{T}\varepsilon}{n}-\lambda\left(\frac{X_{K}^{T}X_{K}}{n}\right)^{-1}\hat{\delta}_{K}.\label{eq:49}
\end{equation}

In what follows, we will condition on $\left\{ \hat{\theta}_{j}\neq0,\,j\in K\right\} \cap\mathcal{E}$
and make use of (\ref{eq:lambda})-(\ref{eq:event}). Let $\Delta=\frac{X_{K}^{T}\varepsilon}{n}-\lambda\hat{\delta}_{K}$.
Note that 
\begin{equation}
\left|\hat{\theta}_{K}-\theta_{K}^{*}\right|\geq\left|\left(\frac{X_{K}^{T}X_{K}}{n}\right)^{-1}\right|\left|\left|\lambda\hat{\delta}_{K}\right|-\left|\frac{X_{K}^{T}\varepsilon}{n}\right|\right|,\label{eq:7}
\end{equation}
where the inequality uses the fact that $\left(\frac{X_{K}^{T}X_{K}}{n}\right)^{-1}$
is diagonal. In Step 2 of the PDW procedure, $\hat{\delta}_{K}$ is
chosen such that $\left|\hat{\delta}_{j}\right|=1$ for any $j\in K$
with $\hat{\theta}_{j}\neq0$; we therefore obtain 
\[
\left|\hat{\theta}_{j}-\theta_{j}^{*}\right|\geq\left|\left|\lambda\right|-\left|\frac{X_{j}^{T}\varepsilon}{n}\right|\right|\geq\frac{\lambda}{2}
\]
where the second inequality follows from (\ref{eq:lambda}).

\subsubsection*{Main proof }

In what follows, we let 
\begin{eqnarray*}
E_{1} & = & \left\{ \textrm{sgn}\left(\hat{\theta}_{j}\right)=-\textrm{sgn}\left(\theta_{j}^{*}\right),\,\textrm{for some }j\in K\right\} ,\\
E_{2} & = & \left\{ \textrm{sgn}\left(\hat{\theta}_{j}\right)=\textrm{sgn}\left(\theta_{j}^{*}\right),\,\textrm{for some }j\in K\textrm{ such that (\ref{eq:min}) holds}\right\} ,\\
E_{3} & = & \left\{ \textrm{sgn}\left(\hat{\theta}_{j}\right)=\textrm{sgn}\left(\theta_{j}^{*}\right),\,\textrm{for some }j\in K\right\} .
\end{eqnarray*}
To show (\ref{eq:nec}) in (iv), recall we have established that conditioning
on $\mathcal{E}$, (\ref{eq:las}) has a unique optimal solution $\hat{\theta}$
such that $\hat{\theta}_{j}=0$ for $j\notin K$. Therefore, conditioning
on $\mathcal{E}$, the KKT condition for (\ref{eq:las}) implies 
\begin{equation}
\theta_{j}^{*}-\hat{\theta}_{j}=\lambda\textrm{sgn}\left(\hat{\theta}_{j}\right)-\frac{X_{j}^{T}\varepsilon}{n}\label{eq:kkt}
\end{equation}
for $j\in K$ such that $\hat{\theta}_{j}\neq0$.

We first show that $\mathbb{P}\left(E_{1}\cap\mathcal{E}\right)=0$.
Suppose $\mathbb{P}\left(E_{1}\cap\mathcal{E}\right)>0$. We may then
condition on the event $E_{1}\cap\mathcal{E}$. Case (i): $\theta_{j}^{*}>0$
and $\hat{\theta}_{j}<0$. Then, the LHS of (\ref{eq:kkt}), $\theta_{j}^{*}-\hat{\theta}_{j}>0$;
consequently, the RHS, $\lambda\textrm{sgn}\left(\hat{\theta}_{j}\right)-\frac{X_{j}^{T}\varepsilon}{n}=-\lambda-\frac{X_{j}^{T}\varepsilon}{n}>0$.
However, given the choice of $\lambda$, conditioning on $\mathcal{E}$,
$\lambda\geq2\left|\frac{X^{T}\varepsilon}{n}\right|_{\infty}$ and
consequently, $-\lambda-\frac{X_{j}^{T}\varepsilon}{n}\leq-\frac{\lambda}{2}<0$.
This leads to a contradiction. Case (ii): $\theta_{j}^{*}<0$ and
$\hat{\theta}_{j}>0$. Then, the LHS of (\ref{eq:kkt}), $\theta_{j}^{*}-\hat{\theta}_{j}<0$;
consequently, the RHS, $\lambda\textrm{sgn}\left(\hat{\theta}_{j}\right)-\frac{X_{j}^{T}\varepsilon}{n}=\lambda-\frac{X_{j}^{T}\varepsilon}{n}<0$.
However, given the choice of $\lambda$, conditioning on $\mathcal{E}$,
$\lambda\geq2\left|\frac{X^{T}\varepsilon}{n}\right|_{\infty}$ and
consequently, $\lambda-\frac{X_{j}^{T}\varepsilon}{n}\geq\frac{\lambda}{2}>0$.
This leads to a contradiction.

It remains to show that $\mathbb{P}\left(E_{2}\cap\mathcal{E}\right)=0$.
We first establish a useful fact under the assumption that $\mathbb{P}\left(E_{3}\cap\mathcal{E}\right)>0$.
Let us condition on the event $E_{3}\cap\mathcal{E}$. If $\theta_{j}^{*}>0$,
we have $\theta_{j}^{*}-\hat{\theta}_{j}=\lambda-\frac{X_{j}^{T}\varepsilon}{n}\geq\frac{\lambda}{2}>0$
(i.e., $\theta_{j}^{*}\geq\hat{\theta}_{j}$); similarly, if $\theta_{j}^{*}<0$,
then we have $\theta_{j}^{*}-\hat{\theta}_{j}=-\lambda-\frac{X_{j}^{T}\varepsilon}{n}\leq-\frac{\lambda}{2}<0$
(i.e., $\theta_{j}^{*}\leq\hat{\theta}_{j}$). Putting the pieces
together implies that, for $j\in K$ such that $\textrm{sgn}\left(\hat{\theta}_{j}\right)=\textrm{sgn}\left(\theta_{j}^{*}\right)$,
\begin{equation}
\left|\theta_{j}^{*}-\hat{\theta}_{j}\right|=\left|\theta_{j}^{*}\right|-\left|\hat{\theta}_{j}\right|.\label{eq:equal}
\end{equation}

We now show that $\mathbb{P}\left(E_{2}\cap\mathcal{E}\right)=0$.
Suppose $\mathbb{P}\left(E_{2}\cap\mathcal{E}\right)>0$. We may then
condition on the event that $E_{2}\cap\mathcal{E}$. Because of (\ref{eq:min})
and (\ref{eq:equal}), we have $\left|\theta_{j}^{*}-\hat{\theta}_{j}\right|<\frac{\lambda}{2}$.
On the other hand, (\ref{eq:37-1}) implies that $\left|\theta_{j}^{*}-\hat{\theta}_{j}\right|\geq\frac{\lambda}{2}$.
We have arrived at a contradiction. Consequently, we must have $\mathbb{P}\left(E_{2}\cap\mathcal{E}\right)=0$.

In summary, we have shown that $\mathbb{P}\left(E_{1}\cap\mathcal{E}\right)=0$
and $\mathbb{P}\left(E_{2}\cap\mathcal{E}\right)=0$. Claim (i) in
``Preliminary'' implies that $\mathbb{P}\left(E_{4}\vert\mathcal{E}\right)=0$
where $E_{4}$ denotes the event that $\hat{\theta}_{j}\neq0$ for
some $j\notin K$. Therefore, on $\mathcal{E}$, none of the events
$E_{1}$, $E_{2}$ and $E_{4}$ can happen. This fact implies that,
if (\ref{eq:min}) is satisfied for all $l\in K$, we must have 
\[
\mathbb{P}\left(\hat{\theta}=0_{p}\right)\geq1-\mathbb{P}\left(\mathcal{E}^{c}\right)\geq1-\frac{1}{p^{\tau}}.
\]

\subsection{Proposition \ref{prop:bias_post_double_formula}}

\label{sec:appendix_post_double_lasso}

We first show the case where $ab>0$. Let the events 
\begin{eqnarray}
\mathcal{E}_{t_{1}} & = & \left\{ \left|\frac{X_{K}^{T}v}{n}\right|_{\infty}\leq t_{1},\,t_{1}>0\right\} ,\label{eq:50}\\
\mathcal{E}_{t_{2}}^{'} & = & \left\{ \frac{1}{n}\sum_{i=1}^{n}v_{i}^{2}\leq\sigma_{v}^{2}+t_{2},\:t_{2}\in(0,\,\sigma_{v}^{2}]\right\} .\nonumber 
\end{eqnarray}
By tail bounds for Gaussian and Chi-Square variables, we have 
\begin{eqnarray}
\mathbb{P}\left(\mathcal{E}_{t_{1}}\right) & \geq & 1-k\exp\left(\frac{-nt_{1}^{2}}{2\sigma_{v}^{2}}\right),\label{eq:32}\\
\mathbb{P}\left(\mathcal{E}_{t_{2}}^{'}\right) & \geq & 1-\exp\left(\frac{-nt_{2}^{2}}{8\sigma_{v}^{4}}\right).\nonumber 
\end{eqnarray}

In the following proof, we exploit the bound 
\begin{eqnarray}
\mathbb{P}\left(\mathcal{E}_{t_{2}}^{'}\vert\hat{I}_{2}=\emptyset,\,\mathcal{E}_{t_{1}}\right) & \geq & \mathbb{P}\left(\mathcal{E}_{t_{2}}^{'}\cap\mathcal{E}_{t_{1}}\cap\left\{ \hat{I}_{2}=\emptyset\right\} \right)\nonumber \\
 & \geq & \mathbb{P}\left(\mathcal{E}_{t_{1}}\right)+\mathbb{P}\left(\mathcal{E}_{t_{2}}^{'}\right)+\mathbb{P}\left(\hat{I}_{2}=\emptyset\right)-2\nonumber \\
 & \geq & 1-\frac{1}{p^{\tau}}-k\exp\left(\frac{-nt_{1}^{2}}{2\sigma_{v}^{2}}\right)-\exp\left(\frac{-nt_{2}^{2}}{8\sigma_{v}^{4}}\right)\label{eq:58}
\end{eqnarray}
where the third inequality follows from Lemma \ref{prop:fixed_design},
which implies $\hat{I}_{2}=\emptyset$ with probability at least $1-\frac{1}{p^{\tau}}$.
Note that $\mathbb{P}\left(\mathcal{E}_{t_{1}}\cap\left\{ \hat{I}_{2}=\emptyset\right\} \right)\geq\mathbb{P}\left(\mathcal{E}_{t_{1}}\right)+\mathbb{P}\left(\hat{I}_{2}=\emptyset\right)-1\geq1-\frac{1}{p^{\tau}}-k\exp\left(\frac{-nt_{1}^{2}}{2\sigma_{v}^{2}}\right)$,
which is a ``high probability'' guarantee for sufficiently large
$p$ and $t_{1}$. Thus, working with $\mathbb{P}\left(\mathcal{E}_{t_{2}}^{'}\vert\hat{I}_{2}=\emptyset,\,\mathcal{E}_{t_{1}}\right)$
is sensible under an appropriate choice of $t_{1}$ (as we will see
below).

We first bound $\frac{\frac{1}{n}D^{T}X_{K}}{\frac{1}{n}D^{T}D}\beta_{K}^{*}$.
Note that 
\begin{eqnarray*}
\frac{\frac{1}{n}D^{T}X_{K}}{\frac{1}{n}D^{T}D}\beta_{K}^{*} & = & \left(\frac{D^{T}D}{n}\right)^{-1}\left[\frac{1}{n}\left(X_{K}\gamma_{K}^{*}+v\right)^{T}X_{K}\beta_{K}^{*}\right]\\
 & = & \left(\frac{D^{T}D}{n}\right)^{-1}\left[\frac{1}{n}\gamma_{K}^{*T}X_{K}^{T}X_{K}\beta_{K}^{*}+\frac{1}{n}v^{T}X_{K}\beta_{K}^{*}\right]\\
 & = & \frac{\gamma_{K}^{*T}\beta_{K}^{*}+\frac{1}{n}v^{T}X_{K}\beta_{K}^{*}}{\frac{1}{n}\left(X_{K}\gamma_{K}^{*}+v\right)^{T}\left(X_{K}\gamma_{K}^{*}+v\right)},
\end{eqnarray*}
and that $\gamma_{K}^{*T}\beta_{K}^{*}=2\left(1+\tau\right)ab\phi^{-2}\sigma_{\eta}\sigma_{v}\frac{k\log p}{n}$.
Moreover, applying (\ref{eq:58}) with $t_{1}=\left|b\right|\frac{\lambda_{2}}{4}$
and $t_{2}=r\sigma_{v}^{2}$ (where $r\in(0,\,1]$) yields

\begin{eqnarray}
\gamma_{K}^{*T}\beta_{K}^{*}+\frac{1}{n}v^{T}X_{K}\beta_{K}^{*} & \geq & \gamma_{K}^{*T}\beta_{K}^{*}-\left|\frac{1}{n}v^{T}X_{K}\right|_{\infty}\left|\beta_{K}^{*}\right|_{1}\nonumber \\
 & \geq & \left(1+\tau\right)ab\phi^{-2}\sigma_{\eta}\sigma_{v}\frac{k\log p}{n}\label{eq:pos}
\end{eqnarray}
as well as 
\begin{eqnarray}
\frac{1}{n}\left(X_{K}\gamma_{K}^{*}+v\right)^{T}\left(X_{K}\gamma_{K}^{*}+v\right) & \leq & \gamma_{K}^{*T}\gamma_{K}^{*}+\left|\frac{2}{n}v^{T}X_{K}\right|_{\infty}\left|\gamma_{K}^{*}\right|_{1}+\frac{1}{n}v^{T}v\nonumber \\
 & \leq & 4\left(1+\tau\right)\phi^{-2}b^{2}\sigma_{v}^{2}\frac{k\log p}{n}+\sigma_{v}^{2}+r\sigma_{v}^{2}\label{eq:den}
\end{eqnarray}
with probability at least 
\[
1-k\exp\left(\frac{-b^{2}\left(1+\tau\right)\log p}{4\phi^{2}}\right)-\frac{1}{p^{\tau}}-\exp\left(\frac{-nr^{2}}{8}\right):=T_{2}\left(r\right).
\]

Conditioning on $\mathcal{E}_{t_{1}}\cap\left\{ \hat{I}_{2}=\emptyset\right\} $
with $t_{1}=t^{*}=\left|b\right|\frac{\lambda_{2}}{4}$, putting the
pieces together yields 
\begin{equation}
\frac{D^{T}X_{K}}{D^{T}D}\beta_{K}^{*}\geq\frac{\left(1+\tau\right)ab\phi^{-2}\sigma_{\eta}\frac{k\log p}{n}}{4\left(1+\tau\right)\phi^{-2}b^{2}\sigma_{v}\frac{k\log p}{n}+\sigma_{v}+r\sigma_{v}}:=T_{1}\left(r\right),\label{eq:47}
\end{equation}
with probability at least $T_{2}\left(r\right)$. That is, 
\[
\mathbb{P}\left(\frac{D^{T}X_{K}}{D^{T}D}\beta_{K}^{*}\geq T_{1}\left(r\right)\vert\hat{I}_{2}=\emptyset,\,\mathcal{E}_{t^{*}}\right)\geq T_{2}\left(r\right).
\]

When $\alpha^{*}=0$ in (\ref{eq:20}), the reduced form coefficients
$\pi^{*}$ in (\ref{eq:reduce}) coincide with $\beta^{*}$ and $u$
coincides with $\eta$. Given the conditions on $X,$ $\eta$, $v$,
$\beta_{K}^{*}$ and $\gamma_{K}^{*}$, we can then apply (\ref{eq:nec})
in Lemma \ref{prop:fixed_design} and the fact $\mathbb{P}\left(\hat{I}_{1}=\hat{I}_{2}=\emptyset\right)\geq\mathbb{P}\left(\hat{I}_{1}=\emptyset\right)+\mathbb{P}\left(\hat{I}_{2}=\emptyset\right)-1$
to show that $E=\left\{ \hat{I}_{1}=\hat{I}_{2}=\emptyset\right\} $
occurs with probability at least $1-\frac{2}{p^{\tau}}$. Note that
with the choice $t_{1}=t^{*}=\left|b\right|\frac{\lambda_{2}}{4}$,
$\mathbb{P}\left(E\cap\mathcal{E}_{t^{*}}\right)\geq\mathbb{P}\left(E\right)+\mathbb{P}\left(\mathcal{E}_{t^{*}}\right)-1\geq1-k\exp\left(\frac{-b^{2}\left(1+\tau\right)\log p}{4\phi^{2}}\right)-\frac{2}{p^{\tau}}$,
which is a ``high probability'' guarantee given sufficiently large
$p$.\footnote{Because $v$ and $\eta$ are independent of each other, the bound
$\mathbb{P}\left(E\cap\mathcal{E}_{t^{*}}\right)\geq1-k\exp\left(\frac{-b^{2}\left(1+\tau\right)\log p}{4\phi^{2}}\right)-\frac{2}{p^{\tau}}$
can be further sharpened to $\mathbb{P}\left(E\cap\mathcal{E}_{t^{*}}\right)\geq\left(1-\frac{1}{p^{\tau}}\right)^{2}-k\exp\left(\frac{-b^{2}\left(1+\tau\right)\log p}{4\phi^{2}}\right)$.} Therefore, it is sensible to work with $\mathbb{E}\left(\tilde{\alpha}-\alpha^{*}\vert\mathcal{M}\right)$
where 
\begin{equation}
\mathcal{M}=E\cap\mathcal{E}_{t^{*}}.\label{eq:M_event}
\end{equation}

Given $E$, (\ref{eq:double}) becomes 
\begin{equation}
\tilde{\alpha}\in\textrm{arg}\min_{\alpha\in\mathbb{R}}\frac{1}{2n}\left|Y-D\alpha\right|_{2}^{2},\qquad\textrm{while }\tilde{\beta}=0_{p}.\label{eq:double-1}
\end{equation}
As a result, we obtain $\mathbb{E}\left(\tilde{\alpha}-\alpha^{*}\vert\mathcal{M}\right)=\mathbb{E}\left(\frac{\frac{1}{n}D^{T}X_{K}}{\frac{1}{n}D^{T}D}\beta_{K}^{*}\vert\mathcal{M}\right)+\mathbb{E}\left(\frac{\frac{1}{n}D^{T}\eta}{\frac{1}{n}D^{T}D}\vert\mathcal{M}\right)$
and 
\begin{eqnarray}
\mathbb{E}\left(\frac{\frac{1}{n}D^{T}\eta}{\frac{1}{n}D^{T}D}\vert\mathcal{M}\right) & = & \frac{1}{\mathbb{P}\left(\mathcal{M}\right)}\mathbb{E}\left[\frac{\frac{1}{n}D^{T}\eta}{\frac{1}{n}D^{T}D}1_{\mathcal{M}}\left(D,\eta\right)\right]\nonumber \\
 & = & \frac{1}{\mathbb{P}\left(\mathcal{M}\right)}\mathbb{E}_{D}\left\{ \mathbb{E}_{\eta}\left[\frac{\frac{1}{n}D^{T}\eta}{\frac{1}{n}D^{T}D}1_{\mathcal{M}}\left(D,\eta\right)\vert D\right]\right\} \nonumber \\
 & = & \frac{1}{\mathbb{P}\left(\mathcal{M}\right)}\mathbb{E}_{D}\left\{ \frac{\frac{1}{n}\sum_{i=1}^{n}D_{i}\mathbb{E}_{\eta}\left[\eta_{i}1_{\mathcal{M}}\left(D,\eta\right)\vert D\right]}{\frac{1}{n}D^{T}D}\right\} \nonumber \\
 & = & 0\label{eq:60}
\end{eqnarray}
where $1_{\mathcal{M}}\left(D,\eta\right)=1\left\{ \left(v,\eta\right):\,\hat{I}_{1}=\hat{I}_{2}=\emptyset,\,\left|\frac{X_{K}^{T}v}{n}\right|_{\infty}\leq t^{*}\right\} $
(recall $X$ is a fixed design); the last line follows from $\frac{1}{n}\sum_{i=1}^{n}D_{i}=0$,
the distributional identicalness of $\left(\eta_{i}\right)_{i=1}^{n}$
and that $\mathbb{E}_{\eta}\left[\eta_{i}1_{\mathcal{M}}\left(D,\eta\right)\vert D\right]$
is a constant over $i$s.

It remains to bound $\mathbb{E}\left(\frac{\frac{1}{n}D^{T}X_{K}}{\frac{1}{n}D^{T}D}\beta_{K}^{*}\vert\mathcal{M}\right)=\mathbb{E}\left(\frac{\frac{1}{n}D^{T}X_{K}}{\frac{1}{n}D^{T}D}\beta_{K}^{*}\vert\hat{I}_{2}=\emptyset,\,\mathcal{E}_{t^{*}}\right)$.
Note that conditioning on $\mathcal{E}_{t^{*}}$, $\frac{\frac{1}{n}D^{T}X_{K}}{\frac{1}{n}D^{T}D}\beta_{K}^{*}$
is positive by (\ref{eq:pos}). Applying a Markov inequality yields
\[
\mathbb{E}\left(\frac{\frac{1}{n}D^{T}X_{K}}{\frac{1}{n}D^{T}D}\beta_{K}^{*}\vert\hat{I}_{2}=\emptyset,\,\mathcal{E}_{t^{*}}\right)\geq T_{1}\left(r\right)\mathbb{P}\left(\frac{D^{T}X_{K}}{D^{T}D}\beta_{K}^{*}\geq T_{1}\left(r\right)\vert\hat{I}_{2}=\emptyset,\,\mathcal{E}_{t^{*}}\right)\geq T_{1}\left(r\right)T_{2}\left(r\right).
\]
Combining the result above with (\ref{eq:60}) and maximizing over
$r\in(0,\,1]$ gives the claim.

We now show the case where $ab<0$. The argument is almost similar.
In particular, we use 
\begin{eqnarray*}
\gamma_{K}^{*T}\beta_{K}^{*}+\frac{1}{n}v^{T}X_{K}\beta_{K}^{*} & \leq & \gamma_{K}^{*T}\beta_{K}^{*}+\left|\frac{1}{n}v^{T}X_{K}\right|_{\infty}\left|\beta_{K}^{*}\right|_{1}\\
 & \leq & -\left(1+\tau\right)\left|ab\right|\phi^{-2}\sigma_{\eta}\sigma_{v}\frac{k\log p}{n}<0
\end{eqnarray*}
and replace (\ref{eq:pos}) with 
\[
-\gamma_{K}^{*T}\beta_{K}^{*}-\frac{1}{n}v^{T}X_{K}\beta_{K}^{*}\geq\left(1+\tau\right)\left|ab\right|\phi^{-2}\sigma_{\eta}\sigma_{v}\frac{k\log p}{n}.
\]
Note that 
\[
\mathbb{E}\left(\alpha^{*}-\tilde{\alpha}\vert\mathcal{M}\right)=\mathbb{E}\left(\frac{-\frac{1}{n}D^{T}X_{K}}{\frac{1}{n}D^{T}D}\beta_{K}^{*}\vert\mathcal{M}\right)-\mathbb{E}\left(\frac{\frac{1}{n}D^{T}\eta}{\frac{1}{n}D^{T}D}\vert\mathcal{M}\right).
\]
So the rest of the proof follows from the argument for the case where
$ab>0$.

\subsection{Lower bounds on the OVBs of post double Lasso when $\alpha^{*}\protect\neq0$\label{subsec:Lower-bounds-alpha not zero}}

For functions $f(n)$ and $g(n)$, we write $f(n)\succsim g(n)$ to
mean that $f(n)\geq cg(n)$ for a universal constant $c\in(0,\,\infty)$
and similarly, $f(n)\precsim g(n)$ to mean that $f(n)\leq c^{'}g(n)$
for a universal constant $c^{'}\in(0,\,\infty)$; $f(n)\asymp g(n)$
when $f(n)\succsim g(n)$ and $f(n)\precsim g(n)$ hold simultaneously.
As a general rule, $c$ constants denote positive universal constants
that are independent of $n$, $p$, $k$, $\sigma_{\eta}$, $\sigma_{v}$,
$\alpha^{*}$ and may change from place to place. 

\begin{proposition}[Scaling of OVB lower bound, Case I]\label{prop:bias_post_double_selection}
Let Assumption \ref{ass: incoherence}(ii)-(iii) and Assumption \ref{ass:inference_normality}
hold. Suppose $\phi^{-1}\precsim1$; the regularization parameters
in (\ref{eq:las-1}) and (\ref{eq:las-2}) are chosen in a similar
fashion as in Lemma \ref{prop:fixed_design} such that $\lambda_{1}\asymp\phi^{-1}\sigma_{\eta}\sqrt{\frac{\log p}{n}}$\,\footnote{In general, $\lambda_{1}\asymp\phi^{-1}\left(\sigma_{\eta}+\left|\alpha^{*}\right|\sigma_{v}\right)\sqrt{\frac{\log p}{n}}$
and the iterative algorithm for choosing $\lambda_{1}$ in \citet{belloni2014inference} achieves this scaling.
Under the conditions on $\left(\beta_{K}^{*},\,\gamma_{K}^{*},\,\alpha^{*}\right)$
in Proposition \ref{prop:bias_post_double_selection}, this scaling
is equivalent to $\phi^{-1}\sigma_{\eta}\sqrt{\frac{\log p}{n}}$.} and $\lambda_{2}\asymp\phi^{-1}\sigma_{v}\sqrt{\frac{\log p}{n}}$;
for all $j\in K$, $\left|\beta_{j}^{*}\right|\leq\frac{\lambda_{1}}{2}$
and $\left|\gamma_{j}^{*}\right|\leq\frac{\lambda_{2}}{2}$, but $\left|\beta_{j}^{*}\right|\asymp\sigma_{\eta}\sqrt{\frac{\log p}{n}}$
and $\left|\gamma_{j}^{*}\right|\asymp\sigma_{v}\sqrt{\frac{\log p}{n}}$.
Let us consider $\tilde{\alpha}$ obtained from (\ref{eq:double}).
If $\alpha^{*}\gamma_{j}^{*}\in(0,\,-\beta_{j}^{*}]$, $\beta_{j}^{*}<0$
for $j\in K$ (or, $\alpha^{*}\gamma_{j}^{*}\in[-\beta_{j}^{*},\,0)$,
$\beta_{j}^{*}>0$ for $j\in K$), then for some positive universal
constants $c^{\dagger},c_{1},c_{2},c_{3},c^{*},c_{0}^{*}$, 
\begin{equation}
\left|\mathbb{E}\left(\tilde{\alpha}-\alpha^{*}\vert\mathcal{M}\right)\right|\geq c^{\dagger}\frac{\sigma_{\eta}}{\sigma_{v}}\left(\frac{k\log p}{n}\wedge1\right)\left[1-c_{1}k\exp\left(-c_{2}\log p\right)-\exp\left(-c_{3}n\right)\right],\label{eq:22}
\end{equation}
where $\mathcal{M}$ is an event with $\mathbb{P}\left(\mathcal{M}\right)\geq1-c^{*}k\exp\left(-c_{0}^{*}\log p\right)$.

\end{proposition}

\begin{proposition}[Scaling of OVB lower bound, Case II] \label{prop:bias_post_double_selection-1}
Let Assumption \ref{ass: incoherence}(ii)-(iii) and Assumption \ref{ass:inference_normality}
hold. Suppose $\phi^{-1}\precsim1$; the regularization parameters
in (\ref{eq:las-1}) and (\ref{eq:las-2}) are chosen in a similar
fashion as in Lemma \ref{prop:fixed_design} such that $\lambda_{1}\asymp\frac{\left(\sigma_{\eta}+\left|\alpha^{*}\right|\sigma_{v}\right)}{\phi}\sqrt{\frac{\log p}{n}}$
and $\lambda_{2}\asymp\phi^{-1}\sigma_{v}\sqrt{\frac{\log p}{n}}$;
for all $j\in K$, $\left|\gamma_{j}^{*}\right|\leq\frac{\lambda_{2}}{2}$
but $\left|\gamma_{j}^{*}\right|\asymp\sigma_{v}\sqrt{\frac{\log p}{n}}$.
Let us consider $\tilde{\alpha}$ obtained from (\ref{eq:double}).

(i) If $\pi_{j}^{*}=0$ for all $j\in K$, then there exist positive
universal constants $c^{\dagger},c_{1},c_{2},c_{3},c^{*},c_{0}^{*}$
such that 
\begin{equation}
\left|\mathbb{E}\left(\tilde{\alpha}-\alpha^{*}\vert\mathcal{M}\right)\right|\geq c^{\dagger}\left|\alpha^{*}\right|\left(\frac{k\log p}{n}\wedge1\right)\left[1-c_{1}k\exp\left(-c_{2}\log p\right)-\exp\left(-c_{3}n\right)\right],\label{eq:23}
\end{equation}
where $\mathbb{P}\left(\mathcal{M}\right)\geq1-c^{*}k\exp\left(-c_{0}^{*}\log p\right)$.

(ii) For all $j\in K$, suppose $\left|\pi_{j}^{*}\right|\asymp\left(\sigma_{\eta}+\left|\alpha^{*}\right|\sigma_{v}\right)\sqrt{\frac{\log p}{n}}$,
and we have either (1) $\alpha^{*}<0$, $\beta_{j}^{*}>0$, $\gamma_{j}^{*}>0$,
$0<\pi_{j}^{*}\leq\frac{\lambda_{1}}{2}$, or (2) $\alpha^{*}>0$,
$\beta_{j}^{*}<0$, $\gamma_{j}^{*}>0$, $-\frac{\lambda_{1}}{2}<\pi_{j}^{*}<0$.
Then there exist positive universal constants $c^{\dagger},c_{1},c_{2},c_{3},c^{*},c_{0}^{*}$
such that (\ref{eq:23}) holds with $\mathbb{P}\left(\mathcal{M}\right)\geq1-c^{*}k\exp\left(-c_{0}^{*}\log p\right)$.

\end{proposition}

\begin{remark}As the error in the reduced form equation \eqref{eq:reduce}
involves $\alpha^{*}$, the choice of $\lambda_{1}$ in \eqref{eq:las-1}
depends on the unknown $\alpha^{*}$. Therefore, the scaling of the
OVB lower bound can involve $\left|\alpha^{*}\right|$ when the relevant
controls are not selected, as suggested by Proposition \eqref{prop:bias_post_double_selection-1}.
The dependence on $\left|\alpha^{*}\right|$ in Proposition \eqref{prop:bias_post_double_selection}
has been suppressed because $\left|\alpha^{*}\right|\leq\max_{j\in K}\frac{\left|\beta_{j}^{\ast}\right|}{\left|\gamma_{j}^{\ast}\right|}\asymp\frac{\sigma_{\eta}}{\sigma_{v}}$.
\end{remark}

\subsection{Proof for Proposition \ref{prop:bias_post_double_selection}}

\label{appendix:bias_post_double_lower_scaling}

Part (i) of Proposition \ref{prop:bias_post_double_selection} follows
immediately from the proof for Proposition \ref{prop:bias_post_double_formula}.
It remains to establish part (ii) where $\alpha^{*}\neq0$, $\alpha^{*}\gamma_{j}^{*}\in(0,\,-\beta_{j}^{*}]$,
$\beta_{j}^{*}<0$ for all $j\in K$ (or, $\alpha^{*}\gamma_{j}^{*}\in[-\beta_{j}^{*},\,0)$,
$\beta_{j}^{*}>0$ for all $j\in K$). Because of these conditions,
we have 
\[
\left|\pi_{j}^{*}\right|=\left|\beta_{j}^{*}+\alpha^{*}\gamma_{j}^{*}\right|<\left|\beta_{j}^{*}\right|\quad\forall j\in K.
\]
Note that $\left|\alpha^{*}\right|\leq\max_{j\in K}\frac{\left|\beta_{j}^{\ast}\right|}{\left|\gamma_{j}^{\ast}\right|}\asymp\frac{\sigma_{\eta}}{\sigma_{v}}$
and 
\begin{eqnarray*}
\left|\frac{X^{T}u}{n}\right|_{\infty} & = & \left|\frac{X^{T}\left(\eta+\alpha^{*}v\right)}{n}\right|_{\infty}\\
 & \leq & \left|\frac{X^{T}\eta}{n}\right|_{\infty}+\left|\frac{\alpha^{*}X^{T}v}{n}\right|_{\infty}\\
 & \precsim & \sigma_{\eta}\sqrt{\frac{\log p}{n}}+\frac{\sigma_{\eta}}{\sigma_{v}}\sigma_{v}\sqrt{\frac{\log p}{n}}\\
 & \precsim & \phi^{-1}\sigma_{\eta}\sqrt{\frac{\log p}{n}}
\end{eqnarray*}
with probability at least $1-c_{1}^{'}\exp\left(-c_{2}^{'}\log p\right)$.
The fact above justifies the choice of $\lambda_{1}$ stated in \ref{prop:bias_post_double_selection}.
We can then apply (\ref{eq:nec}) in Lemma \ref{prop:fixed_design}
to show that $\hat{I}_{1}=\emptyset$ with probability at least $1-c_{5}\exp\left(-c_{6}\log p\right)$.
Furthermore, under the conditions on $X$ and $\gamma_{K}^{*}$, (\ref{eq:nec})
in Lemma \ref{prop:fixed_design} implies that $\hat{I}_{2}=\emptyset$
with probability at least $1-c_{0}\exp\left(-c_{0}^{'}\log p\right)$.
Therefore, we have 
\[
\mathbb{P}\left(\hat{I}_{1}=\hat{I}_{2}=\emptyset\right)\geq\mathbb{P}\left(\hat{I}_{1}=\emptyset\right)+\mathbb{P}\left(\hat{I}_{2}=\emptyset\right)-1\geq1-c_{1}^{''}\exp\left(-c_{2}^{''}\log p\right).
\]

Given $u=\eta+\alpha^{*}v$, when $\alpha^{*}\neq0$, the event $\left\{ \hat{I}_{1}=\emptyset\right\} $
is not independent of $D$, so $\mathbb{E}\left(\frac{\frac{1}{n}D^{T}X_{K}}{\frac{1}{n}D^{T}D}\beta_{K}^{*}\vert E,\,\mathcal{E}_{t^{*}}\right)\neq\mathbb{E}\left(\frac{\frac{1}{n}D^{T}X_{K}}{\frac{1}{n}D^{T}D}\beta_{K}^{*}\vert\hat{I}_{2}=\emptyset,\,\mathcal{E}_{t^{*}}\right)$
(recalling $E=\left\{ \hat{I}_{1}=\hat{I}_{2}=\emptyset\right\} $).
Instead of (\ref{eq:58}), we apply

\begin{eqnarray*}
\mathbb{P}\left(\mathcal{E}_{t_{2}}^{'}\vert E,\,\mathcal{E}_{t_{1}}\right) & \geq & \mathbb{P}\left(\mathcal{E}_{t_{2}}^{'}\cap\mathcal{E}_{t_{1}}\cap E\right)\\
 & \geq & \mathbb{P}\left(\mathcal{E}_{t_{1}}\right)+\mathbb{P}\left(\mathcal{E}_{t_{2}}^{'}\right)+\mathbb{P}\left(E\right)-2\\
 & \geq & 1-c_{1}^{''}\exp\left(-c_{2}^{''}\log p\right)-k\exp\left(\frac{-nt_{1}^{2}}{2\sigma_{v}^{2}}\right)\\
 &  & -\exp\left(\frac{-nt_{2}^{2}}{8\sigma_{v}^{4}}\right),\qquad\textrm{for any }t_{2}\in(0,\,\sigma_{v}^{2}].
\end{eqnarray*}
The rest of the proof follows from the argument for Proposition \ref{prop:bias_post_double_formula}
and the bounds above.

\subsection{Proof for Proposition \ref{prop:bias_post_double_selection-1}}

Note that we have 
\begin{eqnarray*}
\left|\frac{X^{T}u}{n}\right|_{\infty} & = & \left|\frac{X^{T}\left(\eta+\alpha^{*}v\right)}{n}\right|_{\infty}\\
 & \leq & \left|\frac{X^{T}\eta}{n}\right|_{\infty}+\left|\frac{\alpha^{*}X^{T}v}{n}\right|_{\infty}\\
 & \precsim & \sigma_{\eta}\sqrt{\frac{\log p}{n}}+\left|\alpha^{*}\right|\sigma_{v}\sqrt{\frac{\log p}{n}},
\end{eqnarray*}
which justifies the choice of $\lambda_{1}$ stated in Proposition
\ref{prop:bias_post_double_selection-1}.

For part (i), recall that $\pi_{K}^{*}=0_{k}$. By (i) of the intermediate
results in ``Preliminary'' of Section \ref{appendix a1}, $\mathbb{P}\left(\hat{\pi}=0_{p}\right)=\mathbb{P}\left(\hat{I}_{1}=\emptyset\right)\geq1-c\exp\left(-c^{'}\log p\right)$.
Substituting $\beta_{K}^{*}=-\alpha^{*}\gamma_{K}^{*}$ in $\frac{\frac{1}{n}D^{T}X_{K}}{\frac{1}{n}D^{T}D}\beta_{K}^{*}$
and following the rest of proof for Proposition \ref{prop:bias_post_double_selection}
yields the claim.

For part (ii), for all $j\in K$, note that $\beta_{j}^{*}>-\alpha^{*}\gamma_{j}^{*}>0$
in case (1) and $\beta_{j}^{*}<-\alpha^{*}\gamma_{j}^{*}<0$ in case
(2). Under the conditions on $X$ and $\pi_{K}^{*}$, (\ref{eq:nec})
in Lemma \ref{prop:fixed_design} implies that $\mathbb{P}\left(\hat{\pi}=0_{p}\right)=\mathbb{P}\left(\hat{I}_{1}=\emptyset\right)\geq1-c\exp\left(-c^{'}\log p\right)$.
Substituting $\beta_{K}^{*}>-\alpha^{*}\gamma_{K}^{*}>0$ for case
(1) and $\beta_{j}^{*}<-\alpha^{*}\gamma_{j}^{*}<0$ for case (2)
in the derivation of the bounds for $\frac{1}{n}D^{T}X_{K}\beta_{K}^{*}$
and following the rest of proof for Proposition \ref{prop:bias_post_double_selection}
yields the claim.

\subsection{Upper bounds on the OVBs of post double Lasso\label{subsec:Upper-bounds-OVB}}

\begin{proposition}[Scaling of OVB upper bound, Case I] \label{prop:bias_post_double_selection-upper}Let
Assumption \ref{ass: incoherence}(ii)-(iii) and Assumption \ref{ass:inference_normality}
hold. Suppose $\phi^{-1}\precsim1$; the regularization parameters
in (\ref{eq:las-1}) and (\ref{eq:las-2}) are chosen in a similar
fashion as in Proposition \ref{prop:bias_post_double_selection} such
that $\lambda_{1}\asymp\phi^{-1}\sigma_{\eta}\sqrt{\frac{\log p}{n}}$
and $\lambda_{2}\asymp\phi^{-1}\sigma_{v}\sqrt{\frac{\log p}{n}}$;
for all $j\in K$, $\gamma_{j}^{*}=\gamma^{*}$, $\left|\beta_{j}^{*}\right|\leq\frac{\lambda_{1}}{2}$
and $\left|\gamma^{*}\right|\leq\frac{\lambda_{2}}{2}$, but $\left|\beta_{j}^{*}\right|\asymp\sigma_{\eta}\sqrt{\frac{\log p}{n}}$
and $\left|\gamma^{*}\right|\asymp\sigma_{v}\sqrt{\frac{\log p}{n}}$.
Let us consider $\tilde{\alpha}$ obtained from (\ref{eq:double}).
Then for either $\alpha^{*}=0$, or $\alpha^{*}\neq0$ subject to
the conditions in part (ii) of Proposition \ref{prop:bias_post_double_selection},
there exist positive universal constants $c_{1},c_{2},c_{3},c_{4},c^{*},c_{0}^{*}$
such that 
\[
\mathbb{P}\left(\left|\tilde{\alpha}-\alpha^{*}\right|\leq\overline{OVB}\vert\hat{I}_{1}=\hat{I}_{2}=\emptyset\right)\geq1-c_{1}k\exp\left(-c_{2}\log p\right)-c_{3}\exp\left(-c_{4}nr^{2}\right)
\]
where $\mathbb{P}\left(\hat{I}_{1}=\hat{I}_{2}=\emptyset\right)\geq1-c^{*}\exp\left(-c_{0}^{*}\log p\right)$
and 
\[
\overline{OVB}\asymp\max\left\{ \frac{\sigma_{\eta}}{\sigma_{v}}\left(\frac{k\log p}{n}\wedge1\right),\,\frac{\sigma_{v}\sigma_{\eta}\left(r\vee\frac{k\log p}{n}\right)}{\left(\frac{k\log p}{n}\vee1\right)\sigma_{v}^{2}}\right\} 
\]
for any $r\in(0,\,1]$.

\end{proposition}

\begin{proposition}[Scaling of OVB upper bound, Case II] \label{prop:bias_post_double_selection-upper-1}Let
Assumption \ref{ass: incoherence}(ii)-(iii) and Assumption \ref{ass:inference_normality}
hold. Suppose $\phi^{-1}\precsim1$; the regularization parameters
in (\ref{eq:las-1}) and (\ref{eq:las-2}) are chosen in a similar
fashion as in Proposition \ref{prop:bias_post_double_selection-1}
such that $\lambda_{1}\asymp\frac{\sigma_{\eta}+\left|\alpha^{*}\right|\sigma_{v}}{\phi}\sqrt{\frac{\log p}{n}}$
and $\lambda_{2}\asymp\phi^{-1}\sigma_{v}\sqrt{\frac{\log p}{n}}$;
for all $j\in K$, $\gamma_{j}^{*}=\gamma^{*}$, $\left|\gamma^{*}\right|\leq\frac{\lambda_{2}}{2}$,
but $\left|\gamma^{*}\right|\asymp\sigma_{v}\sqrt{\frac{\log p}{n}}$.
Let us consider $\tilde{\alpha}$ obtained from (\ref{eq:double}).

(i) If $\pi_{j}^{*}=0$ for all $j\in K$, then there exist positive
universal constants $c_{1},c_{2},c_{3},c_{4},c^{*},c_{0}^{*}$ such
that 
\begin{equation}
\mathbb{P}\left(\left|\tilde{\alpha}-\alpha^{*}\right|\leq\overline{OVB}\vert\hat{I}_{1}=\hat{I}_{2}=\emptyset\right)\geq1-c_{1}k\exp\left(-c_{2}\log p\right)-c_{3}\exp\left(-c_{4}nr^{2}\right)\label{eq:OVB_upper}
\end{equation}
where $\mathbb{P}\left(\hat{I}_{1}=\hat{I}_{2}=\emptyset\right)\geq1-c^{*}\exp\left(-c_{0}^{*}\log p\right)$
and 
\[
\overline{OVB}\asymp\max\left\{ \left|\alpha^{*}\right|\left(\frac{k\log p}{n}\wedge1\right),\,\frac{\sigma_{v}\sigma_{\eta}\left(r\vee\frac{k\log p}{n}\right)}{\left(\frac{k\log p}{n}\vee1\right)\sigma_{v}^{2}}\right\} 
\]
for any $r\in(0,\,1]$.

(ii) If $0<\left|\pi_{j}^{*}\right|\leq\frac{\lambda_{1}}{2}$ but
$\left|\pi_{j}^{*}\right|\asymp\left(\sigma_{\eta}+\left|\alpha^{*}\right|\sigma_{v}\right)\sqrt{\frac{\log p}{n}}$
for all $j\in K$, then there exist positive universal constants $c_{1},c_{2},c_{3},c_{4},c^{*},c_{0}^{*}$
such that (\ref{eq:OVB_upper}) holds with $\mathbb{P}\left(\hat{I}_{1}=\hat{I}_{2}=\emptyset\right)\geq1-c^{*}\exp\left(-c_{0}^{*}\log p\right)$
and 
\[
\overline{OVB}\asymp\max\left\{ \left(\left|\alpha^{*}\right|\vee\frac{\sigma_{\eta}}{\sigma_{v}}\right)\left(\frac{k\log p}{n}\wedge1\right),\,\frac{\sigma_{v}\sigma_{\eta}\left(r\vee\frac{k\log p}{n}\right)}{\left(\frac{k\log p}{n}\vee1\right)\sigma_{v}^{2}}\right\} 
\]
for any $r\in(0,\,1]$. \end{proposition}

\begin{remark}Suppose $\sigma_{v}\asymp1$ and $\frac{c^{'}k}{p^{c^{''}}}$
is small for some positive universal constants $c^{'}$ and $c^{''}$.
As $\frac{k\log p}{n}\rightarrow\infty$ (where both Lasso steps are
inconsistent in the sense that $\sqrt{\frac{1}{n}\sum_{i=1}^{n}\left(X_{i}\hat{\pi}-X_{i}\pi^{*}\right)^{2}}\rightarrow\infty$
and $\sqrt{\frac{1}{n}\sum_{i=1}^{n}\left(X_{i}\hat{\gamma}-X_{i}\gamma^{*}\right)^{2}}\rightarrow\infty$
with high probability), Proposition \ref{prop:bias_post_double_selection-upper}
implies that $\overline{OVB}\asymp\frac{\sigma_{\eta}}{\sigma_{v}}$
and Proposition \ref{prop:bias_post_double_selection-upper-1} implies
that $\overline{OVB}\asymp\left(\left|\alpha^{*}\right|\vee\frac{\sigma_{\eta}}{\sigma_{v}}\right)$,
and $\mathbb{P}\left(\left|\tilde{\alpha}-\alpha^{*}\right|\leq\overline{OVB}\vert\hat{I}_{1}=\hat{I}_{2}=\emptyset\right)$
is large.

\end{remark}

\subsection{Proof for Proposition \ref{prop:bias_post_double_selection-upper}
\label{appendix:Proposition-OVB-upper-proof}}

Given $\left\{ \hat{I}_{1}=\hat{I}_{2}=\emptyset\right\} $, note
that $\left|\tilde{\alpha}-\alpha^{*}\right|\leq\left|\frac{\frac{1}{n}D^{T}X_{K}}{\frac{1}{n}D^{T}D}\beta_{K}^{*}\right|+\left|\frac{\frac{1}{n}D^{T}\eta}{\frac{1}{n}D^{T}D}\right|$.
We make use of the following bound on Chi-Square variables: 
\begin{equation}
\mathbb{P}\left[\frac{1}{n}\sum_{i=1}^{n}v_{i}^{2}-\mathbb{E}\left(\frac{1}{n}\sum_{i=1}^{n}v_{i}^{2}\right)\leq-\sigma_{v}^{2}r^{'}\right]\leq\exp\left(\frac{-nr^{'2}}{16}\right)\label{eq:41}
\end{equation}
for all $r^{'}\geq0$. On the event $\left\{ \hat{I}_{1}=\hat{I}_{2}=\emptyset\right\} $,
choosing $t_{1}=t^{*}=\frac{\left|\gamma^{*}\right|}{4}$ in (\ref{eq:50})
and $r^{'}=\frac{1}{2}$ in (\ref{eq:41}) yields 
\begin{eqnarray*}
\left|\frac{\frac{1}{n}D^{T}X_{K}}{\frac{1}{n}D^{T}D}\beta_{K}^{*}\right| & \leq & \frac{\left|\gamma_{K}^{*T}\beta_{K}^{*}\right|+\left|\beta_{K}^{*}\right|_{1}t^{*}}{\gamma_{K}^{*T}\gamma_{K}^{*}-2\left|\gamma_{K}^{*}\right|_{1}t^{*}+\frac{1}{2}\sigma_{v}^{2}}\\
 & \leq & \frac{c_{3}\sigma_{\eta}\sigma_{v}\frac{k\log p}{n}}{c_{1}\frac{k\log p}{n}\sigma_{v}^{2}+c_{2}\sigma_{v}^{2}}\\
 & \leq & c_{4}\frac{\sigma_{\eta}}{\sigma_{v}}\left(\frac{k\log p}{n}\wedge1\right)
\end{eqnarray*}
with probability at least $1-c_{5}k\exp\left(-c_{6}\log p\right)-\exp\left(\frac{-n}{64}\right)$.

We can also show that 
\begin{eqnarray*}
 &  & \mathbb{P}\left(\left|\frac{1}{n}D^{T}\eta\right|\leq t\vert\hat{I}_{1}=\hat{I}_{2}=\emptyset\right)\\
 & \geq & \mathbb{P}\left(\left\{ \left(\eta,v\right):\,\left|\frac{1}{n}D^{T}\eta\right|\leq t\right\} \cap\left\{ \hat{I}_{1}=\hat{I}_{2}=\emptyset\right\} \right)\\
 & \geq & \mathbb{P}\left(\left|\frac{1}{n}D^{T}\eta\right|\leq t\right)+\mathbb{P}\left(\hat{I}_{1}=\hat{I}_{2}=\emptyset\right)-1\\
 & \geq & 1-c_{5}\exp\left(-c_{6}\log p\right)-\mathbb{P}\left(\left|\frac{1}{n}D^{T}\eta\right|>t\right).
\end{eqnarray*}
Note that $\left|\frac{1}{n}D^{T}\eta\right|\leq k\left|\gamma^{*}\right|\left|\frac{1}{n}X^{T}\eta\right|_{\infty}+\left|\frac{1}{n}v^{T}\eta\right|$
where 
\begin{eqnarray*}
\mathbb{P}\left(\left|\frac{1}{n}v^{T}\eta\right|\precsim\sigma_{v}\sigma_{\eta}r\right) & \geq & 1-2\exp\left(-c_{7}nr^{2}\right)\quad\textrm{for any }r\in(0,\,1],\\
\mathbb{P}\left(\left|\frac{1}{n}X^{T}\eta\right|_{\infty}\precsim\sigma_{\eta}\sqrt{\frac{\log p}{n}}\right) & \geq & 1-2\exp\left(-c_{8}\log p\right).
\end{eqnarray*}
The inequalities above yield 
\begin{align*}
 & \mathbb{P}\left\{ \left|\frac{1}{n}D^{T}\eta\right|\precsim\left[\left(\sigma_{v}\sigma_{\eta}r\right)\vee\underset{\asymp\sigma_{v}\sigma_{\eta}\frac{k\log p}{n}}{\underbrace{\left(k\left|\gamma^{*}\right|\sigma_{\eta}\sqrt{\frac{\log p}{n}}\right)}}\right]\vert\hat{I}_{1}=\hat{I}_{2}=\emptyset\right\} \\
\geq & 1-c_{5}\exp\left(-c_{6}\log p\right)-2\exp\left(-c_{7}nr^{2}\right)-2\exp\left(-c_{8}\log p\right).
\end{align*}
We have already shown that, conditioning on $\left\{ \hat{I}_{1}=\hat{I}_{2}=\emptyset\right\} $,
$\frac{1}{n}D^{T}D\succsim\left(\frac{k\log p}{n}\vee1\right)\sigma_{v}^{2}$
with probability at least $1-c_{5}k\exp\left(-c_{6}\log p\right)-\exp\left(\frac{-n}{64}\right)$.
As a consequence, 
\begin{eqnarray*}
 &  & \mathbb{P}\left\{ \left|\frac{\frac{1}{n}D^{T}\eta}{\frac{1}{n}D^{T}D}\right|\precsim\frac{\sigma_{v}\sigma_{\eta}\left(r\vee\frac{k\log p}{n}\right)}{\left(\frac{k\log p}{n}\vee1\right)\sigma_{v}^{2}}\vert\hat{I}_{1}=\hat{I}_{2}=\emptyset\right\} \\
 & \geq & \mathbb{P}\left\{ \left|\frac{1}{n}D^{T}\eta\right|\precsim\sigma_{v}\sigma_{\eta}\left(r\vee\frac{k\log p}{n}\right)\,\,and\,\,\frac{1}{n}D^{T}D\succsim\left(\frac{k\log p}{n}\vee1\right)\sigma_{v}^{2}\vert\hat{I}_{1}=\hat{I}_{2}=\emptyset\right\} \\
 & \geq & \mathbb{P}\left\{ \left|\frac{1}{n}D^{T}\eta\right|\precsim\sigma_{v}\sigma_{\eta}\left(r\vee\frac{k\log p}{n}\right)\vert\hat{I}_{1}=\hat{I}_{2}=\emptyset\right\} \\
 &  & +\mathbb{P}\left\{ \frac{1}{n}D^{T}D\succsim\left(\frac{k\log p}{n}\vee1\right)\sigma_{v}^{2}\vert\hat{I}_{1}=\hat{I}_{2}=\emptyset\right\} -1\\
 & \geq & 1-c_{9}k\exp\left(-c_{10}\log p\right)-c_{12}\exp\left(-c_{11}nr^{2}\right).
\end{eqnarray*}

Putting the pieces above together yields 
\[
\mathbb{P}\left(\left|\tilde{\alpha}-\alpha^{*}\right|\leq\overline{OVB}\vert\hat{I}_{1}=\hat{I}_{2}=\emptyset\right)\geq1-c_{1}^{'}k\exp\left(-c_{2}^{'}\log p\right)-c_{4}^{'}\exp\left(-c_{3}^{'}nr^{2}\right)
\]
where $\overline{OVB}\asymp\max\left\{ \frac{\sigma_{\eta}}{\sigma_{v}}\left(\frac{k\log p}{n}\wedge1\right),\,\frac{\sigma_{v}\sigma_{\eta}\left(r\vee\frac{k\log p}{n}\right)}{\left(\frac{k\log p}{n}\vee1\right)\sigma_{v}^{2}}\right\} $.

\subsection{Proof for Proposition \ref{prop:bias_post_double_selection-upper-1} }

\label{appendix:Proposition-OVB-upper-proof-1}

For part (i), substituting $\beta_{K}^{*}=-\alpha^{*}\gamma_{K}^{*}$
in $\frac{\frac{1}{n}D^{T}X_{K}}{\frac{1}{n}D^{T}D}\beta_{K}^{*}$
and following the rest of proof for Proposition \ref{prop:bias_post_double_selection-upper}
yields the claim. For part (ii), note that $\left|\beta_{j}^{*}\right|\leq\left|\pi_{j}^{*}\right|+\left|\alpha^{*}\gamma_{j}^{*}\right|\precsim\left(\sigma_{\eta}+\left|\alpha^{*}\right|\sigma_{v}\right)\sqrt{\frac{\log p}{n}}$
for all $j\in K$. Substituting $\left|\beta_{j}^{*}\right|\precsim\left(\sigma_{\eta}+\left|\alpha^{*}\right|\sigma_{v}\right)\sqrt{\frac{\log p}{n}}$
in the derivation of the upper bound for $\left|\frac{\frac{1}{n}D^{T}X_{K}}{\frac{1}{n}D^{T}D}\beta_{K}^{*}\right|$
and following the rest of proof for Proposition \ref{prop:bias_post_double_selection-upper}
yields the claim.

\section{Debiased Lasso}
\label{app: debiased lasso}
In this section, we present theoretical and simulation results on
the OVB of the debiased Lasso proposed by \citet{vandergeer2014asymptotically}.

\subsection{Theoretical results}

\label{sec:bias_debiased_lasso}

The idea of debiased Lasso is to start with an initial Lasso estimate
$\hat{\theta}=\left(\hat{\alpha},\,\hat{\beta}\right)$ of $\theta^{*}=\left(\alpha^{*},\,\beta^{*}\right)$
in equation (\ref{eq:main-y}), where 
\begin{equation}
\left(\hat{\alpha},\,\hat{\beta}\right)\in\textrm{arg}\min_{\alpha\in\mathbb{R},\beta\in\mathbb{R}^{p}}\frac{1}{2n}\left|Y-D\alpha-X\beta\right|_{2}^{2}+\lambda_{1}\left(\left|\alpha\right|+\left|\beta\right|_{1}\right).\label{eq:21-2}
\end{equation}
Given the initial Lasso estimator $\hat{\alpha}$, the debiased Lasso
adds a correction term to $\hat{\alpha}$ to reduce the bias introduced
by regularization. In particular, the debiased Lasso takes the form
\begin{equation}
\tilde{\alpha}=\hat{\alpha}+\frac{\hat{\Omega}_{1}}{n}\sum_{i=1}^{n}Z_{i}^{T}\left(Y_{i}-Z_{i}\hat{\theta}\right),\label{eq:5-3}
\end{equation}
where $Z_{i}=\left(D_{i},\,X_{i}\right)$ and $\hat{\Omega}_{1}$
is the first row of $\hat{\Omega}$, which is an approximate inverse
of $\frac{1}{n}Z^{T}Z$, $Z=\left\{ Z_{i}\right\} _{i=1}^{n}$. Several
different strategies have been proposed for constructing the approximate
inverse $\hat{\Omega}$; see, for example, \citet{javanmard2014confidence},
\citet{vandergeer2014asymptotically}, and \citet{Zhang_Zhang}. We
will focus on the widely used method proposed by \citet{vandergeer2014asymptotically},
which sets 
\begin{eqnarray*}
\hat{\Omega}_{1} & := & \hat{\tau}_{1}^{-2}\left(\begin{array}{cccc}
1 & -\hat{\gamma}_{1} & \cdots & -\hat{\gamma}_{p}\end{array}\right),\\
\hat{\tau}_{1}^{2} & := & \frac{1}{n}\left|D-X\hat{\gamma}\right|_{2}^{2}+\lambda_{2}\left|\hat{\gamma}\right|_{1},
\end{eqnarray*}
where $\hat{\gamma}$ is defined in (\ref{eq:las-2}).

\begin{proposition}[Scaling of OVB lower bound for debiased Lasso]\label{prop:bias_debiased_lasso}
Let Assumption \ref{ass: incoherence}(ii) and Assumption \ref{ass:inference_normality}
hold. Suppose: with probability at least $1-\kappa$, $\left\Vert \left(Z_{-K}^{T}X_{K}\right)\left(X_{K}^{T}X_{K}\right)^{-1}\right\Vert _{\infty}\leq1-\frac{\phi}{2}$
for some $\phi\in(0,\,1]$ such that $\phi^{-1}\precsim1$, where
$Z_{-K}$ denotes the columns in $Z=\left(D,\,X\right)$ excluding
$X_{K}$; the regularization parameters in (\ref{eq:las-2}) and (\ref{eq:21-2})
are chosen in a similar fashion as in Lemma \ref{prop:fixed_design}
such that $\lambda_{1}\asymp\phi^{-1}\left(1\vee\sigma_{v}\right)\sigma_{\eta}\sqrt{\frac{\log p}{n}}$
and $\lambda_{2}\asymp\phi^{-1}\sigma_{v}\sqrt{\frac{\log p}{n}}$;
for all $j\in K$, $\left|\beta_{j}^{*}\right|\leq\frac{\lambda_{1}}{2}$
and $\left|\gamma_{j}^{*}\right|\leq\frac{\lambda_{2}}{2}$, but $\left|\beta_{j}^{*}\right|\asymp\left[1\vee\sigma_{v}\right]\sigma_{\eta}\sqrt{\frac{\log p}{n}}$
and $\left|\gamma_{j}^{*}\right|\asymp\sigma_{v}\sqrt{\frac{\log p}{n}}$.
Let us consider $\tilde{\alpha}$ obtained from (\ref{eq:5-3}). If
$\alpha^{*}=0$, then there exist positive universal constants $c^{\dagger},c_{1},c_{2},c_{3},c^{*},c_{0}^{*}$
such that 
\[
\left|\mathbb{E}\left(\tilde{\alpha}-\alpha^{*}\vert\mathcal{M}^{'}\right)\right|\geq c^{\dagger}\frac{\sigma_{\eta}}{\sigma_{v}}\left(\frac{k\log p}{n}\wedge1\right)\left[1-2\kappa-c_{1}k\exp\left(-c_{2}\log p\right)-\exp\left(-c_{3}n\right)\right],
\]
where $\mathcal{M}^{'}$ is an event with $\mathbb{P}\left(\mathcal{M}^{'}\right)\geq1-2\kappa-c^{*}k\exp\left(-c_{0}^{*}\log p\right)$.\end{proposition}

\begin{proposition}[Scaling of OVB upper bound for debiased Lasso]\label{prop:bias_debiased_lasso-upper}
Let Assumption \ref{ass: incoherence}(ii) and Assumption \ref{ass:inference_normality}
hold. Suppose: with probability at least $1-\kappa$, $\left\Vert \left(Z_{-K}^{T}X_{K}\right)\left(X_{K}^{T}X_{K}\right)^{-1}\right\Vert _{\infty}\leq1-\frac{\phi}{2}$
for some $\phi\in(0,\,1]$ such that $\phi^{-1}\precsim1$, where
$Z_{-K}$ denotes the columns in $Z=\left(D,\,X\right)$ excluding
$X_{K}$; the regularization parameters in (\ref{eq:las-2}) and (\ref{eq:21-2})
are chosen in a similar fashion as in Proposition \ref{prop:bias_debiased_lasso}
such that $\lambda_{1}\asymp\phi^{-1}\left(1\vee\sigma_{v}\right)\sigma_{\eta}\sqrt{\frac{\log p}{n}}$
and $\lambda_{2}\asymp\phi^{-1}\sigma_{v}\sqrt{\frac{\log p}{n}}$;
for all $j\in K$, $\gamma_{j}^{*}=\gamma^{*}$, $\left|\beta_{j}^{*}\right|\leq\frac{\lambda_{1}}{2}$
and $\left|\gamma^{*}\right|\leq\frac{\lambda_{2}}{2}$, but $\left|\beta_{j}^{*}\right|\asymp\left[1\vee\sigma_{v}\right]\sigma_{\eta}\sqrt{\frac{\log p}{n}}$
and $\left|\gamma^{*}\right|\asymp\sigma_{v}\sqrt{\frac{\log p}{n}}$.
Let us consider $\tilde{\alpha}$ obtained from (\ref{eq:5-3}). If
$\alpha^{*}=0$, then there exist positive universal constants $c_{1},c_{2},c_{3},c_{4},c^{*},c_{0}^{*}$
such that 
\[
\mathbb{P}\left(\left|\tilde{\alpha}-\alpha^{*}\right|\leq\overline{OVB}\vert\hat{\theta}=0_{p+1},\,\hat{\gamma}=0_{p}\right)\geq1-c_{1}k\exp\left(-c_{2}\log p\right)-c_{3}\exp\left(-c_{4}nr^{2}\right)
\]
where $\mathbb{P}\left(\hat{\theta}=0_{p+1},\,\hat{\gamma}=0_{p}\right)\geq1-2\kappa-c^{*}k\exp\left(-c_{0}^{*}\log p\right)$
and 
\[
\overline{OVB}\asymp\max\left\{ \frac{\sigma_{\eta}}{\sigma_{v}}\left(\frac{k\log p}{n}\wedge1\right),\,\frac{\sigma_{v}\sigma_{\eta}\left(r\vee\frac{k\log p}{n}\right)}{\left(\frac{k\log p}{n}\vee1\right)\sigma_{v}^{2}}\right\} 
\]
for any $r\in(0,\,1]$.

\end{proposition}

\begin{remark}One can show that a population version of the mutual
incoherence condition, $\left\Vert \left[\mathbb{E}\left(Z_{-K}^{T}\right)X_{K}\right]\left(X_{K}^{T}X_{K}\right)^{-1}\right\Vert _{\infty}=1-\phi$,
implies $\left\Vert \left(Z_{-K}^{T}X_{K}\right)\left(X_{K}^{T}X_{K}\right)^{-1}\right\Vert _{\infty}\leq1-\frac{\phi}{2}$
with high probability (that is, $\kappa$ is small and vanishes polynomially
in $p$). For example, we can apply (\ref{eq:77}) in Lemma \ref{lem:A3}
with slight notational changes.\end{remark}

\begin{remark} The event $\mathcal{M}^{'}$ in Proposition \ref{prop:bias_debiased_lasso}
is the intersection of $\left\{ \hat{\theta}=0_{p+1},\,\hat{\gamma}=0_{p}\right\} $
and an additional event, both of which occur with high probabilities.
The additional event is needed in our analyses for technical reasons.
See Appendix \ref{sec:appendix_debiased_lasso} for details. \end{remark}

\subsection{Proof for Propositions \ref{prop:bias_debiased_lasso} and \ref{prop:bias_debiased_lasso-upper}}

\label{sec:appendix_debiased_lasso}

Under the conditions in Proposition \ref{prop:bias_debiased_lasso},
(\ref{eq:nec}) in Lemma \ref{prop:fixed_design} implies that $\hat{\gamma}=0_{p}$
with probability at least $1-c_{0}\exp\left(-c_{0}^{'}\log p\right)$.
Conditioning on this event, $\hat{\Omega}_{1}=\left(\frac{1}{n}D^{T}D\right)^{-1}e_{1}$
where $e_{1}=\left(\begin{array}{cccc}
1 & 0 & \cdots & 0\end{array}\right)$. If $\alpha^{*}=0$, under the conditions in Proposition \ref{prop:bias_debiased_lasso},
we show that $\hat{\theta}=0_{p+1}$ with probability at least $1-2\kappa-c_{5}\exp\left(-c_{6}\log p\right)$.
To achieve this goal, we slightly modify the argument for (\ref{eq:nec})
in Lemma \ref{prop:fixed_design} by replacing (\ref{eq:event}) with
$\mathcal{E}=\mathcal{E}_{1}\cap\mathcal{E}_{2}$, where

\begin{eqnarray*}
\mathcal{E}_{1} & = & \left\{ \left|\frac{Z^{T}\eta}{n}\right|_{\infty}\precsim\phi^{-1}\left(1\vee\sigma_{v}\right)\sigma_{\eta}\sqrt{\frac{\log p}{n}}\right\} ,\\
\mathcal{E}_{2} & = & \left\{ \left\Vert \left(Z_{-K}^{T}X_{K}\right)\left(X_{K}^{T}X_{K}\right)^{-1}\right\Vert _{\infty}\leq1-\frac{\phi}{2}\right\} ,
\end{eqnarray*}
and $Z_{-K}$ denotes the columns in $Z$ excluding $X_{K}$. Note
that by (\ref{eq:A37-1}), $\mathbb{P}\left(\mathcal{E}_{1}\right)\geq1-c_{1}^{'}\exp\left(-c_{2}^{'}\log p\right)$
and therefore, $\mathbb{P}\left(\mathcal{E}\right)\geq1-\kappa-c_{1}^{'}\exp\left(-c_{2}^{'}\log p\right)$.
We then follow the argument used in the proof for Lemma \ref{prop:fixed_design}
to show $\mathbb{P}\left(E_{1}\cap\mathcal{E}\right)=0$ and $\mathbb{P}\left(E_{2}\cap\mathcal{E}\right)=0$,
where 
\begin{eqnarray*}
E_{1} & = & \left\{ \textrm{sgn}\left(\hat{\beta}_{j}\right)=-\textrm{sgn}\left(\beta_{j}^{*}\right),\,\textrm{for some }j\in K\right\} ,\\
E_{2} & = & \left\{ \textrm{sgn}\left(\hat{\beta}_{j}\right)=\textrm{sgn}\left(\beta_{j}^{*}\right),\,\textrm{for some }j\in K\right\} .
\end{eqnarray*}
Moreover, conditioning on $\mathcal{E}$, $\hat{\alpha}=0$ and $\hat{\beta}_{K^{c}}=0_{p-k}$.
Putting these facts together yield the claim that $\hat{\theta}=0_{p+1}$
with probability at least $1-2\kappa-c_{5}\exp\left(-c_{6}\log p\right)$.

Letting $E=\left\{ \hat{\theta}=0_{p+1},\,\hat{\gamma}=0_{p}\right\} $
with $\mathbb{P}\left(E\right)\geq1-2\kappa-c_{1}\exp\left(-c_{2}\log p\right)$
and recalling the event $\mathcal{E}_{t^{*}}$ in the proof for Proposition
\ref{prop:bias_post_double_formula}, we can then show 
\begin{eqnarray*}
\mathbb{E}\left(\tilde{\alpha}-\alpha^{*}\vert\mathcal{M}^{'}\right) & = & \frac{1}{n}\mathbb{E}\left(\hat{\Omega}_{1}Z^{T}\eta\vert\mathcal{M}^{'}\right)+\mathbb{E}\left[\frac{D^{T}X_{K}}{D^{T}D}\left(\beta_{K}^{*}-\hat{\beta}_{K}\right)\vert\mathcal{M}^{'}\right]\\
 & = & \mathbb{E}\left(\frac{\frac{1}{n}D^{T}\eta}{\frac{1}{n}D^{T}D}\vert\mathcal{M}^{'}\right)+\mathbb{E}\left(\frac{D^{T}X_{K}}{D^{T}D}\beta_{K}^{*}\vert\mathcal{M}^{'}\right)\\
 & = & \mathbb{E}\left(\frac{D^{T}X_{K}}{D^{T}D}\beta_{K}^{*}\vert\mathcal{M}^{'}\right)
\end{eqnarray*}
where $\mathcal{M}^{'}=E\cap\mathcal{E}_{t^{*}}$ such that $\mathbb{P}\left(\mathcal{M}^{'}\right)\geq1-2\kappa-c_{3}^{*}k\exp\left(-c_{4}^{*}\log p\right)$
and the last line follows from the argument used to show (\ref{eq:60}).

The rest of argument is similar to what is used in showing Proposition
\ref{prop:bias_post_double_formula}. However, because (\ref{eq:21-2})
involves $D$, $\mathbb{E}\left(\frac{\frac{1}{n}D^{T}X_{K}}{\frac{1}{n}D^{T}D}\beta_{K}^{*}\vert E,\,\mathcal{E}_{t^{*}}\right)\neq\mathbb{E}\left(\frac{\frac{1}{n}D^{T}X_{K}}{\frac{1}{n}D^{T}D}\beta_{K}^{*}\vert\hat{\gamma}=0_{p},\,\mathcal{E}_{t^{*}}\right)$.
Instead of (\ref{eq:50}) and (\ref{eq:58}), we apply 
\begin{eqnarray*}
\mathbb{P}\left(\mathcal{E}_{t_{2}}^{'}\vert E,\,\mathcal{E}_{t_{1}}\right) & \geq & \mathbb{P}\left(\mathcal{E}_{t_{2}}^{'}\cap\mathcal{E}_{t_{1}}\cap E\right)\\
 & \geq & \mathbb{P}\left(\mathcal{E}_{t_{1}}\right)+\mathbb{P}\left(\mathcal{E}_{t_{2}}^{'}\right)+\mathbb{P}\left(E\right)-2\\
 & \geq & 1-2\kappa-c_{1}\exp\left(-c_{2}\log p\right)-k\exp\left(\frac{-nt_{1}^{2}}{2\sigma_{v}^{2}}\right)\\
 &  & -\exp\left(\frac{-nt_{2}^{2}}{8\sigma_{v}^{4}}\right),\qquad\textrm{for any }t_{2}\in(0,\,\sigma_{v}^{2}].
\end{eqnarray*}
Consequently, we have the claim in Proposition \ref{prop:bias_debiased_lasso}.

Following the argument used to show Proposition \ref{prop:bias_post_double_selection-upper},
we also have the claim in Proposition \ref{prop:bias_debiased_lasso-upper}.

\subsection{Simulation evidence}

Here we evaluate the performance of the debiased Lasso proposed by \citet{vandergeer2014asymptotically} based on the simulation setting of the main text. We use cross-validation  to choose the regularization parameters as this is the most commonly-used method in this literature. Figure \ref{fig:db_performance} presents the results. Debiased Lasso exhibits substantial biases (relative to the standard deviation) and under-coverage for all except very small values of $R^2$, and its performance is very sensitive to the regularization choice. A comparison to the results in the main text shows that post double Lasso performs better than debiased Lasso.\footnote{We found that one of the reasons for the relatively poor performance of debiased Lasso is that $D$ is highly correlated with the relevant controls. Debiased Lasso exhibits a better performance when $(D,X)$ exhibit a Toeplitz dependence structure as in the simulations reported by \citet{vandergeer2014asymptotically}.}

\begin{figure}[ht]
\caption{Finite sample behavior of the debiased Lasso}

    \begin{subfigure}[b]{\textwidth}
    \caption{Ratio of bias to standard deviation}
    \centering
    \includegraphics[width=0.325\textwidth, trim = {0 0 0 0cm}]{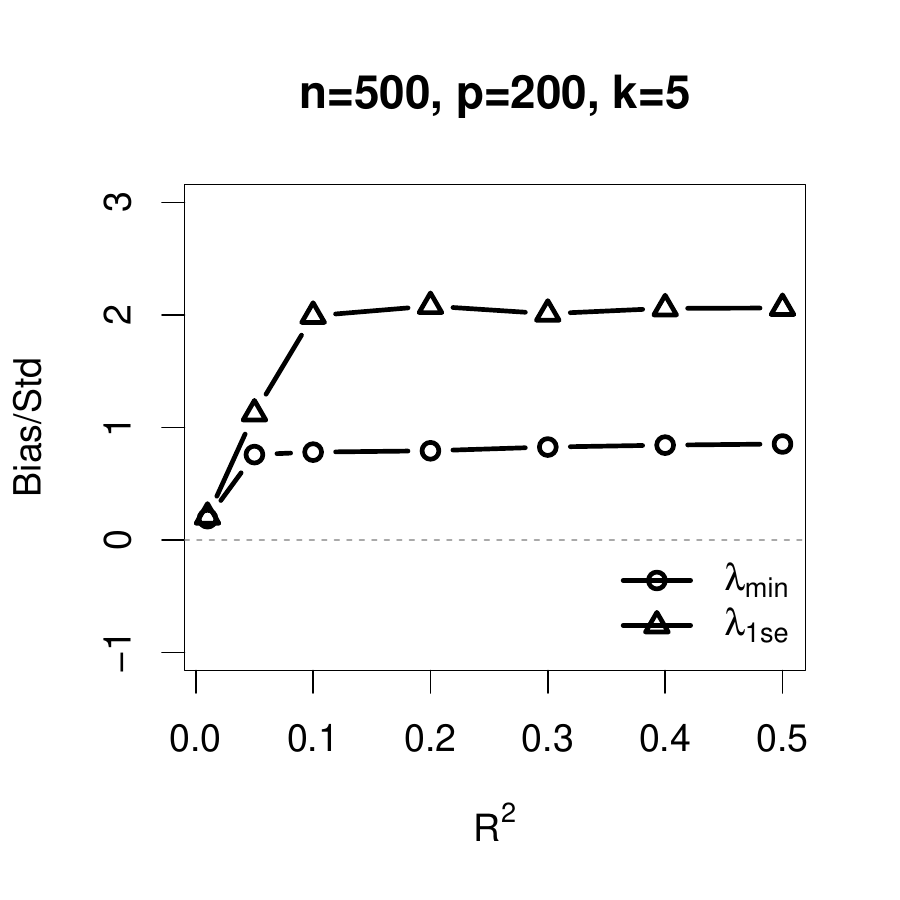}
    \includegraphics[width=0.325\textwidth, trim = {0 0 0 0cm}]{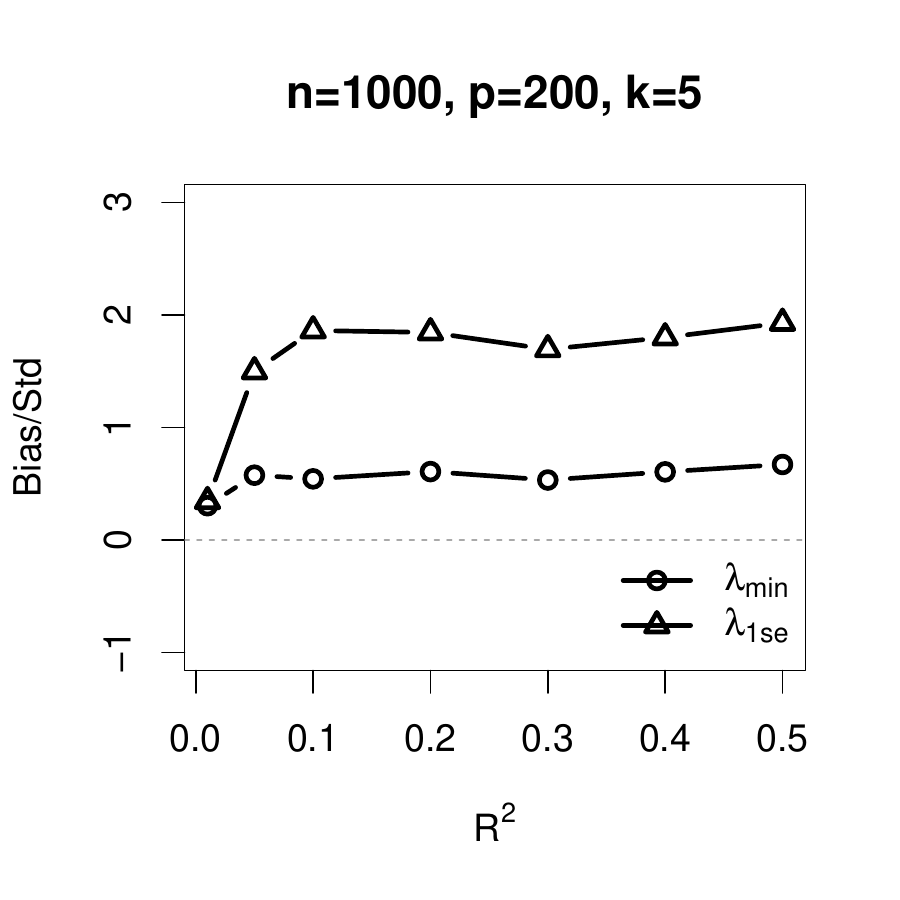}
    \includegraphics[width=0.325\textwidth, trim = {0 0 0 0cm}]{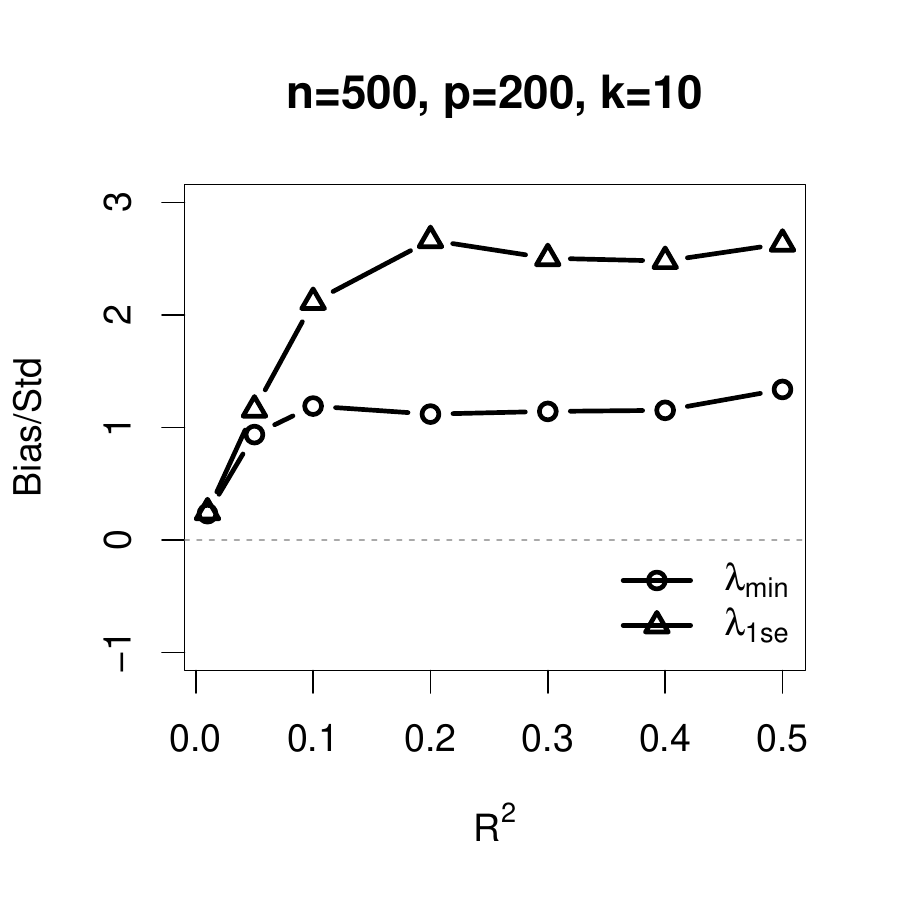}
    \end{subfigure}  
    
    \begin{subfigure}[b]{\textwidth}
    \caption{Coverage 90\% confidence intervals}
    \centering
    \includegraphics[width=0.325\textwidth, trim = {0 0 0 0cm}]{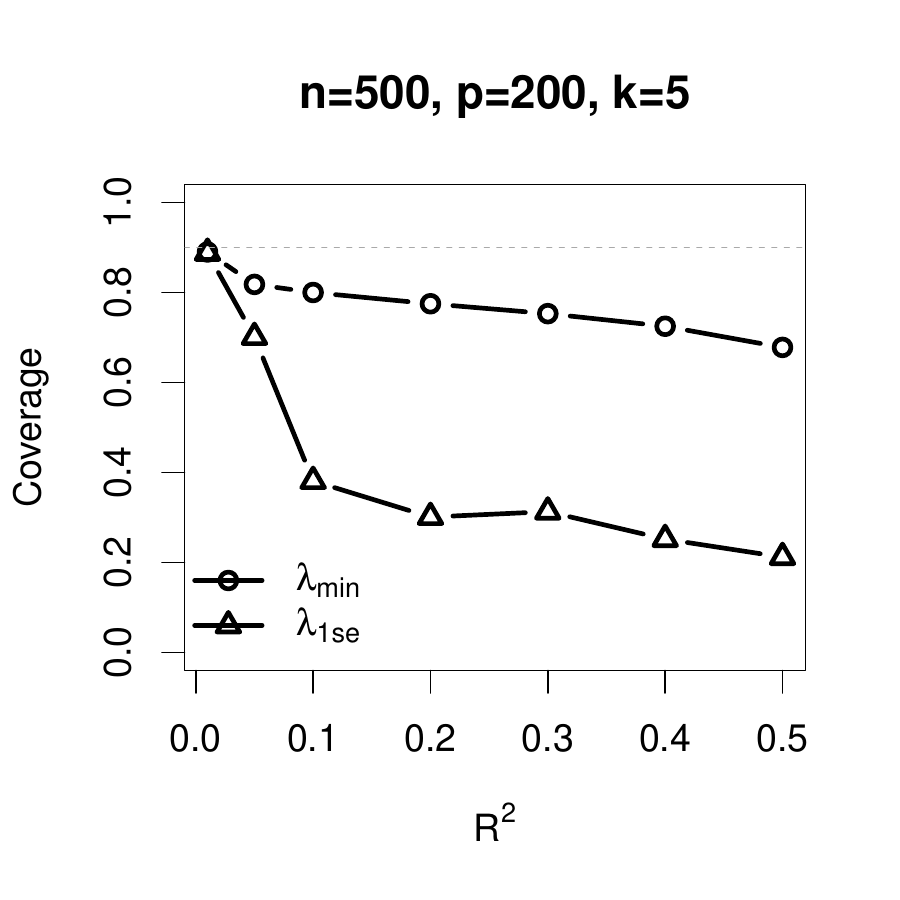}
    \includegraphics[width=0.325\textwidth, trim = {0 0 0 0cm}]{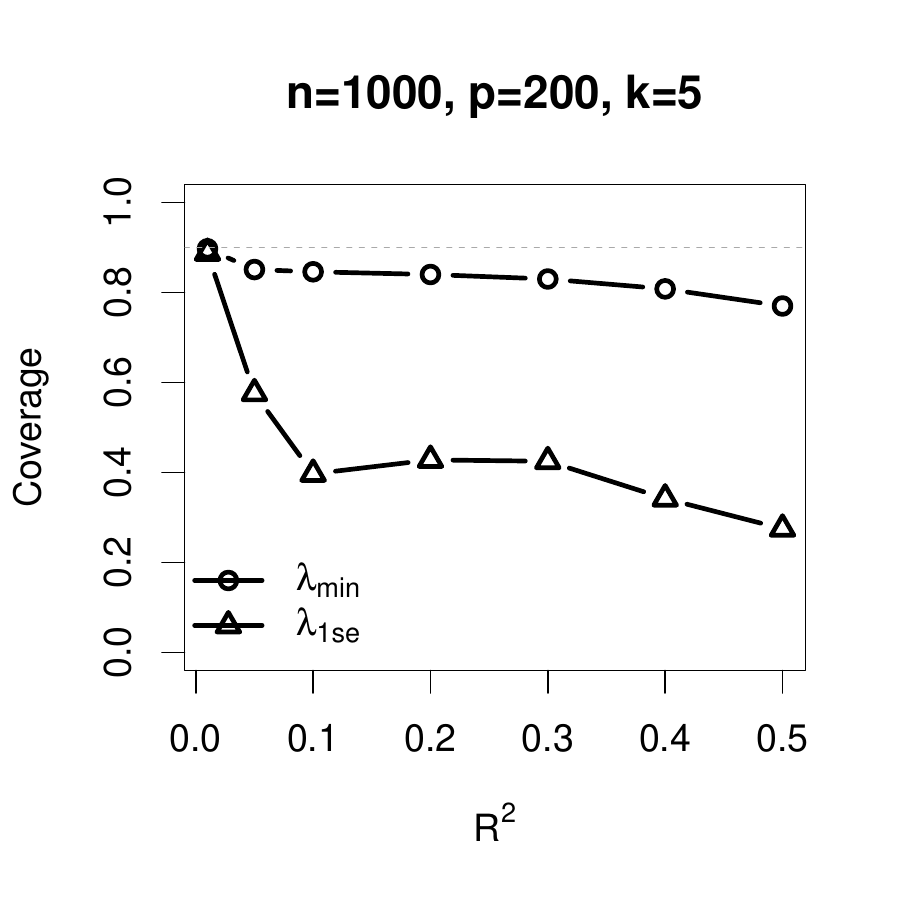}
    \includegraphics[width=0.325\textwidth, trim = {0 0 0 0cm}]{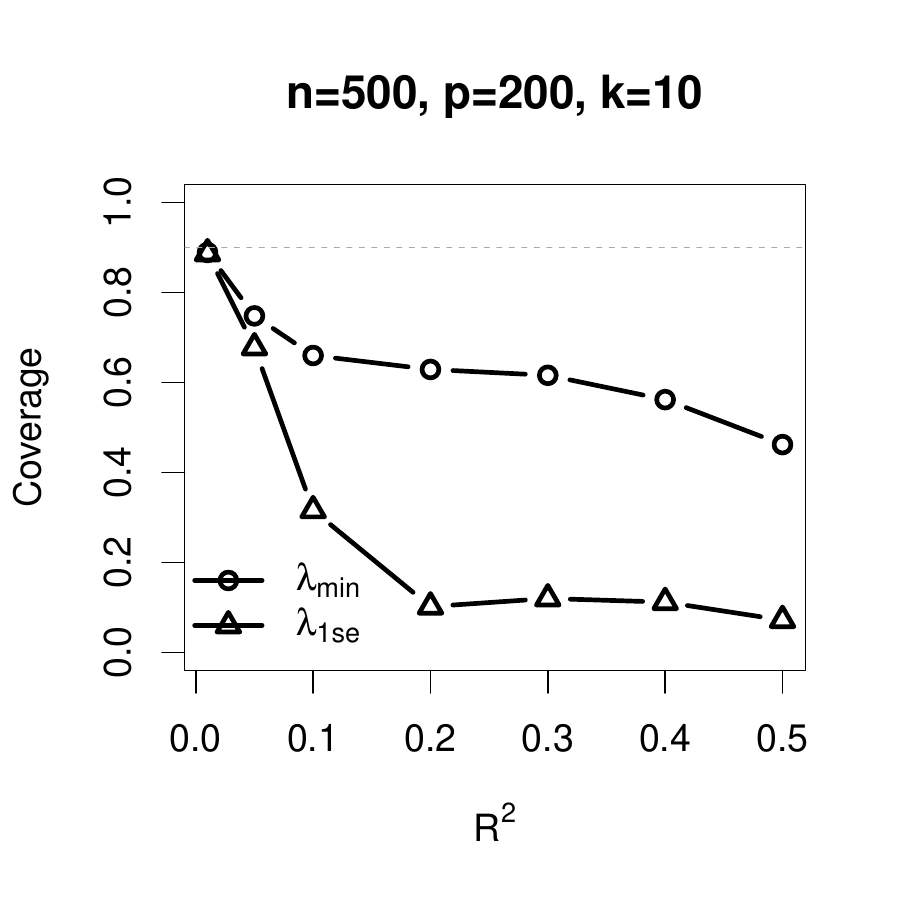}
    \end{subfigure}  
\label{fig:db_performance}

\end{figure}

\section{Additional simulations for post double Lasso}
\label{app:additional_simulations}

\subsection{Large sample simulations based on the numerical example}
\label{app: large sample}
The issue documented in the numerical example is not a small-sample phenomenon. Figure \ref{fig:num_example_large} shows that, even when $(n,p)=(14238,384)$ as in \citet{angrist2019machine} and $k=10$, the finite sample distribution of post double Lasso may not be centered at the true value, and the bias can be large relative to the standard deviation. Compared to the results for $(n,p)=(500,200)$, under-selection and large biases occur at lower values of $R^2$, and all the relevant controls get selected when $R^2=0.2$.

\begin{figure}[ht]
\caption{Performance with $(n,p,k)=(14238,384,10)$}

    \begin{subfigure}[b]{\textwidth}
    \caption{Finite sample distribution}
    \centering
    \includegraphics[width=0.75\textwidth, trim = {0 0cm 0 0cm}]{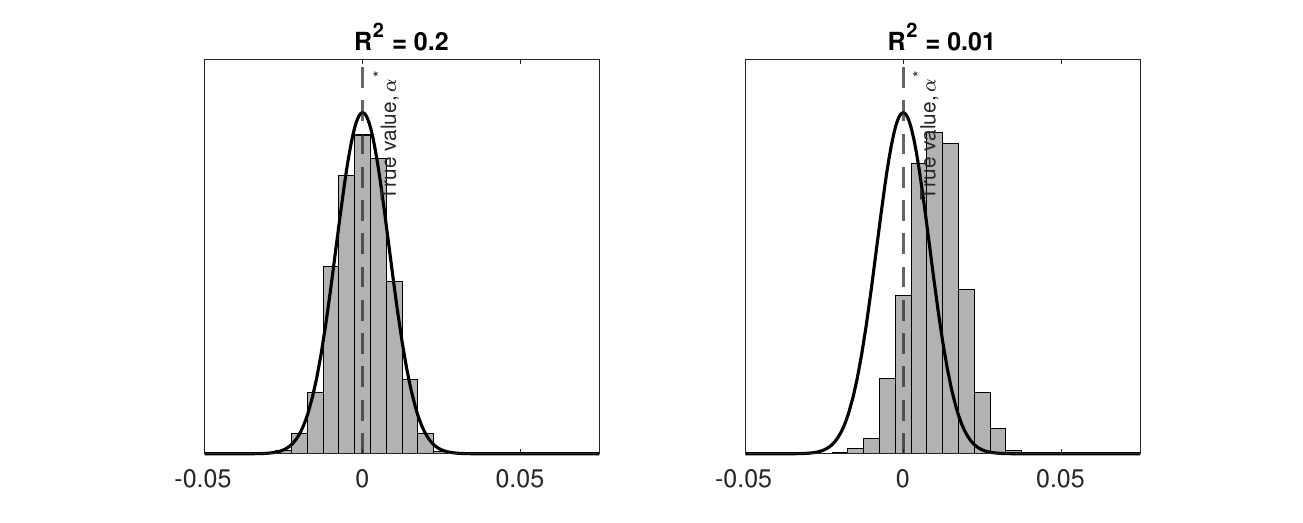}
    \end{subfigure}  
    
    \begin{subfigure}[b]{\textwidth}
    \caption{Number of selected relevant controls}
    \centering
    \includegraphics[width=0.75\textwidth, trim = {0 0cm 0 0cm}]{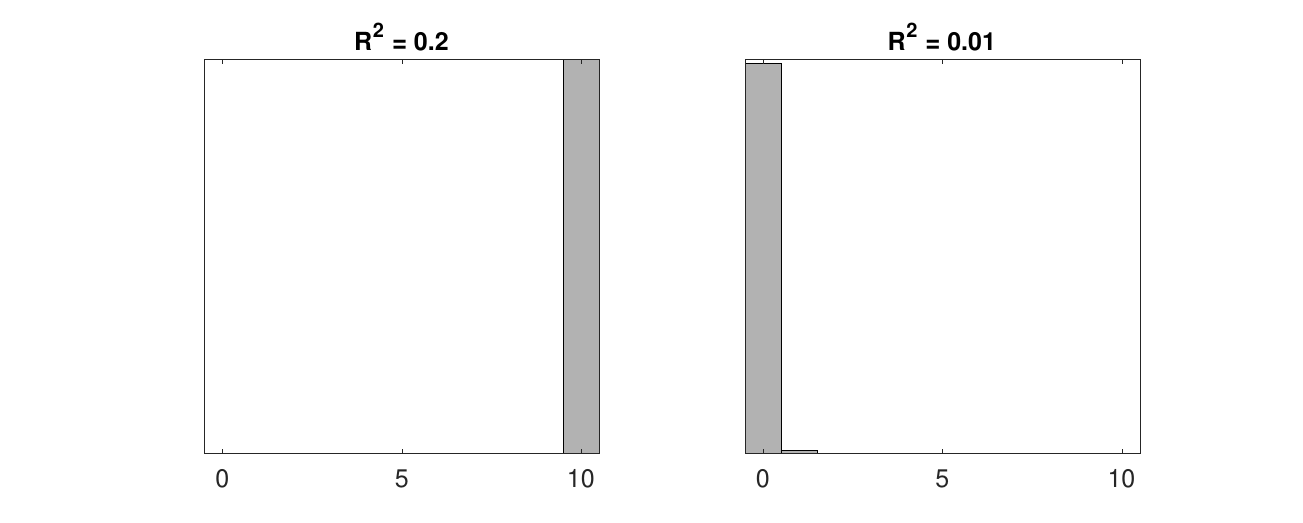}
    \end{subfigure}  
\label{fig:num_example_large}
\footnotesize{\textit{Notes:} The grey histograms in Panel (a) show the finite sample distributions and the black curves show the densities of the oracle estimators. }
\end{figure}

\subsection{Additional simulations}
In the main text, we consider a setting with normally distributed control variables, normally distributed homoscedastic errors terms, and $\alpha^\ast=0$. Here we provide additional simulation evidence based on a more general model used in the simulations of \citet{belloni2014inference}: 
\begin{eqnarray}
Y_{i} & = & D_{i}\alpha^{*}+X_{i}\beta^{*}+\sigma_{y}(D_{i},X_{i})\eta_{i},\label{eq:dgp_illustration_inference1_general}\\
D_{i} & = & X_{i}\gamma^{*}+\sigma_{d}(X_{i})v_{i},\label{eq:dgp_illustration_inference2_general}
\end{eqnarray}
where $\eta_{i}$ and $v_{i}$ are independent of each other and $\left\{ X_{i},\eta_{i},v_{i}\right\} _{i=1}^{n}$ consists of iid entries. The object of interest is $\alpha^{*}$. We set $p=200$ and consider a sparse setting where $\beta_j^{*}=\gamma_j^{*}=c \cdot 1\left\{j\le k\right\}$ for $j=1,\dots,p$ and $k=5$. We consider six DGPs that differ with respect to $n$, the distributions of $X_i$, $\eta_{i}$, and $v_{i}$, the specifications of $\sigma_{y}(D_{i},X_{i})$ and $\sigma_{d}(X_{i})$, as well as $\alpha^\ast$. For DGP A4, we do not show results for OLS since $n=p$. We show results for $R^2\in \{0.01,0.05,0.1,0.2,0.3,0.4,0.5\}$ based on 1,000 simulation repetitions.

\begin{table}[ht]
\caption{DGPs A1--A6}
\begin{footnotesize}
\begin{center}

\begin{tabular}{| c | c | c | c | c | c | c | c | c | c |}
\hline
DGP&$X_i$ & $n$& $\sigma_y(D_i,X_i)$ & $\sigma_d(X_i)$ & $\eta_i$ & $v_i$& $\alpha^\ast$ \\
\hline
A1 &Indep.\ $\text{Bern}\left(0.5\right)$ &500& 1 & 1 & $ \mathcal{N}(0,1)$ & $ \mathcal{N}(0,1)$&0 \\
A2 & $\mathcal{N}(0_p,I_p)$ &500& 1 & 1 &   $\frac{t(5)}{\sqrt{(5/3)}}$ & $ \frac{t(5)}{\sqrt{(5/3)}}$&0 \\
A3 & $\mathcal{N}(0_p,I_p)$ &500& $\sqrt{\frac{(1+D_{i}\alpha^{*}+X_{i}\beta^{*})^{2}}{\frac{1}{n}\sum_{i}(1+D_{i}\alpha^{*}+X_{i}\beta^{*})^{2}}}$ & $\sqrt{\frac{(1+X_{i}\gamma^{*})^{2}}{\frac{1}{n}\sum_{i}(1+X_{i}\gamma^{*})^{2}}}$ & $ \mathcal{N}(0,1)$ & $ \mathcal{N}(0,1)$&0 \\
A4 &$\mathcal{N}(0_p,I_p)$ &200& 1 & 1 & $ \mathcal{N}(0,1)$ & $ \mathcal{N}(0,1)$&0 \\
A5 & $\mathcal{N}(0_p,I_p)$ &500& 1 & 1 &  $ \mathcal{N}(0,1)$ & $ \mathcal{N}(0,1)$&1 \\
A6 & $\mathcal{N}(0_p,I_p)$&500& 1 & 1 &  $ \mathcal{N}(0,1)$ & $ \mathcal{N}(0,1)$&-1 \\
\hline
\end{tabular}
\end{center}

\textit{Notes:}  The specification of multiplicative heteroscedasticity is the same as in the simulations in Section 4.2 of \citet{belloni2014inference}.
\end{footnotesize}
\end{table}

Figures \ref{fig:dml_bias_std_app}--\ref{fig:dml_ci_app} present the results. The two most important determinants of the performance of post double Lasso are $n$ and $\alpha^{*}$. To see why $\alpha^\ast$ is important, recall that the reduced form parameter and the error term in the first step of post double Lasso (i.e., program (\ref{eq:las-1})) are $\pi^{*}=\alpha^{*}\gamma^{*}+\beta^{*}$ and $u_{i}=\eta_{i}+\alpha^{*}v_{i}$. This implies that the magnitude of $\pi^{*}$ as well as the variance of $u_i$ depend on $\alpha^{*}$. Consequently, the selection performance of Lasso in the first step is directly affected by $\alpha^{*}$. In the extreme case where $\alpha^{*}$ is such that $\pi^{*}=0_{p}$, Lasso does not select any controls with high probability if the regularization parameter is chosen according to the standard recommendations. 
The simulation results further show that there is no practical recommendation for choosing the regularization parameters. While $\lambda_{\text{min}}$ leads to the best performance when $\alpha^\ast=0$, this choice can yield poor performances when $\alpha^\ast\ne 0$. Finally, across all DGPs, OLS outperforms post double Lasso in terms of bias and coverage accuracy, but yields somewhat wider confidence intervals. 
\begin{figure}[ht]
\caption{Ratio of bias to standard deviation}
\begin{center}
\includegraphics[width=0.325\textwidth,trim = {0 1cm 0 1cm}]{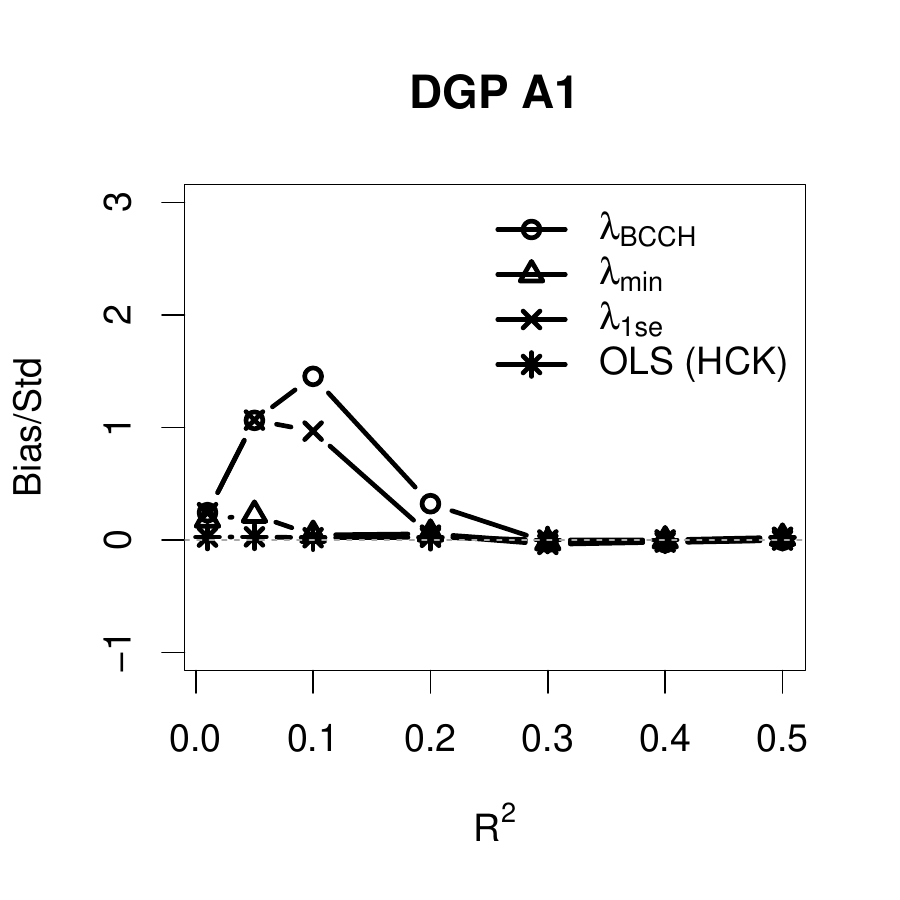}
\includegraphics[width=0.325\textwidth,trim = {0 1cm 0 1cm}]{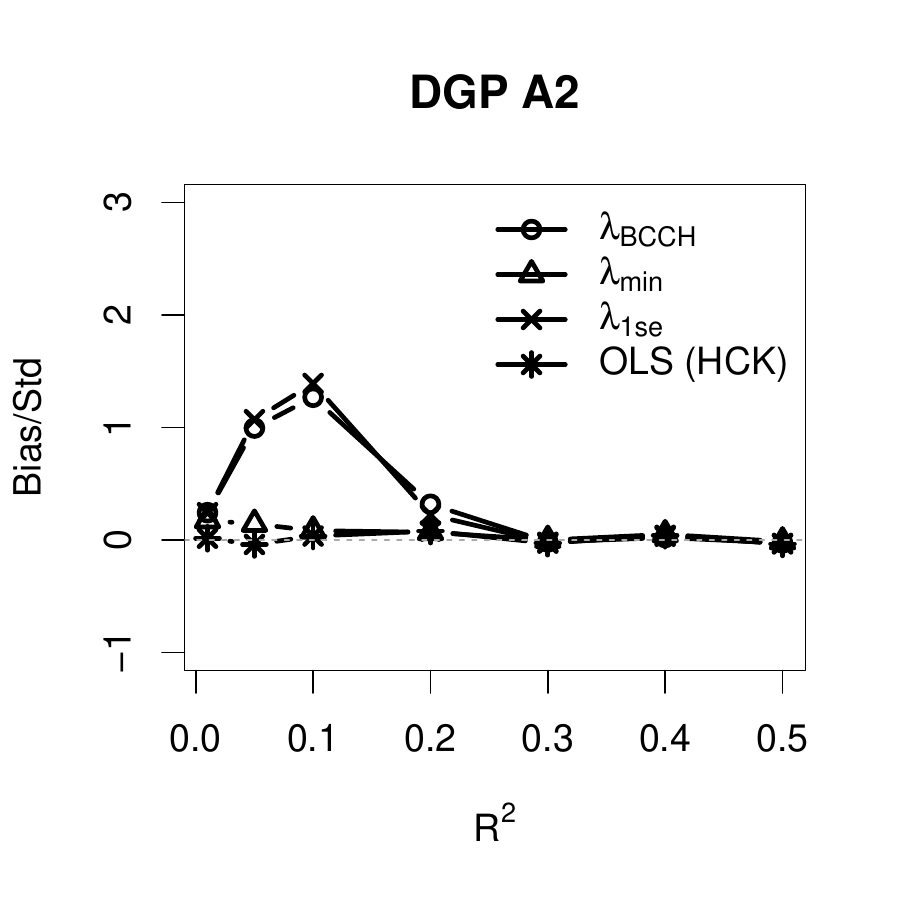}
\includegraphics[width=0.325\textwidth,trim = {0 1cm 0 1cm}]{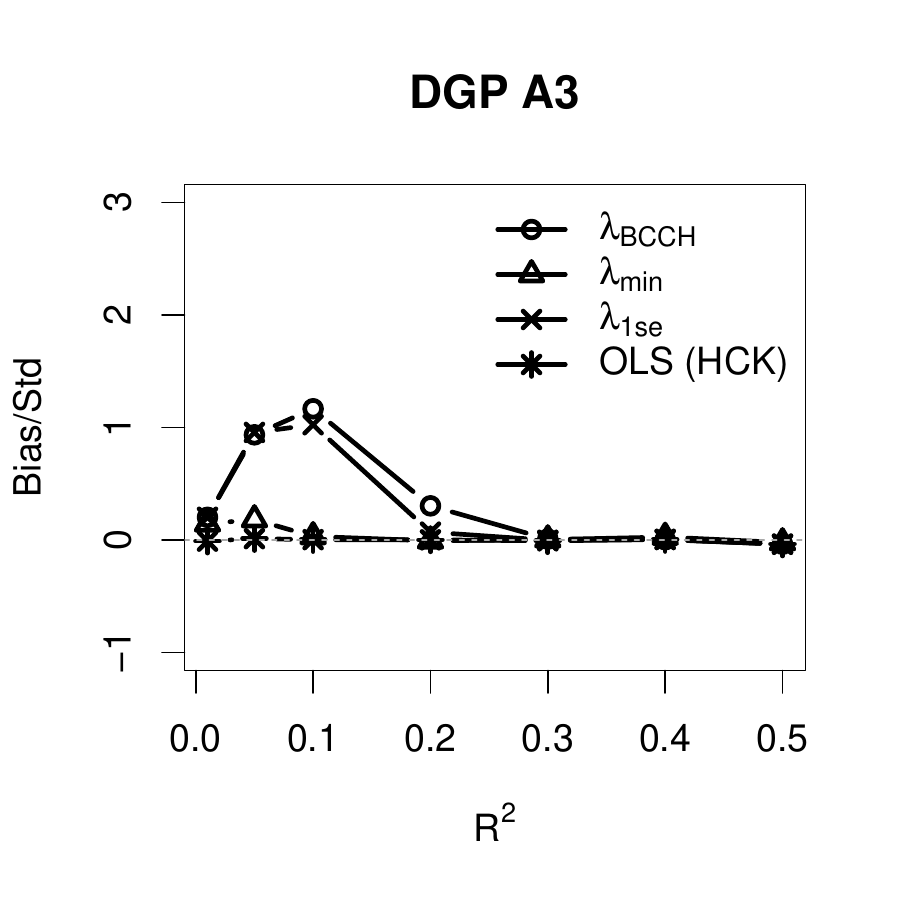}
\includegraphics[width=0.325\textwidth,trim = {0 1cm 0 .5cm}]{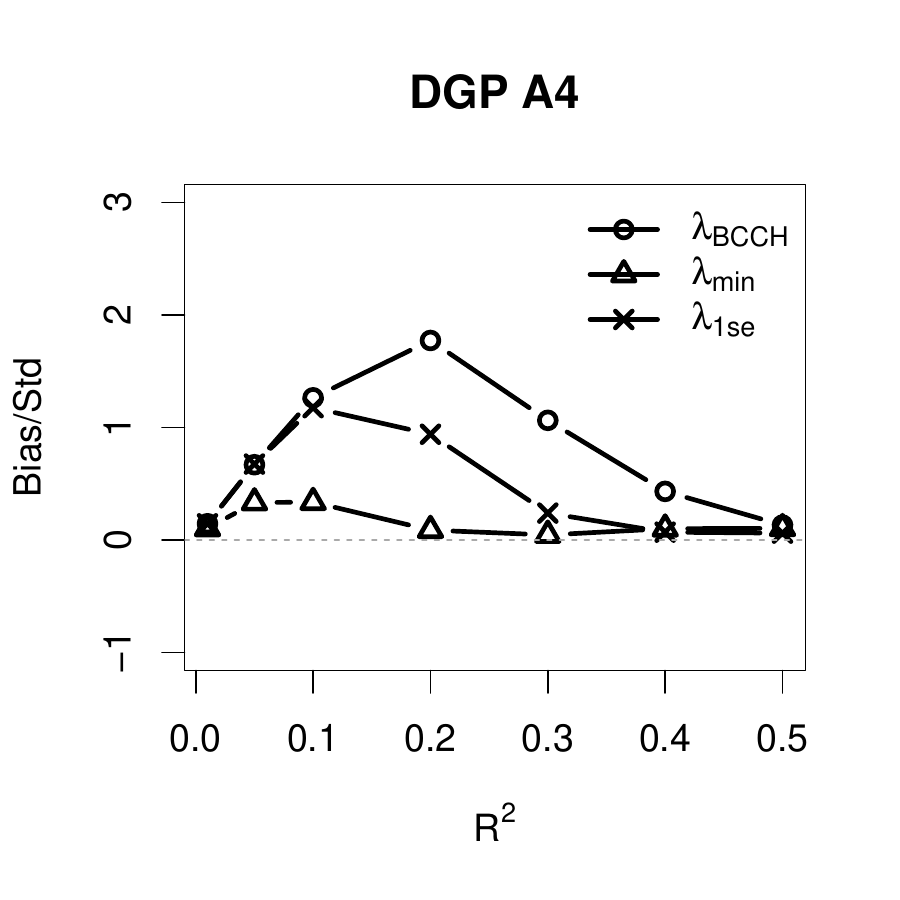}
\includegraphics[width=0.325\textwidth,trim = {0 1cm 0 .5cm}]{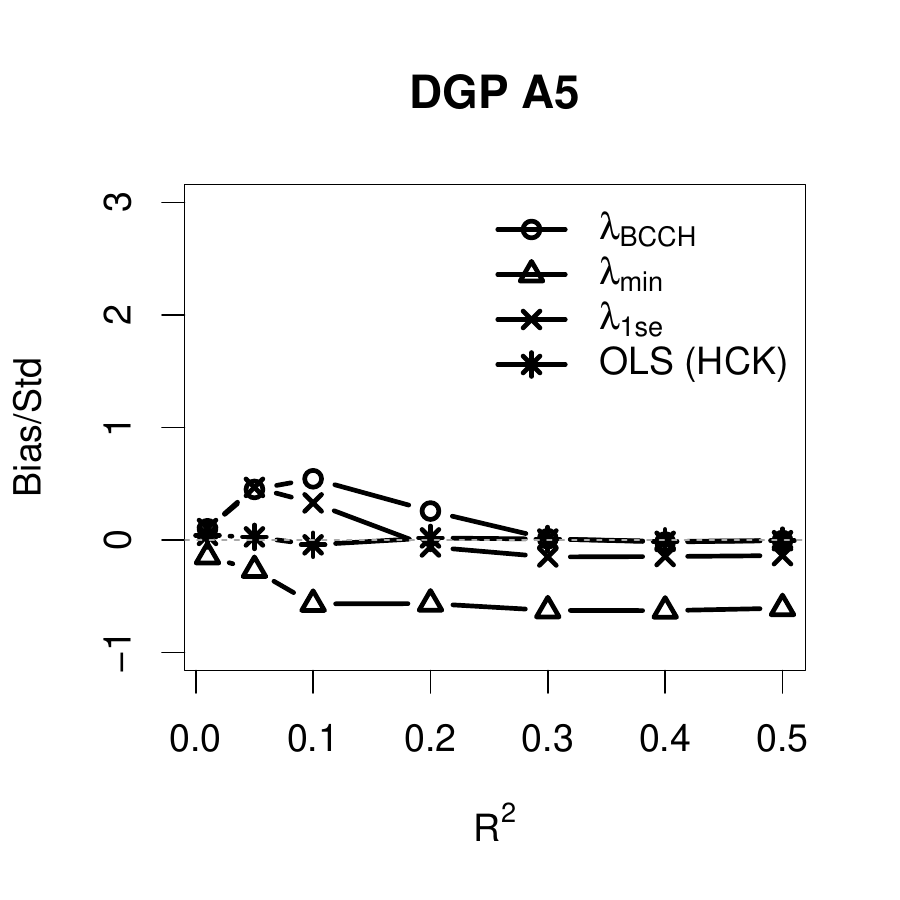}
\includegraphics[width=0.325\textwidth,trim = {0 1cm 0 .5cm}]{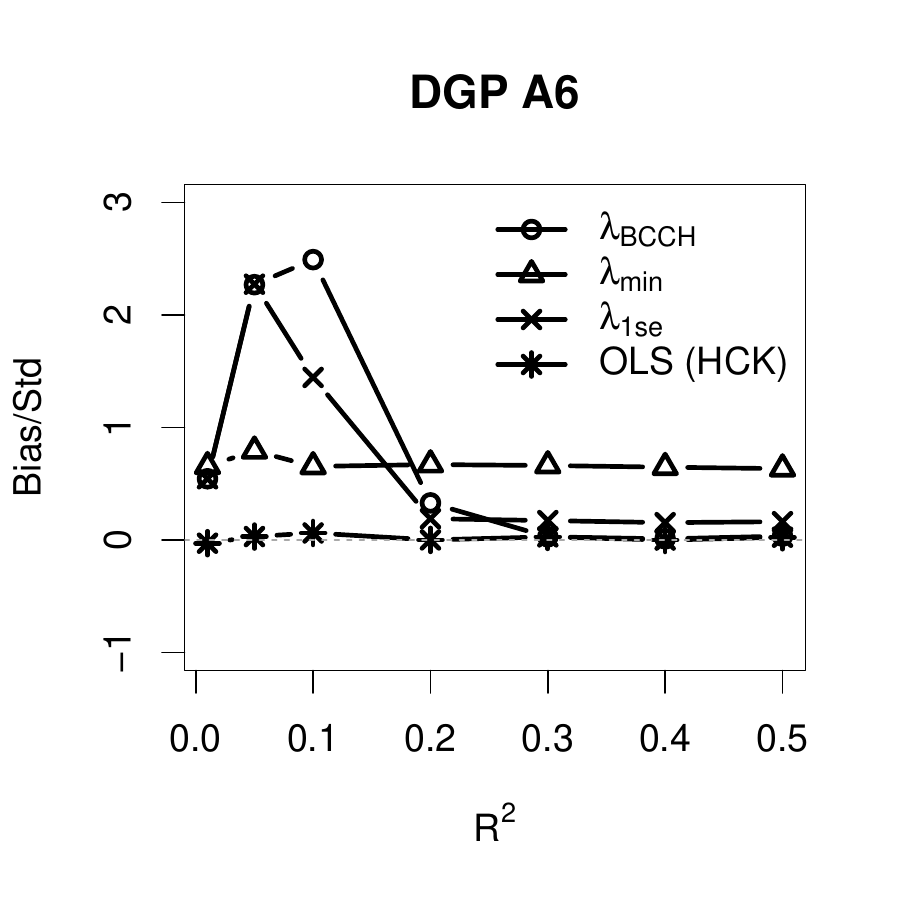}
\end{center}
\label{fig:dml_bias_std_app}
\end{figure}

\begin{figure}[ht]
\caption{Coverage 90\% confidence intervals}
\begin{center}
\includegraphics[width=0.325\textwidth,trim = {0 1cm 0 1cm}]{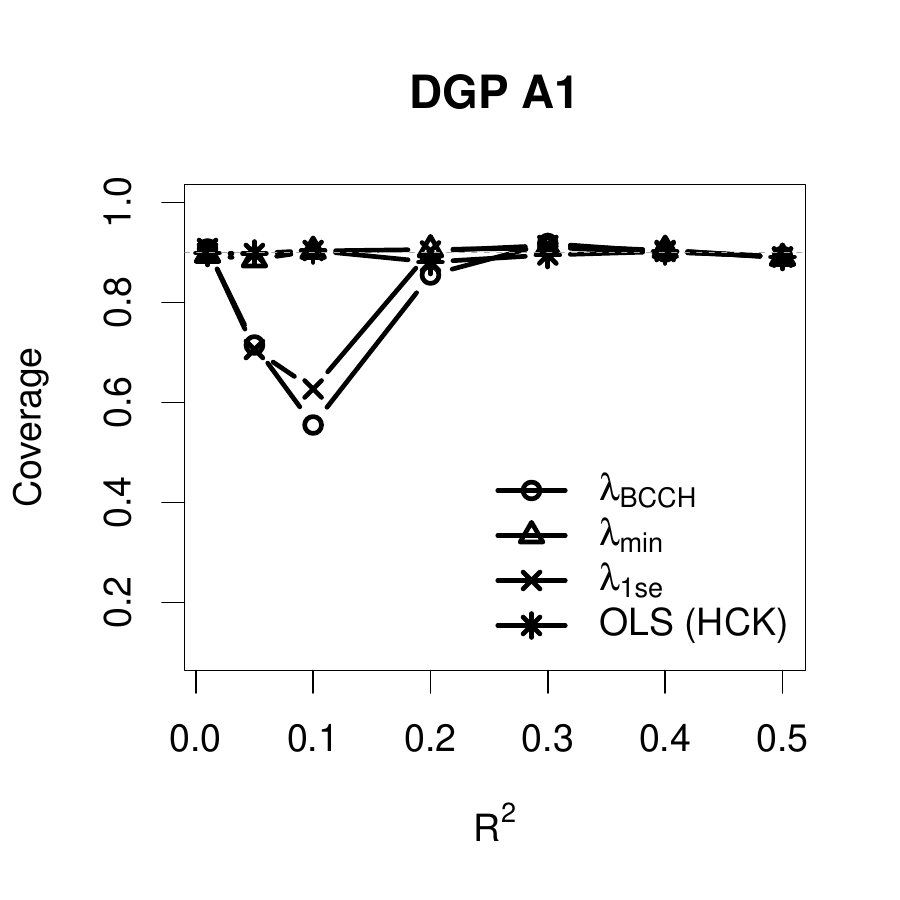}
\includegraphics[width=0.325\textwidth,trim = {0 1cm 0 1cm}]{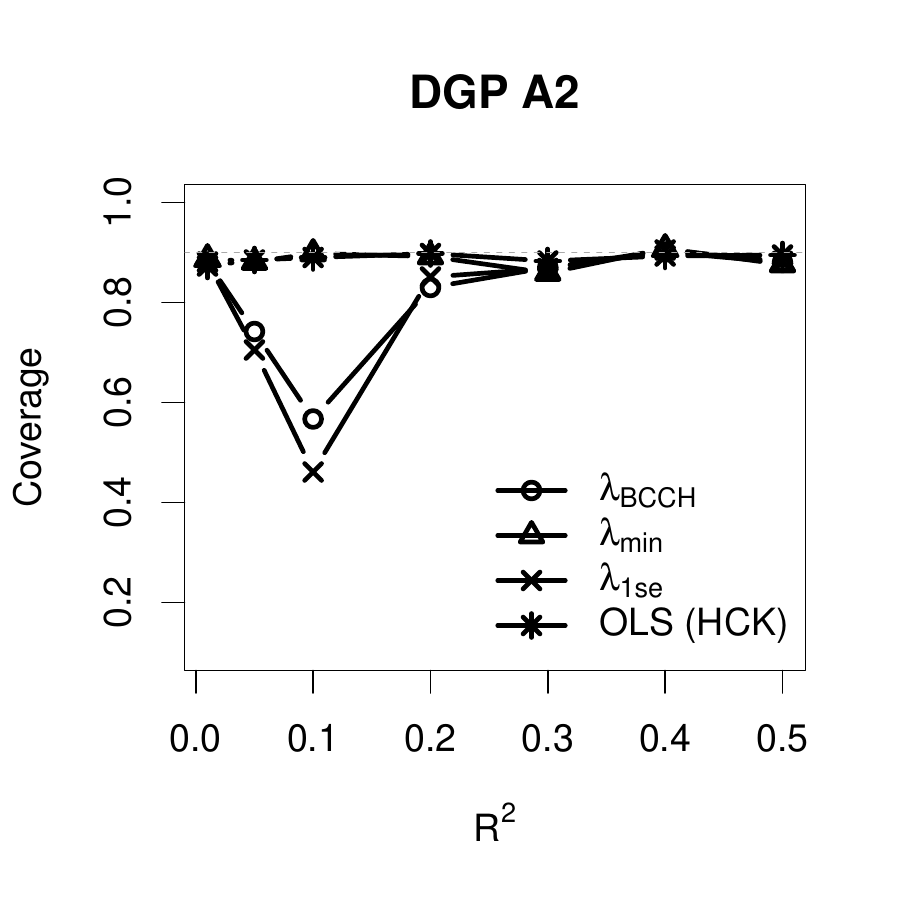}
\includegraphics[width=0.325\textwidth,trim = {0 1cm 0 1cm}]{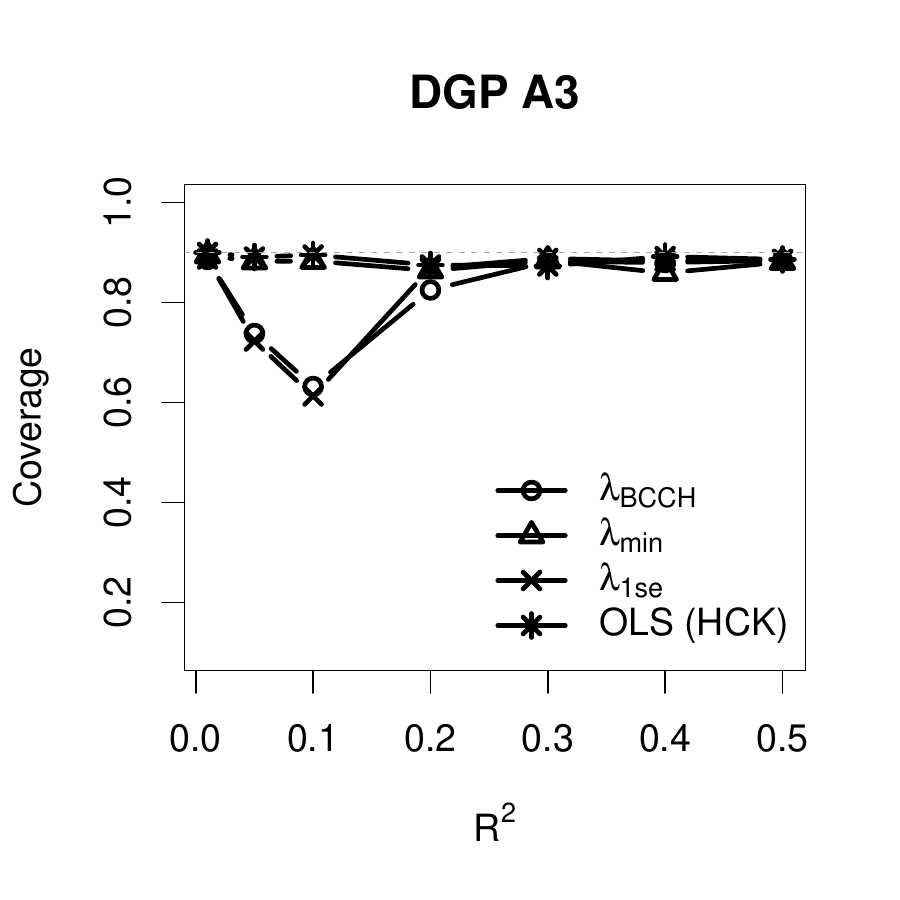}
\includegraphics[width=0.325\textwidth,trim = {0 1cm 0 .5cm}]{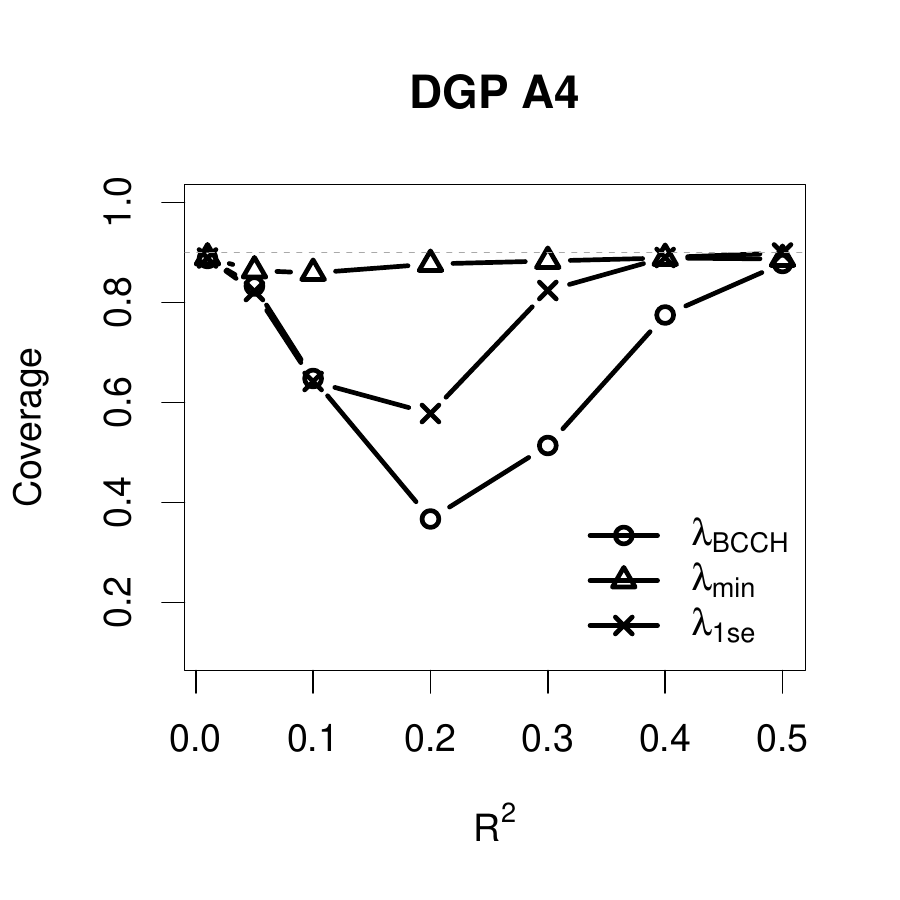}
\includegraphics[width=0.325\textwidth,trim = {0 1cm 0 .5cm}]{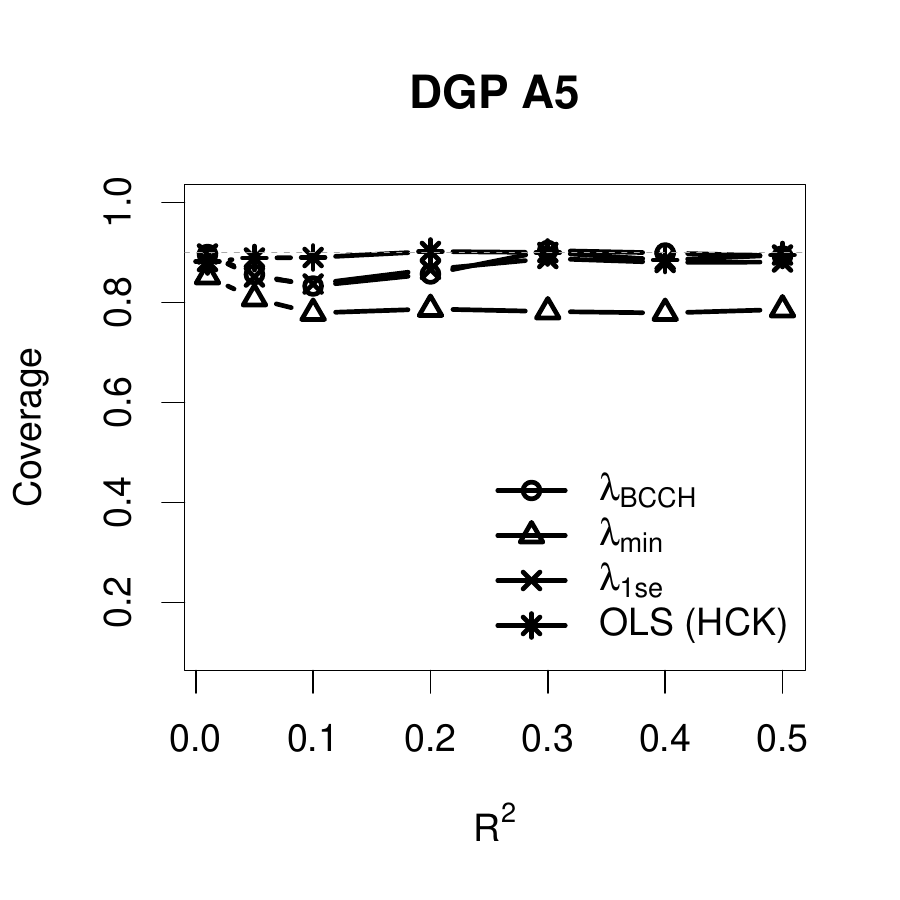}
\includegraphics[width=0.325\textwidth,trim = {0 1cm 0 .5cm}]{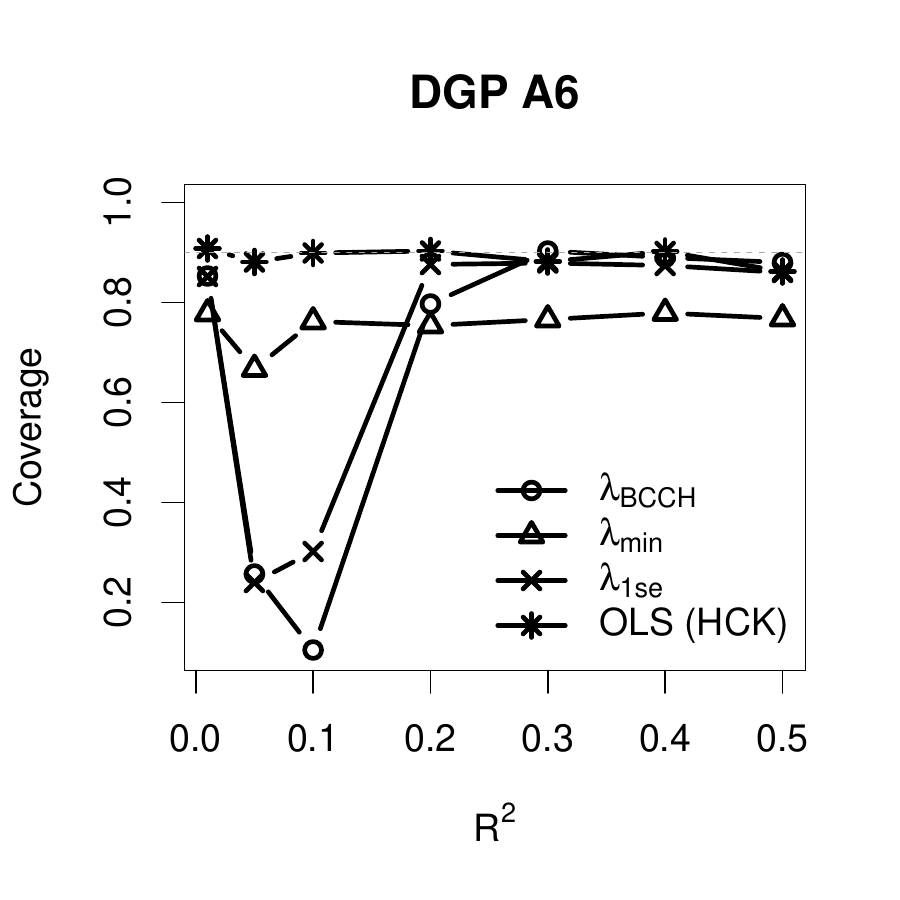}
\end{center}
\label{fig:dml_cov_app}
\end{figure}

\begin{figure}[ht]
\caption{Average length 90\% confidence intervals}
\begin{center}
\includegraphics[width=0.325\textwidth,trim = {0 1cm 0 1cm}]{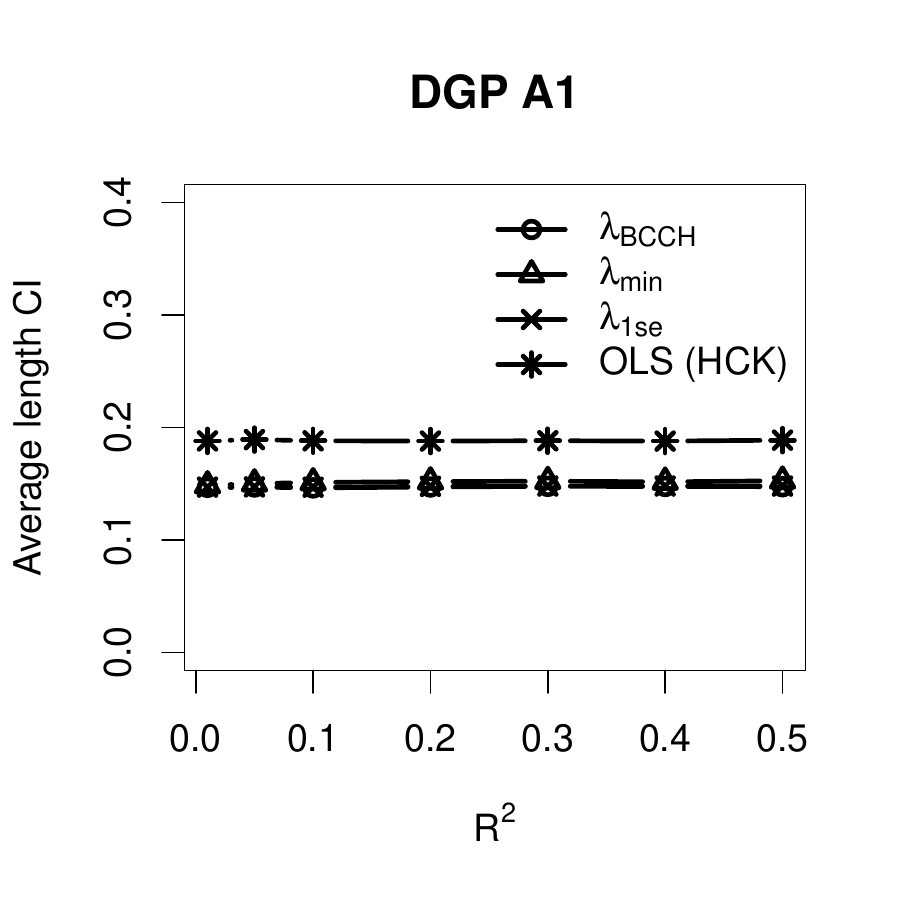}
\includegraphics[width=0.325\textwidth,trim = {0 1cm 0 1cm}]{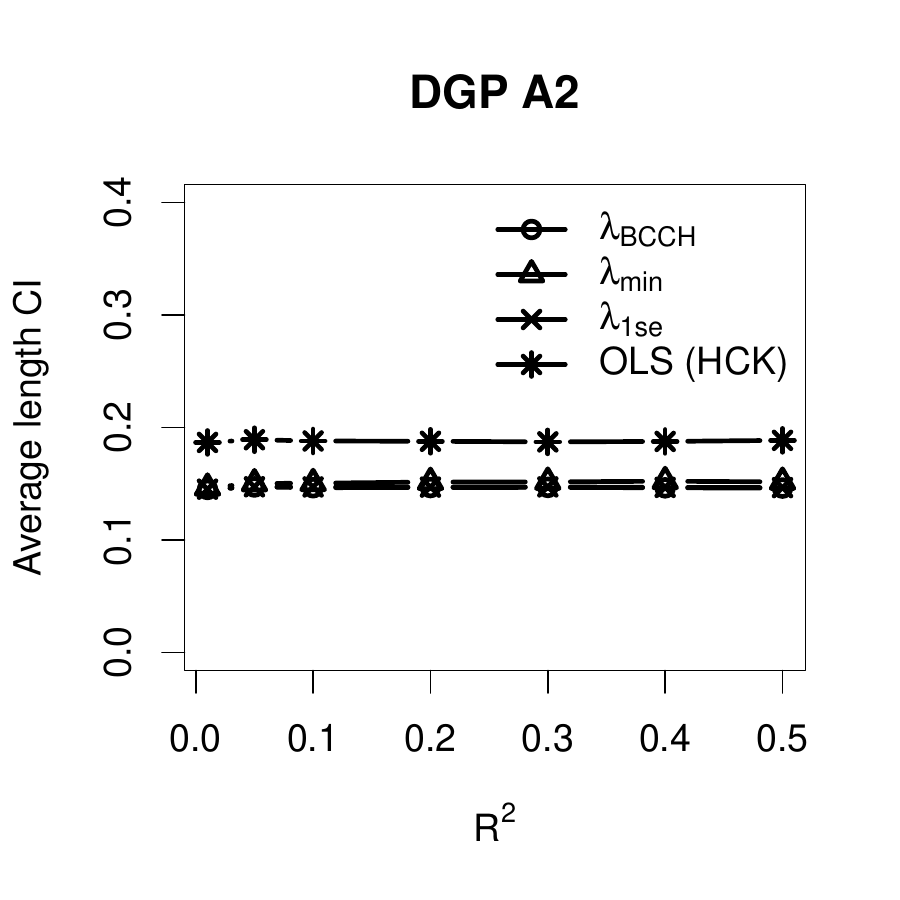}
\includegraphics[width=0.325\textwidth,trim = {0 1cm 0 1cm}]{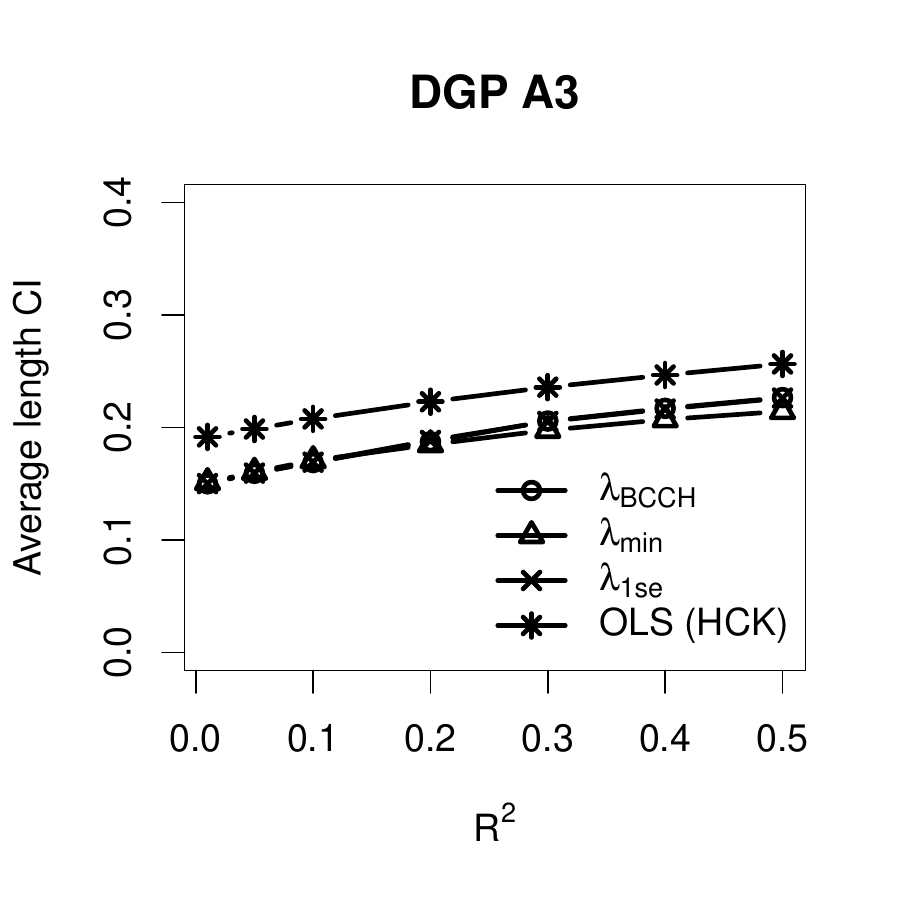}
\includegraphics[width=0.325\textwidth,trim = {0 1cm 0 .5cm}]{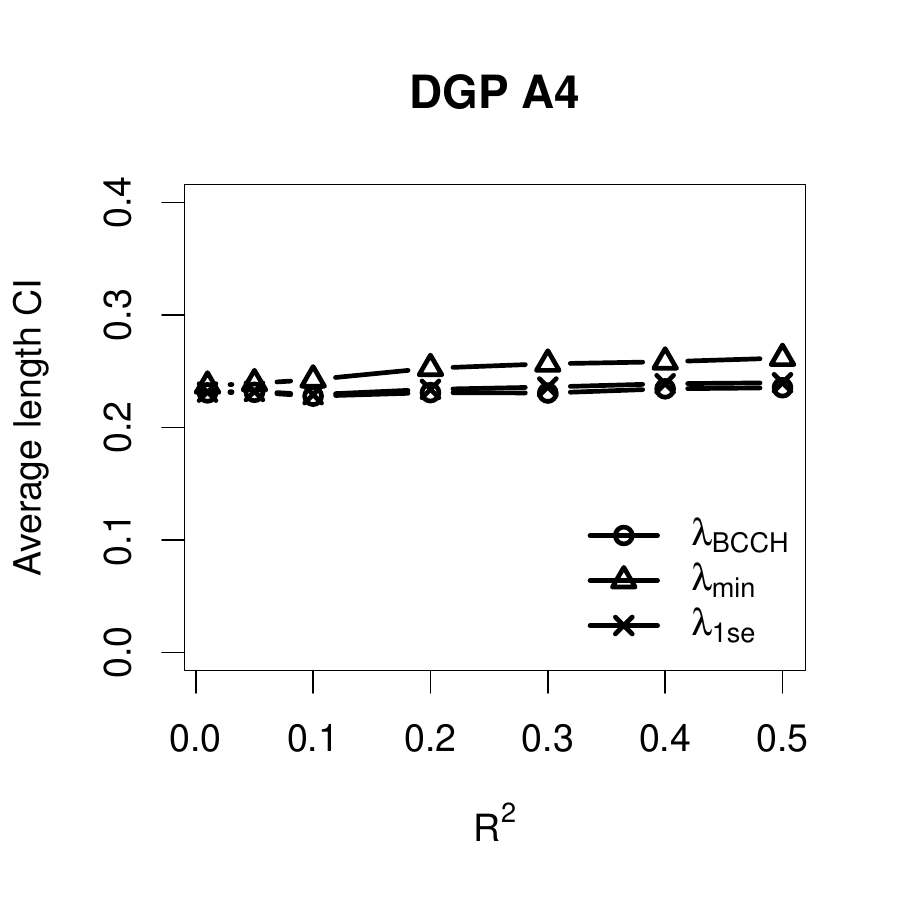}
\includegraphics[width=0.325\textwidth,trim = {0 1cm 0 .5cm}]{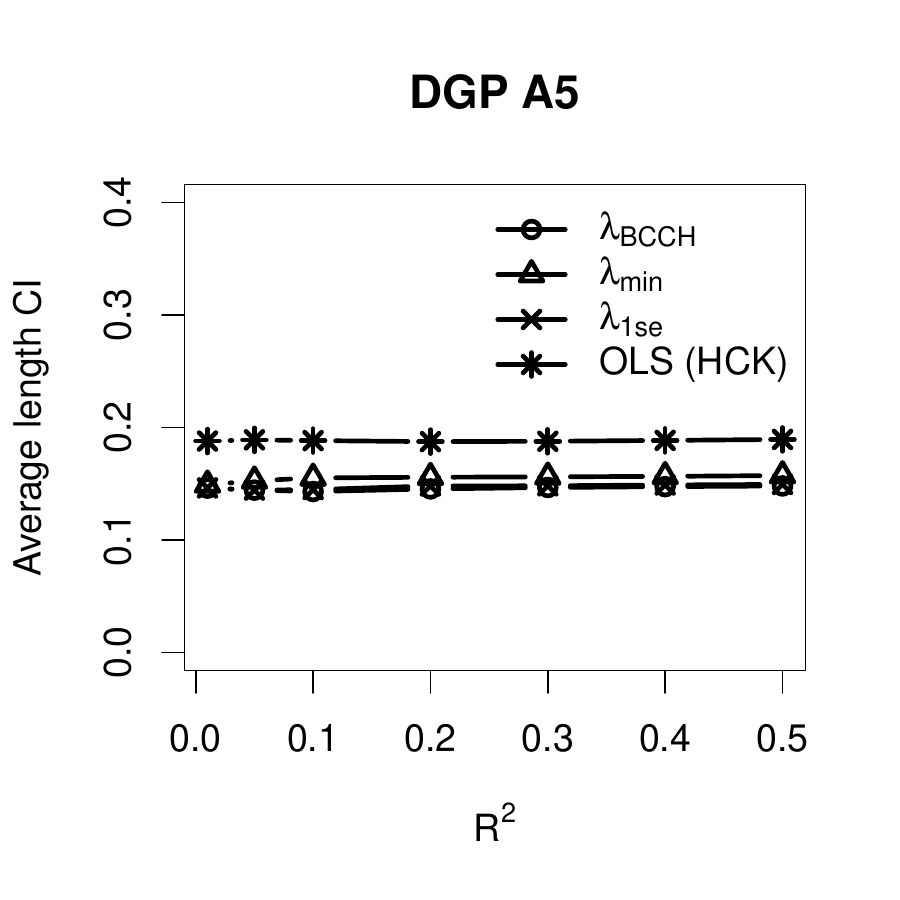}
\includegraphics[width=0.325\textwidth,trim = {0 1cm 0 .5cm}]{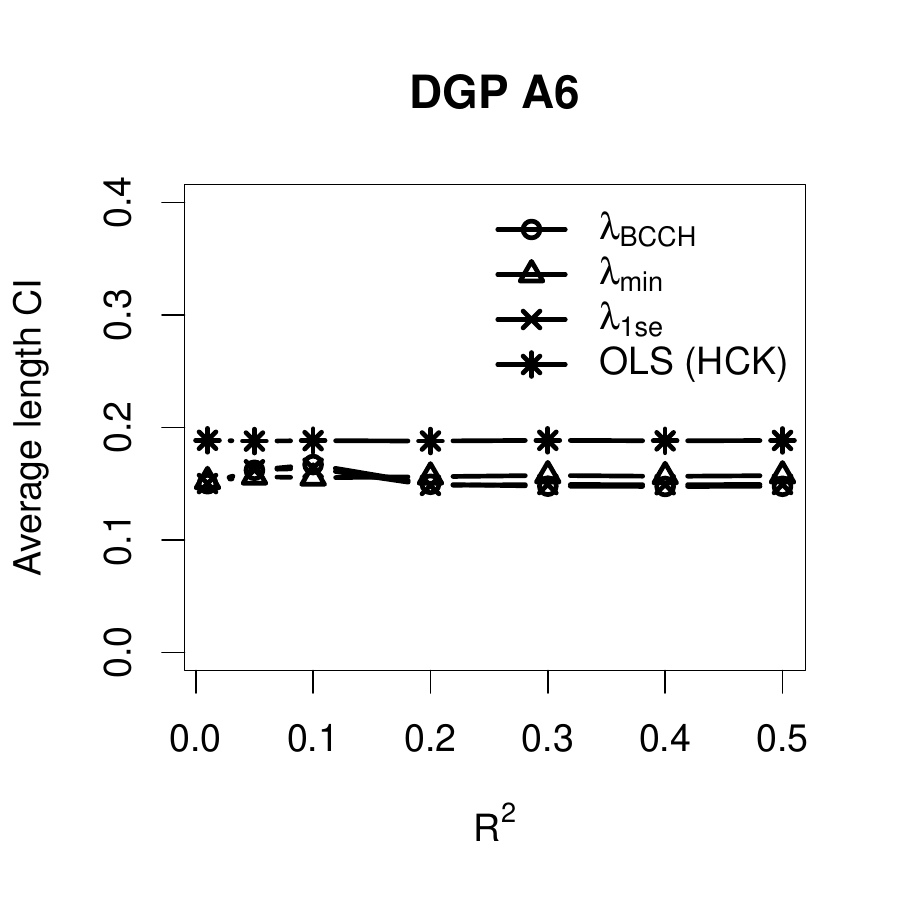}
\end{center}
\label{fig:dml_ci_app}
\end{figure}

\section{Theoretical results for random designs \label{sec:Random-design}}

\subsection{Results}

In this section, we provide some results in Lemma \ref{prop:random_design}
for the Lasso with a random design $X$. The necessary result on the
Lasso's inclusion established in Lemma \ref{prop:random_design} can
be adopted in a similar fashion as in Propositions \ref{prop:bias_post_double_selection}-\ref{prop:bias_post_double_selection-upper-1}
and \ref{prop:bias_debiased_lasso}-\ref{prop:bias_debiased_lasso-upper}
to establish the OVBs.

We make the following assumption about (\ref{eq:1}). \begin{assumption}
\label{ass:random_design} Each row of $X$ is sampled independently;
for all $i=1,\dots,n$ and $j=1,\dots,p$, $\sup_{r\geq1}r^{-\frac{1}{2}}\left(\mathbb{E}\left|X_{ij}\right|^{r}\right)^{\frac{1}{r}}\leq\alpha<\infty$;
for any unit vector $a\in\mathbb{R}^{k}$ and $i=1,\dots,n$, $\sup_{r\geq1}r^{-\frac{1}{2}}\left(\mathbb{E}\left|a^{T}X_{i,K}^{T}\right|^{r}\right)^{\frac{1}{r}}\leq\tilde{\alpha}<\infty$,
where $X_{i,K}$ is the $i$th row of $X_{K}$ and $K=\left\{ j:\,\theta_{j}^{*}\neq0\right\} $.
Moreover, the error terms $\varepsilon_{1},\dots,\varepsilon_{n}$
are independent such that $\sup_{r\geq1}r^{-\frac{1}{2}}\left(\mathbb{E}\left|\varepsilon_{i}\right|^{r}\right)^{\frac{1}{r}}\leq\sigma<\infty$
and $\mathbb{E}\left(X_{i}\varepsilon_{i}\right)=0_{p}$ for all $i=1,\dots,n$.
\end{assumption}

Assumption \ref{ass:random_design} is known as the sub-Gaussian tail
condition defined in \citet{vershynin2012}. Examples of sub-Gaussian
variables include Gaussian mixtures and distributions with bounded
support. The first and last part of Assumption \ref{ass:random_design}
imply that $X_{ij},j=1,\dots,p,$ and $\varepsilon_{i}$ are sub-Gaussian
variables and is used in deriving the lower bounds on the regularization
parameters. The second part of Assumption \ref{ass:random_design}
is only used to establish some eigenvalue condition on $\frac{X_{K}^{T}X_{K}}{n}$.
\begin{assumption}\label{ass:sparsity_incoherence_random_design}
The following conditions are satisfied: (i) $\theta^{*}$ is exactly
sparse with at most $k$ non-zero coefficients and $K\neq\emptyset$;
(ii) 
\begin{equation}
\left\Vert \left[\mathbb{E}\left(X_{K^{c}}^{T}X_{K}\right)\right]\left[\mathbb{E}\left(X_{K}^{T}X_{K}\right)\right]^{-1}\right\Vert _{\infty}=1-\phi\label{eq:14-1}
\end{equation}
for some $\phi\in(0,\,1]$ such that $\phi^{-1}\precsim1$; (iii)
$\mathbb{E}\left(X_{ij}\right)=0$ for all $j\in K$ and $\mathbb{E}\left(X_{j}^{T}X_{j}\right)\leq s$
for all $j=1,\dots,p$; (iv) 
\begin{eqnarray}
\max\left\{ \frac{\phi}{12(1-\phi)k^{\frac{3}{2}}},\,\frac{\phi}{6k^{\frac{3}{2}}},\,\frac{\phi}{k}\right\} \sqrt{\frac{\log p}{n}} & \leq & \alpha^{2}\quad\textrm{if }\phi\in\left(0,\,1\right),\label{eq:7-2}\\
\max\left\{ \frac{1}{6k^{\frac{3}{2}}},\,\frac{1}{k}\right\} \sqrt{\frac{\log p}{n}} & \leq & \alpha^{2}\quad\textrm{if }\phi=1,\label{eq:7-4}\\
\max\left\{ 2\tilde{\alpha}^{2},\,12\alpha^{2},\,1\right\} \sqrt{\frac{\log p}{n}} & \leq & \lambda_{\min}\left(\mathbb{E}\left[\frac{1}{n}X_{K}^{T}X_{K}\right]\right).\label{eq:7-3}
\end{eqnarray}
\end{assumption} Part (iv) of Assumption \ref{ass:sparsity_incoherence_random_design}
is imposed to ensure that 
\begin{eqnarray*}
\left\Vert \left(\frac{1}{n}X_{K}^{T}X_{K}\right)^{-1}-\left[\mathbb{E}\left(\frac{1}{n}X_{K}^{T}X_{K}\right)\right]^{-1}\right\Vert _{\infty} & \precsim & \frac{1}{\lambda_{\min}\left(\mathbb{E}\left[\frac{1}{n}X_{K}^{T}X_{K}\right]\right)},\\
\left\Vert \frac{1}{n}X_{K^{c}}^{T}X_{K}\left(\frac{1}{n}X_{K}^{T}X_{K}\right)^{-1}\right\Vert _{\infty} & \leq & 1-\frac{\phi}{2},
\end{eqnarray*}
with high probability. To gain some intuition for \eqref{eq:7-2}--\eqref{eq:7-3},
let us further assume $k\asymp1$, $X_{i}$ is normally distributed
for all $i=1,\dots,n$, and $\mathbb{E}\left(X_{K}^{T}X_{K}\right)$
is a diagonal matrix with the diagonal entries $\mathbb{E}\left(X_{j}^{T}X_{j}\right)=s\neq0$.
As a result, $\tilde{\alpha}=\alpha\asymp\sqrt{\frac{s}{n}}$ by the
definition of a sub-Gaussian variable (e.g., \citet{vershynin2012})
and \eqref{eq:7-2}--\eqref{eq:7-3} essentially require $\sqrt{\frac{\log p}{n}}\precsim\frac{s}{n}$.

Given 
\begin{equation}
\mathbb{P}\left(\left|\frac{X^{T}\varepsilon}{n}\right|_{\infty}\geq t\right)\leq2\exp\left(\frac{-nt^{2}}{c_{0}\sigma^{2}\alpha^{2}}+\log p\right).\label{eq:38}
\end{equation}
and $\lambda\geq\frac{c\alpha\sigma\left(2-\frac{\phi}{2}\right)}{\phi}\sqrt{\frac{\log p}{n}}$
for some sufficiently large universal constant $c>0$, we have 
\begin{equation}
\lambda\geq2\left|\frac{X^{T}\varepsilon}{n}\right|_{\infty}\label{eq:lambda-1}
\end{equation}
with probability at least $1-c^{'}\exp\left(-c^{''}\log p\right)$.

Define the following events 
\begin{eqnarray*}
\mathcal{E}_{1} & = & \left\{ \left|\frac{X^{T}\varepsilon}{n}\right|_{\infty}\precsim\frac{\alpha\sigma\left(2-\frac{\phi}{2}\right)}{\phi}\sqrt{\frac{\log p}{n}}\right\} ,\\
\mathcal{E}_{2} & = & \left\{ \lambda_{\max}(\hat{\Sigma}_{KK})\leq\frac{3}{2}\lambda_{\max}(\Sigma_{KK})\right\} ,\\
\mathcal{E}_{3} & = & \left\{ \left\Vert \left(\frac{1}{n}X_{K}^{T}X_{K}\right)^{-1}-\left[\mathbb{E}\left(\frac{1}{n}X_{K}^{T}X_{K}\right)\right]^{-1}\right\Vert _{\infty}\precsim\frac{1}{\lambda_{\min}\left(\mathbb{E}\left[\frac{1}{n}X_{K}^{T}X_{K}\right]\right)}\right\} ,\\
\mathcal{E}_{4} & = & \left\{ \left\Vert \frac{1}{n}X_{K^{c}}^{T}X_{K}\left(\frac{1}{n}X_{K}^{T}X_{K}\right)^{-1}\right\Vert _{\infty}\leq1-\frac{\phi}{2}\right\} .
\end{eqnarray*}
By (\ref{eq:38}), $\mathbb{P}\left(\mathcal{E}_{1}\right)\geq1-c^{'}\exp\left(-c^{''}\log p\right)$;
by (\ref{eq:S7}), $\mathbb{P}\left(\mathcal{E}_{2}\right)\geq1-c_{1}^{'}\exp\left(-c_{1}^{''}\log p\right)$;
by (\ref{eq:S4}), $\mathbb{P}\left(\mathcal{E}_{3}\right)\geq1-c_{2}^{'}\exp\left(-c_{2}^{''}\left(\frac{\log p}{k^{3}}\right)\right)$;
by (\ref{eq:77}), $\mathbb{P}\left(\mathcal{E}_{4}\right)\geq1-c_{3}^{''}\exp\left(-b\left(\frac{\log p}{k^{3}}\right)\right)$,
where $b$ is some positive constant that only depends on $\phi$
and $\alpha$.

\begin{lemma}\label{prop:random_design} Let Assumptions \ref{ass:random_design}
and \ref{ass:sparsity_incoherence_random_design} hold. We solve the
Lasso (\ref{eq:las}) with $\lambda\geq\frac{c\alpha\sigma\left(2-\frac{\phi}{2}\right)}{\phi}\sqrt{\frac{\log p}{n}}$
for some sufficiently large universal constant $c>0$. Suppose $\mathbb{E}\left[X_{K}^{T}X_{K}\right]$
is a positive definite matrix.

(i) Then, conditioning on $\mathcal{E}_{1}\cap\mathcal{E}_{4}$ (which
holds with probability at least $1-c_{1}\exp\left(-b\frac{\log p}{k^{3}}\right)$),
(\ref{eq:las}) has a unique optimal solution $\hat{\theta}$ such
that $\hat{\theta}_{j}=0$ for $j\notin K$.

(ii) With probability at least $1-c_{1}\exp\left(-b\frac{\log p}{k^{3}}\right)$,
\begin{equation}
\left|\hat{\theta}_{K}-\theta_{K}^{*}\right|_{2}\leq\frac{3\lambda\sqrt{k}}{\lambda_{\min}\left(\mathbb{E}\left[\frac{1}{n}X_{K}^{T}X_{K}\right]\right)}\label{eq:5gen-1}
\end{equation}
where $\theta_{K}=\left\{ \theta_{j}\right\} _{j\in K}$ and $b$
is some positive constant that only depends on $\phi$ and $\alpha$;
if $\mathbb{P}\left(\left\{ \textrm{supp}(\hat{\theta})=K\right\} \cap\mathcal{E}_{1}\cap\mathcal{E}_{2}\right)>0$,
conditioning on $\left\{ \textrm{supp}(\hat{\theta})=K\right\} \cap\mathcal{E}_{1}\cap\mathcal{E}_{2}$,
we must have 
\begin{equation}
\left|\hat{\theta}_{K}-\theta_{K}^{*}\right|_{2}\geq\frac{\lambda\sqrt{k}}{3\lambda_{\max}\left(\mathbb{E}\left[\frac{1}{n}X_{K}^{T}X_{K}\right]\right)}\geq\frac{\lambda\sqrt{k}}{3\sum_{j\in K}\left(\mathbb{E}\left[\frac{1}{n}X_{j}^{T}X_{j}\right]\right)}.\label{eq:4gen-1}
\end{equation}

(iii) If $\mathbb{E}\left(X_{K}^{T}X_{K}\right)$ is a diagonal matrix
with the diagonal entries $\mathbb{E}\left(X_{j}^{T}X_{j}\right)=s\neq0$,
then 
\begin{equation}
\left|\hat{\theta}_{j}-\theta_{j}^{*}\right|\leq\frac{7\lambda n}{4s}\qquad\forall j\in K\label{eq:5-1}
\end{equation}
with probability at least $1-c_{1}\exp\left(-b\frac{\log p}{k^{3}}\right)$;
if $\mathbb{P}\left(\left\{ \hat{\theta}_{j}\neq0,\,j\in K\right\} \cap\mathcal{E}_{1}\cap\mathcal{E}_{3}\cap\mathcal{E}_{4}\right)>0$,
conditioning on $\left\{ \hat{\theta}_{j}\neq0,\,j\in K\right\} \cap\mathcal{E}_{1}\cap\mathcal{E}_{3}\cap\mathcal{E}_{4}$,
we must have 
\begin{equation}
\left|\hat{\theta}_{j}-\theta_{j}^{*}\right|\geq\frac{\lambda n}{4s}\geq c_{0}\frac{\sigma}{\phi}\sqrt{\frac{n}{s}}\sqrt{\frac{\log p}{n}}.\label{eq:4-1}
\end{equation}

(iv) Suppose $K=\left\{ 1\right\} $ and $\mathbb{E}\left(X_{1}^{T}X_{1}\right)=s\neq0$.
If 
\begin{equation}
\left|\theta_{1}^{*}\right|\leq\frac{\lambda n}{4s},\label{eq:min-1}
\end{equation}
then we must have 
\begin{equation}
\mathbb{P}\left(\hat{\theta}=0_{p}\right)\geq1-c\exp\left(-b\log p\right).\label{eq:nec1}
\end{equation}

\end{lemma}

\begin{remark}The part $\frac{\lambda n}{4s}\geq c_{0}\frac{\sigma}{\phi}\sqrt{\frac{n}{s}}\sqrt{\frac{\log p}{n}}$
in bound (\ref{eq:4-1}) follows from the fact that $\alpha\succsim\sqrt{\frac{s}{n}}=\sqrt{\mathbb{E}\left(\frac{1}{n}\sum X_{ij}^{2}\right)}$
where $j\in K$. \end{remark}

%If we further assume $k\asymp1$ and $X_{i}$ is
%normally distributed for all $i=1,\dots,n$, then $\tilde{\alpha}=\alpha\asymp\sqrt{\frac{s}{n}}$
%and (\ref{eq:7-2})-(\ref{eq:7-3}) imply that $\sqrt{\frac{\log p}{n}}\precsim\frac{s}{n}$
%(recalling our discussion following Assumption \ref{ass:sparsity_incoherence_random_design}).
%Taking the worst case $\sqrt{\frac{\log p}{n}}\asymp\frac{s}{n}$
%and the optimal choice $\lambda\asymp\sqrt{\frac{s}{n}}\sqrt{\frac{\log p}{n}}$,
%(\ref{eq:4-1}) implies that 
%\[
%\left|\hat{\theta}_{j}-\theta_{j}^{*}\right|\succsim\sigma\left(\frac{\log p}{n}\right)^{\frac{1}{4}}
%\]
%conditioning on $\left\{ \hat{\theta}_{j}\neq0,\,j\in K\right\} \cap\mathcal{E}_{1}\cap\mathcal{E}_{3}\cap\mathcal{E}_{4}$;
%therefore, the minimax rate $\sqrt{\frac{k\log\frac{p}{k}}{n}}\asymp\sqrt{\frac{\log p}{n}}$
%(given $k\asymp1$) can no longer be attained by the Lasso. In the
%example where $K=1$, if (\ref{eq:min-1}) is satisfied, then Lasso
%sets $\hat{\theta}_{1}=0$ with high probability.

\subsection{Main proof for Lemma \ref{prop:random_design}}

In what follows, we let $\Sigma_{KK}:=\mathbb{E}\left[\frac{1}{n}X_{K}^{T}X_{K}\right]$,
$\hat{\Sigma}_{KK}:=\frac{1}{n}X_{K}^{T}X_{K}$, and $\lambda_{\min}\left(\Sigma\right)$
denote the minimum eigenvalue of the matrix $\Sigma$. The proof for
Proposition \ref{prop:random_design}(i) follows similar argument
as before but requires a few extra steps. In applying Lemma 7.23 from
Chapter 7.5 of \citet{wainwright_2019} to establish the uniqueness
of $\hat{\theta}$ upon the success of PDW construction, it suffices
to show that $\lambda_{\min}(\hat{\Sigma}_{KK})\geq\frac{1}{2}\lambda_{\min}(\Sigma_{KK})$
and this fact is verified in (\ref{eq:S7-1}) in the appendix. As
a consequence, the subproblem (\ref{eq:sub}) is strictly convex and
has a unique minimizer. The details that show the PDW construction
succeeds conditioning on $\mathcal{E}_{1}\cap\mathcal{E}_{4}$ (which
holds with probability at least $1-c_{1}\exp\left(-b\frac{\log p}{k^{3}}\right)$)
can be found in Lemma \ref{lem:A4} (where $b$ is some positive constant
that only depends on $\phi$ and $\alpha$).

To show (\ref{eq:5gen-1}), note that our choice of $\lambda$ and
$\left|\hat{\delta}_{K}\right|\leq1$ yield 
\[
\left|\Delta\right|\leq\left|\lambda\hat{\delta}_{K}\right|+\left|\frac{X_{K}^{T}\varepsilon}{n}\right|\leq\frac{3\lambda}{2}1_{k},
\]
which implies that $\left|\Delta\right|_{2}\leq\frac{3\lambda}{2}\sqrt{k}$.
Moreover, we can show 
\begin{eqnarray}
\left|\hat{\theta}_{K}-\theta_{K}^{*}\right|_{2} & = & \frac{\left|\left(\frac{X_{K}^{T}X_{K}}{n}\right)^{-1}\Delta\right|_{2}}{\left|\Delta\right|_{2}}\left|\Delta\right|_{2}\label{eq:30-1}\\
 & \leq & \frac{1}{\lambda_{\min}\left(\frac{1}{n}X_{K}^{T}X_{K}\right)}\frac{3\lambda}{2}\sqrt{k}.\nonumber 
\end{eqnarray}
Applying (\ref{eq:S7-1}) and the bound $\left|\Delta\right|_{2}\leq\frac{3\lambda}{2}\sqrt{k}$
yields the claim.

In showing (\ref{eq:4gen-1}) in (ii) and (\ref{eq:4-1}) in (iii),
we will condition on $\left\{ \textrm{supp}(\hat{\theta})=K\right\} \cap\mathcal{E}_{1}\cap\mathcal{E}_{2}$
and $\left\{ \hat{\theta}_{j}\neq0,\,j\in K\right\} \cap\mathcal{E}_{1}\cap\mathcal{E}_{3}\cap\mathcal{E}_{4}$,
respectively.

To show (\ref{eq:4gen-1}), note that in Step 2 of the PDW procedure,
$\hat{\delta}_{K}$ is chosen such that $\left|\hat{\delta}_{j}\right|=1$
for any $j\in K$ whenever $\textrm{supp}(\hat{\theta})=K$. Given
the choice of $\lambda$, we are ensured to have 
\[
\left|\Delta\right|\geq\left|\left|\lambda\hat{\delta}_{K}\right|-\left|\frac{X_{K}^{T}\varepsilon}{n}\right|\right|\geq\frac{\lambda}{2}1_{k},
\]
which implies that $\left|\Delta\right|_{2}\geq\frac{\lambda}{2}\sqrt{k}$.
Moreover, we can show 
\begin{equation}
\left|\hat{\theta}_{K}-\theta_{K}^{*}\right|_{2}=\frac{\left|\left(\frac{X_{K}^{T}X_{K}}{n}\right)^{-1}\Delta\right|_{2}}{\left|\Delta\right|_{2}}\left|\Delta\right|_{2}\geq\frac{1}{\lambda_{\max}\left(\frac{1}{n}X_{K}^{T}X_{K}\right)}\frac{\lambda}{2}\sqrt{k}.\label{eq:30}
\end{equation}
It remains to bound $\lambda_{\max}\left(\hat{\Sigma}_{KK}\right)$.
We first write 
\begin{eqnarray*}
\lambda_{\max}(\Sigma_{KK}) & = & \max_{||h^{'}||_{2}=1}\mu^{'T}\Sigma_{KK}\mu^{'}\\
 & = & \max_{||h^{'}||_{2}=1}\left[\mu^{'T}\hat{\Sigma}_{KK}\mu^{'}+\mu^{'T}(\Sigma_{KK}-\hat{\Sigma}_{KK})\mu^{'}\right]\\
 & \geq & \mu^{T}\hat{\Sigma}_{KK}\mu+\mu^{T}(\Sigma_{KK}-\hat{\Sigma}_{KK})\mu
\end{eqnarray*}
where $\mu\in\mathbb{R}^{k}$ is a unit-norm maximal eigenvector of
$\hat{\Sigma}_{KK}$. Applying Lemma \ref{lem:A1}(b) with $t=\tilde{\alpha}^{2}\sqrt{\frac{\log p}{n}}$
yields 
\[
\mu^{T}\left(\Sigma_{KK}-\hat{\Sigma}_{KK}\right)\mu\geq-\tilde{\alpha}^{2}\sqrt{\frac{\log p}{n}}
\]
with probability at least $1-c_{1}\exp\left(-c_{2}\log p\right)$,
provided that $\sqrt{\frac{\log p}{n}}\leq1$; therefore, $\lambda_{\max}(\Sigma_{KK})\geq\lambda_{\max}(\hat{\Sigma}_{KK})-\tilde{\alpha}^{2}\sqrt{\frac{\log p}{n}}$.
Because $\tilde{\alpha}^{2}\sqrt{\frac{\log p}{n}}\leq\frac{\lambda_{\max}(\Sigma_{KK})}{2}$
(implied by (\ref{eq:7-3})), we have 
\begin{equation}
\lambda_{\max}(\hat{\Sigma}_{KK})\leq\frac{3}{2}\lambda_{\max}(\Sigma_{KK})\label{eq:S7}
\end{equation}
with probability at least $1-c_{1}\exp\left(-c_{2}\log p\right)$.

As a consequence,

\[
\left|\hat{\theta}_{K}-\theta_{K}^{*}\right|_{2}\geq\frac{1}{\lambda_{\max}\left(\frac{1}{n}X_{K}^{T}X_{K}\right)}\frac{\lambda}{2}\sqrt{k}\geq\frac{1}{\lambda_{\max}(\Sigma_{KK})}\frac{\lambda}{3}\sqrt{k}.
\]
The second inequality in (\ref{eq:4gen-1}) simply follows from the
fact $\lambda_{\max}\left(\mathbb{E}\left[\frac{1}{n}X_{K}^{T}X_{K}\right]\right)\leq\sum_{j\in K}\left(\mathbb{E}\left[\frac{1}{n}X_{j}^{T}X_{j}\right]\right)$.

To show (\ref{eq:5-1}), note that 
\begin{eqnarray}
\left|\hat{\theta}_{K}-\theta_{K}^{*}\right|_{\infty} & \leq & \left|\hat{\Sigma}_{KK}^{-1}\frac{X_{K}^{T}\varepsilon}{n}\right|_{\infty}+\lambda\left\Vert \hat{\Sigma}_{KK}^{-1}\right\Vert _{\infty}\nonumber \\
 & \leq & \left\Vert \hat{\Sigma}_{KK}^{-1}\right\Vert _{\infty}\left|\frac{X_{K}^{T}\varepsilon}{n}\right|_{\infty}+\lambda\left\Vert \hat{\Sigma}_{KK}^{-1}\right\Vert _{\infty}\nonumber \\
 & \leq & \frac{3\lambda}{2}\left\Vert \hat{\Sigma}_{KK}^{-1}\right\Vert _{\infty}.\label{eq:10-1}
\end{eqnarray}
We then apply (\ref{eq:S4}) of Lemma \ref{lem:A2} in the appendix,
and the fact $\left\Vert \hat{\Sigma}_{KK}^{-1}\right\Vert _{\infty}-\left\Vert \Sigma_{KK}^{-1}\right\Vert _{\infty}\leq\left\Vert \hat{\Sigma}_{KK}^{-1}-\Sigma_{KK}^{-1}\right\Vert _{\infty}$
(so that $\left\Vert \hat{\Sigma}_{KK}^{-1}\right\Vert _{\infty}\leq\frac{7n}{6s}$);
putting everything yields the claim.

To show (\ref{eq:4-1}), we again carry over the argument in the proof
for Lemma \ref{prop:fixed_design}. Letting $M=\hat{\Sigma}_{KK}^{-1}-\Sigma_{KK}^{-1}$,
we have 
\begin{eqnarray*}
\left|\hat{\theta}_{K}-\theta_{K}^{*}\right| & = & \left|\left(\Sigma_{KK}^{-1}+M\right)\left[\left(\frac{X_{K}^{T}\varepsilon}{n}\right)-\lambda\hat{\delta}_{K}\right]\right|\\
 & \geq & \left|\Sigma_{KK}^{-1}\left[\left(\frac{X_{K}^{T}\varepsilon}{n}\right)+\lambda\hat{\delta}_{K}\right]\right|-\left|M\left[\left(\frac{X_{K}^{T}\varepsilon}{n}\right)-\lambda\hat{\delta}_{K}\right]\right|\\
 & \geq & \left|\Sigma_{KK}^{-1}\right|\left|\left|\lambda\hat{\delta}_{K}\right|-\left|\frac{X_{K}^{T}\varepsilon}{n}\right|\right|-\left\Vert M\right\Vert _{\infty}\left|\left(\frac{X_{K}^{T}\varepsilon}{n}\right)-\lambda\hat{\delta}_{K}\right|_{\infty}1_{k},
\end{eqnarray*}
where the third line uses the fact that $\Sigma_{KK}^{-1}$ is diagonal.

Note that as before, the choice of $\lambda$ stated in Lemma \ref{prop:random_design}
and the fact $\Sigma_{KK}^{-1}=\frac{n}{s}I_{k}$ yield 
\begin{eqnarray*}
\left|\hat{\theta}_{j}-\theta_{j}^{*}\right| & \geq & \frac{\lambda n}{2s}-\left\Vert M\right\Vert _{\infty}\left|\left(\frac{X_{K}^{T}\varepsilon}{n}\right)-\lambda\hat{\delta}_{K}\right|_{\infty}\\
 & \geq & \frac{\lambda n}{2s}-\frac{3}{2}\lambda\left\Vert M\right\Vert _{\infty}.
\end{eqnarray*}
By (\ref{eq:S4}) of Lemma \ref{lem:A2} in the appendix, with probability
at least $1-c_{1}\exp\left(-b\frac{\log p}{k^{3}}\right)$, $\left\Vert M\right\Vert _{\infty}\leq\frac{1}{6}\lambda_{\min}^{-1}(\Sigma_{KK})=\frac{n}{6s}$.

As a result, we have (\ref{eq:4-1}). The part $\frac{\lambda n}{4s}\geq c_{0}\frac{\sigma}{\phi}\sqrt{\frac{n}{s}}\sqrt{\frac{\log p}{n}}$
in bound (\ref{eq:4-1}) follows from the fact that $\alpha\succsim\sqrt{\frac{s}{n}}=\sqrt{\mathbb{E}\left(\frac{1}{n}\sum X_{ij}^{2}\right)}$
where $j\in K$.

To establish (\ref{eq:nec1}), we adopt argument similar to what is
used in showing (\ref{eq:nec}) by applying the KKT condition 
\[
\left(\frac{1}{n}X_{1}^{T}X_{1}\right)\left(\theta_{1}^{*}-\hat{\theta}_{1}\right)=\lambda\textrm{sgn}\left(\hat{\theta}_{1}\right)-\frac{X_{1}^{T}\varepsilon}{n}
\]
and defining $\mathcal{E}=\mathcal{E}_{1}\cap\mathcal{E}_{4}$.

\subsection{Additional technical lemmas and proofs}

In this section, we show that the PDW construction succeeds with high
probability in Lemma \ref{lem:A4}, which is proved using results
from Lemmas \ref{lem:A1}--\ref{lem:A3}. The derivations for Lemmas
\ref{lem:A2} and \ref{lem:A3} modify the argument in \citet{wainwright2009sharp}
and \citet{ravikumar2010} to make it suitable for our purposes. In
what follows, we let $\Sigma_{K^{c}K}:=\mathbb{E}\left[\frac{1}{n}X_{K^{c}}^{T}X_{K}\right]$
and $\hat{\Sigma}_{K^{c}K}:=\frac{1}{n}X_{K^{c}}^{T}X_{K}$. Similarly,
let $\Sigma_{KK}:=\mathbb{E}\left[\frac{1}{n}X_{K}^{T}X_{K}\right]$
and $\hat{\Sigma}_{KK}:=\frac{1}{n}X_{K}^{T}X_{K}$.

\begin{lemma}\label{lem:A1} (a) Let $\left(W_{i}\right)_{i=1}^{n}$
and $\left(W_{i}^{'}\right)_{i=1}^{n}$ consist of independent components,
respectively. Suppose there exist parameters $\alpha$ and $\alpha^{'}$
such that 
\begin{eqnarray*}
\sup_{r\geq1}r^{-\frac{1}{2}}\left(\mathbb{E}\left|W_{i}\right|^{r}\right)^{\frac{1}{r}} & \leq & \alpha,\\
\sup_{r\geq1}r^{-\frac{1}{2}}\left(\mathbb{E}\left|W_{i}^{'}\right|^{r}\right)^{\frac{1}{r}} & \leq & \alpha^{'},
\end{eqnarray*}
for all $i=1,\dots,n$. Then 
\begin{equation}
\mathbb{P}\left[\left|\frac{1}{n}\sum_{i=1}^{n}\left(W_{i}W_{i}^{'}\right)-\mathbb{E}\left[\frac{1}{n}\sum_{i=1}^{n}\left(W_{i}W_{i}^{'}\right)\right]\right|\geq t\right]\leq2\exp\left(-cn\left(\frac{t^{2}}{\alpha^{2}\alpha^{'2}}\wedge\frac{t}{\alpha\alpha^{'}}\right)\right).\label{eq:A37-1}
\end{equation}
(b) For any unit vector $v\in\mathbb{R}^{d}$, suppose there exists
a parameter $\tilde{\alpha}$ such that 
\[
\sup_{r\geq1}r^{-\frac{1}{2}}\left(\mathbb{E}\left|a^{T}Z_{i}^{T}\right|^{r}\right)^{\frac{1}{r}}\leq\tilde{\alpha},
\]
where $Z_{i}$ is the $i$th row of $Z\in\mathbb{R}^{n\times d}$,
then we have 
\[
\mathbb{P}(\left|\left|Zv\right|_{2}^{2}-\mathbb{E}\left(\left|Zv\right|_{2}^{2}\right)\right|\geq nt)\leq2\exp\left(-c^{'}n\left(\frac{t^{2}}{\tilde{\alpha}^{4}}\wedge\frac{t}{\tilde{\alpha}^{2}}\right)\right).
\]
\end{lemma} \begin{remark}Lemma \ref{lem:A1} is based on Lemma
5.14 and Corollary 5.17 in \citet{vershynin2012}. \end{remark}

\begin{lemma} \label{lem:A2} Suppose Assumption \ref{ass:random_design}
holds. For any $t>0$ and some constant $c>0$, we have

\begin{equation}
\mathbb{P}\left\{ \left\Vert \hat{\Sigma}_{K^{c}K}-\Sigma_{K^{c}K}\right\Vert _{\infty}\geq t\right\} \leq2(p-k)k\exp\left(-cn\left(\frac{t^{2}}{k^{2}\alpha^{4}}\wedge\frac{t}{k\alpha^{2}}\right)\right),\label{eq:S2}
\end{equation}
\begin{equation}
\mathbb{P}\left\{ \left\Vert \hat{\Sigma}_{KK}-\Sigma_{KK}\right\Vert _{\infty}\geq t\right\} \leq2k^{2}\exp\left(-cn\left(\frac{t^{2}}{k^{2}\alpha^{4}}\wedge\frac{t}{k\alpha^{2}}\right)\right).\label{eq:S3}
\end{equation}
Furthermore, if $k\geq1$, $\frac{\log p}{n}\leq1$, $\tilde{\alpha}^{2}\sqrt{\frac{\log p}{n}}\leq\frac{\lambda_{\min}(\Sigma_{KK})}{2}$,
and $\alpha^{2}\sqrt{\frac{\log p}{n}}\leq\frac{\lambda_{\min}(\Sigma_{KK})}{12}$,
we have

\begin{eqnarray}
\mathbb{P}\left\{ \left\Vert \hat{\Sigma}_{KK}^{-1}\right\Vert _{2}\leq\frac{2}{\lambda_{\min}(\Sigma_{KK})}\right\}  & \geq & 1-c_{1}^{'}\exp\left(-c_{2}^{'}\log p\right),\label{eq:S8}\\
\mathbb{P}\left\{ \left\Vert \hat{\Sigma}_{KK}^{-1}-\Sigma_{KK}^{-1}\right\Vert _{\infty}\leq\frac{1}{6\lambda_{\min}(\Sigma_{KK})}\right\}  & \geq & 1-c_{1}\exp\left(-c_{2}\left(\frac{\log p}{k^{3}}\right)\right).\label{eq:S4}
\end{eqnarray}
\end{lemma} 
\begin{proof}
Let $u_{j^{'}j}$ denote the element $(j^{'},\,j)$ of the matrix
difference $\hat{\Sigma}_{K^{c}K}-\Sigma_{K^{c}K}$. The definition
of the $l_{\infty}$matrix norm implies that 
\begin{eqnarray*}
\mathbb{P}\left\{ \left\Vert \hat{\Sigma}_{K^{c}K}-\Sigma_{K^{c}K}\right\Vert _{\infty}\geq t\right\}  & = & \mathbb{P}\left\{ \max_{j^{'}\in K^{c}}\sum_{j\in K}|u_{j^{'}j}|\geq t\right\} \\
 & \leq & (p-k)\mathbb{P}\left\{ \sum_{j\in K}|u_{j^{'}j}|\geq t\right\} \\
 & \leq & (p-k)\mathbb{P}\left\{ \exists j\in K\,\vert\,|u_{j^{'}j}|\geq\frac{t}{k}\right\} \\
 & \leq & (p-k)k\mathbb{P}\left\{ |u_{j^{'}j}|\geq\frac{t}{k}\right\} \\
 & \leq & (p-k)k\cdot2\exp\left(-cn\left(\frac{t^{2}}{k^{2}\alpha^{4}}\wedge\frac{t}{k\alpha^{2}}\right)\right),
\end{eqnarray*}
where the last inequality follows Lemma \ref{lem:A1}(a). Bound (\ref{eq:S3})
can be derived in a similar fashion except that the pre-factor $(p-k)$
is replaced by $k$.

To prove (\ref{eq:S4}), note that

\begin{eqnarray}
\left\Vert \hat{\Sigma}_{KK}^{-1}-\Sigma_{KK}^{-1}\right\Vert _{\infty} & = & \left\Vert \Sigma_{KK}^{-1}\left[\Sigma_{KK}-\hat{\Sigma}_{KK}\right]\hat{\Sigma}_{KK}^{-1}\right\Vert _{\infty}\nonumber \\
 & \leq & \sqrt{k}\left\Vert \Sigma_{KK}^{-1}\left[\Sigma_{KK}-\hat{\Sigma}_{KK}\right]\hat{\Sigma}_{KK}^{-1}\right\Vert _{2}\nonumber \\
 & \leq & \sqrt{k}\left\Vert \Sigma_{KK}^{-1}\right\Vert _{2}\left\Vert \Sigma_{KK}-\hat{\Sigma}_{KK}\right\Vert _{2}\left\Vert \hat{\Sigma}_{KK}^{-1}\right\Vert _{2}\nonumber \\
 & \leq & \frac{\sqrt{k}}{\lambda_{\min}(\Sigma_{KK})}\left\Vert \Sigma_{KK}-\hat{\Sigma}_{KK}\right\Vert _{2}\left\Vert \hat{\Sigma}_{KK}^{-1}\right\Vert _{2}.\label{eq:S5}
\end{eqnarray}
To bound $\left\Vert \Sigma_{KK}-\hat{\Sigma}_{KK}\right\Vert _{2}$
in (\ref{eq:S5}), we apply (\ref{eq:S3}) with $t=\frac{\alpha^{2}}{\sqrt{k}}\sqrt{\frac{\log p}{n}}$
and obtain 
\[
\left\Vert \hat{\Sigma}_{KK}-\Sigma_{KK}\right\Vert _{2}\leq\frac{\alpha^{2}}{\sqrt{k}}\sqrt{\frac{\log p}{n}},
\]
with probability at least $1-c_{1}\exp\left(-c_{2}\frac{\log p}{k^{3}}\right)$,
provided that $k^{-3}\frac{\log p}{n}\leq1$. To bound \textbf{$\left\Vert \hat{\Sigma}_{KK}^{-1}\right\Vert _{2}$}
in (\ref{eq:S5}), let us write 
\begin{eqnarray}
\lambda_{\min}(\Sigma_{KK}) & = & \min_{||\mu^{'}||_{2}=1}\mu^{'T}\Sigma_{KK}\mu^{'}\nonumber \\
 & = & \min_{||\mu^{'}||_{2}=1}\left[\mu^{'T}\hat{\Sigma}_{KK}\mu^{'}+\mu^{'T}(\Sigma_{KK}-\hat{\Sigma}_{KK})\mu^{'}\right]\nonumber \\
 & \leq & \mu^{T}\hat{\Sigma}_{KK}\mu+\mu^{T}(\Sigma_{KK}-\hat{\Sigma}_{KK})\mu\label{eq:S6}
\end{eqnarray}
where $\mu\in\mathbb{R}^{k}$ is a unit-norm minimal eigenvector of
$\hat{\Sigma}_{KK}$. We then apply Lemma \ref{lem:A1}(b) with $t=\tilde{\alpha}^{2}\sqrt{\frac{\log p}{n}}$
to show 
\[
\left|\mu^{T}\left(\Sigma_{KK}-\hat{\Sigma}_{KK}\right)\mu\right|\leq\tilde{\alpha}^{2}\sqrt{\frac{\log p}{n}}
\]
with probability at least $1-c_{1}^{'}\exp\left(-c_{2}^{'}\log p\right)$,
provided that $\sqrt{\frac{\log p}{n}}\leq1$. Therefore, $\lambda_{\min}(\Sigma_{KK})\leq\lambda_{\min}(\hat{\Sigma}_{KK})+\tilde{\alpha}^{2}\sqrt{\frac{\log p}{n}}$.
As long as $\tilde{\alpha}^{2}\sqrt{\frac{\log p}{n}}\leq\frac{\lambda_{\min}(\Sigma_{KK})}{2}$,
we have 
\begin{equation}
\lambda_{\min}(\hat{\Sigma}_{KK})\geq\frac{1}{2}\lambda_{\min}(\Sigma_{KK}),\label{eq:S7-1}
\end{equation}
and consequently (\ref{eq:S8}), 
\[
\left\Vert \hat{\Sigma}_{KK}^{-1}\right\Vert _{2}\leq\frac{2}{\lambda_{\min}(\Sigma_{KK})}
\]
with probability at least $1-c_{1}^{'}\exp\left(-c_{2}^{'}\log p\right)$.

Putting the pieces together, as long as $\frac{\alpha^{2}}{\lambda_{\min}(\Sigma_{KK})}\sqrt{\frac{\log p}{n}}\leq\frac{1}{12}$,
\begin{equation}
\left\Vert \hat{\Sigma}_{KK}^{-1}-\Sigma_{KK}^{-1}\right\Vert _{\infty}\leq\frac{\sqrt{k}}{\lambda_{\min}(\Sigma_{KK})}\frac{\alpha^{2}}{\sqrt{k}}\sqrt{\frac{\log p}{n}}\frac{2}{\lambda_{\min}(\Sigma_{KK})}\leq\frac{1}{6\lambda_{\min}(\Sigma_{KK})}\label{eq:19}
\end{equation}
with probability at least $1-c_{1}\exp\left(-c_{2}\frac{\log p}{k^{3}}\right)$. 
\end{proof}
\begin{lemma}\label{lem:A3} Let Assumption \ref{ass:random_design}
hold. Suppose 
\begin{equation}
\left\Vert \mathbb{E}\left[X_{K^{c}}^{T}X_{K}\right]\left[\mathbb{E}(X_{K}^{T}X_{K})\right]^{-1}\right\Vert _{\infty}=1-\phi\label{eq:S1}
\end{equation}
for some $\phi\in(0,\,1]$. If $k\geq1$ and 
\begin{eqnarray}
\max\left\{ \frac{\phi}{12(1-\phi)k^{\frac{3}{2}}},\,\frac{\phi}{6k^{\frac{3}{2}}},\,\frac{\phi}{k}\right\} \sqrt{\frac{\log p}{n}} & \leq & \alpha^{2}\quad\textrm{if }\phi\in\left(0,\,1\right),\label{eq:7-2-1}\\
\max\left\{ \frac{1}{6k^{\frac{3}{2}}},\,\frac{1}{k}\right\} \sqrt{\frac{\log p}{n}} & \leq & \alpha^{2}\quad\textrm{if }\phi=1,\label{eq:7-4-1}\\
\max\left\{ 2\tilde{\alpha}^{2},\,12\alpha^{2},\,1\right\} \sqrt{\frac{\log p}{n}} & \leq & \lambda_{\min}(\Sigma_{KK}),\label{eq:7-3-1}
\end{eqnarray}
then for some positive constant $b$ that only depends on $\phi$
and $\alpha$, we have 
\begin{equation}
\mathbb{P}\left[\left\Vert \frac{1}{n}X_{K^{c}}^{T}X_{K}\left(\frac{1}{n}X_{K}^{T}X_{K}\right)^{-1}\right\Vert _{\infty}\geq1-\frac{\phi}{2}\right]\leq c^{'}\exp\left(-b\left(\frac{\log p}{k^{3}}\right)\right).\label{eq:77}
\end{equation}
\end{lemma}
\begin{proof}
Using the decomposition in \citet{ravikumar2010}, we have\textbf{
\[
\hat{\Sigma}_{K^{c}K}\hat{\Sigma}_{KK}^{-1}-\Sigma_{K^{c}K}\Sigma_{KK}^{-1}=R_{1}+R_{2}+R_{3},
\]
}where\textbf{ }

\begin{eqnarray*}
R_{1} & = & \Sigma_{K^{c}K}\left[\hat{\Sigma}_{KK}^{-1}-\Sigma_{KK}^{-1}\right],\\
R_{2} & = & \left[\hat{\Sigma}_{K^{c}K}-\Sigma_{K^{c}K}\right]\Sigma_{KK}^{-1},\\
R_{3} & = & \left[\hat{\Sigma}_{K^{c}K}-\Sigma_{K^{c}K}\right]\left[\hat{\Sigma}_{KK}^{-1}-\Sigma_{KK}^{-1}\right].
\end{eqnarray*}
By (\ref{eq:S1}), we have $\left\Vert \Sigma_{K^{c}K}\Sigma_{KK}^{-1}\right\Vert _{\infty}=1-\phi$.
It suffices to show $\left\Vert R_{i}\right\Vert _{\infty}\leq\frac{\phi}{6}$\textbf{
}for $i=1,...,3$.

For $R_{1}$, note that 
\[
R_{1}=-\Sigma_{K^{c}K}\Sigma_{KK}^{-1}[\hat{\Sigma}_{KK}-\Sigma_{KK}]\hat{\Sigma}_{KK}^{-1}.
\]
Applying the facts $\left\Vert AB\right\Vert _{\infty}\leq\left\Vert A\right\Vert _{\infty}\left\Vert B\right\Vert _{\infty}$
and $\left\Vert A\right\Vert _{\infty}\leq\sqrt{a}\left\Vert A\right\Vert _{2}$
for any symmetric matrix $A\in\mathbb{R}^{a\times a}$, we can bound
$R_{1}$ in the following fashion: 
\begin{eqnarray*}
\left\Vert R_{1}\right\Vert _{\infty} & \leq & \left\Vert \Sigma_{K^{c}K}\Sigma_{KK}^{-1}\right\Vert _{\infty}\left\Vert \hat{\Sigma}_{KK}-\Sigma_{KK}\right\Vert _{\infty}\left\Vert \hat{\Sigma}_{KK}^{-1}\right\Vert _{\infty}\\
 & \leq & (1-\phi)\left\Vert \hat{\Sigma}_{KK}-\Sigma_{KK}\right\Vert _{\infty}\sqrt{k}\left\Vert \hat{\Sigma}_{KK}^{-1}\right\Vert _{2},
\end{eqnarray*}
where the last inequality uses (\ref{eq:S1}). If $\phi=1$, then
$\left\Vert R_{1}\right\Vert _{\infty}=0$ so we may assume $\phi<1$
in the following. Bound (\ref{eq:S8}) from the proof for Lemma \ref{lem:A2}
yields 
\[
\left\Vert \hat{\Sigma}_{KK}^{-1}\right\Vert _{2}\leq\frac{2}{\lambda_{\min}(\Sigma_{KK})}
\]
with probability at least $1-c_{1}\exp\left(-c_{2}\log p\right)$.
Now, we apply bound (\ref{eq:S3}) from Lemma \ref{lem:A2} with $t=\frac{\phi}{12(1-\phi)}\sqrt{\frac{\log p}{kn}}$
and obtain 
\[
\mathbb{P}\left[\left\Vert \hat{\Sigma}_{KK}-\Sigma_{KK}\right\Vert _{\infty}\geq\frac{\phi}{12(1-\phi)}\sqrt{\frac{\log p}{kn}}\right]\leq2\exp\left(-c\left(\frac{\phi^{2}\log p}{\alpha^{4}(1-\phi)^{2}k^{3}}\right)\right),
\]
provided $\frac{\phi}{12(1-\phi)\alpha^{2}k}\sqrt{\frac{\log p}{kn}}\leq1$.
Then, if $\sqrt{\frac{\log p}{n}}\leq\lambda_{\min}(\Sigma_{KK})$,
we are guaranteed that 
\[
\mathbb{P}\left[\left\Vert R_{1}\right\Vert _{\infty}\geq\frac{\phi}{6}\right]\leq2\exp\left(-c\left(\frac{\phi^{2}\log p}{\alpha^{4}(1-\phi)^{2}k^{3}}\right)\right)+c_{1}\exp\left(-c_{2}\log p\right).
\]
For $R_{2}$, note that 
\begin{eqnarray*}
\left\Vert R_{2}\right\Vert _{\infty} & \leq & \sqrt{k}\left\Vert \Sigma_{KK}^{-1}\right\Vert _{2}\left\Vert \hat{\Sigma}_{K^{c}K}-\Sigma_{K^{c}K}\right\Vert _{\infty}\\
 & \leq & \frac{\sqrt{k}}{\lambda_{\min}(\Sigma_{KK})}\left\Vert \hat{\Sigma}_{K^{c}K}-\Sigma_{K^{c}K}\right\Vert _{\infty}.
\end{eqnarray*}
If $\frac{\phi}{6\alpha^{2}k}\sqrt{\frac{\log p}{kn}}\leq1$ and $\sqrt{\frac{\log p}{n}}\leq\lambda_{\min}(\Sigma_{KK})$,
applying bound (\ref{eq:S2}) from Lemma \ref{lem:A2} with $t=\frac{\phi}{6}\sqrt{\frac{\log p}{kn}}$
yields 
\[
\mathbb{P}\left[\left\Vert R_{2}\right\Vert _{\infty}\geq\frac{\phi}{6}\right]\leq2\exp\left(-c\left(\frac{\phi^{2}\log p}{\alpha^{4}k^{3}}\right)\right).
\]
For $R_{3}$, applying (\ref{eq:S2}) with $t=\phi\sqrt{\frac{\log p}{n}}$
to bound $\left\Vert \hat{\Sigma}_{K^{c}K}-\Sigma_{K^{c}K}\right\Vert _{\infty}$
and (\ref{eq:S4}) to bound $\left\Vert \hat{\Sigma}_{KK}^{-1}-\Sigma_{KK}^{-1}\right\Vert _{\infty}$
yields

\[
\mathbb{P}\left[\left\Vert R_{3}\right\Vert _{\infty}\geq\frac{\phi}{6}\right]\leq c^{'}\left[\exp\left(-c\left(\frac{\phi^{2}\log p}{\alpha^{4}k^{3}}\right)\right)+\exp\left(-c\left(\frac{\log p}{k^{3}}\right)\right)\right],
\]
provided that $\frac{\phi}{\alpha^{2}k}\sqrt{\frac{\log p}{n}}\leq1$
and $\sqrt{\frac{\log p}{n}}\leq\lambda_{\min}(\Sigma_{KK})$.

Putting everything together, we conclude that 
\[
\mathbb{P}\left[\left\Vert \hat{\Sigma}_{K^{c}K}\hat{\Sigma}_{KK}^{-1}\right\Vert _{\infty}\geq1-\frac{\phi}{2}\right]\leq c^{'}\exp\left(-b\left(\frac{\log p}{k^{3}}\right)\right)
\]
for some positive constant $b$ that only depends on $\phi$ and $\alpha$. 
\end{proof}
\begin{lemma}\label{lem:A4} Let the assumptions in Lemmas \ref{lem:A2}
and \ref{lem:A3} hold. Suppose $\theta^{*}$ is exactly sparse with
at most $k$ non-zero coefficients and $K=\left\{ j:\,\theta_{j}^{*}\neq0\right\} \neq\emptyset$.
If we choose $\lambda\geq\frac{c\alpha\sigma\left(2-\frac{\phi}{2}\right)}{\phi}\sqrt{\frac{\log p}{n}}$
for some sufficiently large universal constant $c>0$, $\left|\hat{\delta}_{K^{c}}\right|_{\infty}\leq1-\frac{\phi}{4}$
with probability at least $1-c_{1}\exp\left(-b\frac{\log p}{k^{3}}\right)$,
where $b$ is some positive constant that only depends on $\phi$
and $\alpha$. \end{lemma} 
\begin{proof}
By construction, the subvectors $\hat{\theta}_{K}$, $\hat{\delta}_{K}$,
and $\hat{\delta}_{K^{c}}$ satisfy the zero-subgradient condition
in the PDW construction. With the fact that $\hat{\theta}_{K^{c}}=\theta_{K^{c}}^{*}=0_{p-k}$,
we have 
\begin{eqnarray*}
\hat{\Sigma}_{KK}\left(\hat{\theta}_{K}-\theta_{K}^{*}\right)-\frac{1}{n}X_{K}^{T}\varepsilon+\lambda\hat{\delta}_{K} & = & 0_{k},\\
\hat{\Sigma}_{K^{c}K}\left(\hat{\theta}_{K}-\theta_{K}^{*}\right)-\frac{1}{n}X_{K^{c}}^{T}\varepsilon+\lambda\hat{\delta}_{K^{c}} & = & 0_{p-k}.
\end{eqnarray*}
The equations above yields 
\begin{eqnarray*}
\hat{\delta}_{K^{c}} & = & -\frac{1}{\lambda}\hat{\Sigma}_{K^{c}K}\left(\hat{\theta}_{K}-\theta_{K}^{*}\right)+X_{K^{c}}^{T}\frac{\varepsilon}{n\lambda},\\
\hat{\theta}_{K}-\theta_{K}^{*} & = & \hat{\Sigma}_{KK}^{-1}\frac{X_{K}^{T}\varepsilon}{n}-\lambda\hat{\Sigma}_{KK}^{-1}\hat{\delta}_{K},
\end{eqnarray*}
which yields 
\[
\hat{\delta}_{K^{c}}=\left(\hat{\Sigma}_{K^{c}K}\hat{\Sigma}_{KK}^{-1}\right)\hat{\delta}_{K}+\left(X_{K^{c}}^{T}\frac{\varepsilon}{n\lambda}\right)-\left(\hat{\Sigma}_{K^{c}K}\hat{\Sigma}_{KK}^{-1}\right)X_{K}^{T}\frac{\varepsilon}{n\lambda}.
\]
Using elementary inequalities and the fact that $\left|\hat{\delta}_{K}\right|_{\infty}\leq1$,
we obtain 
\[
\left|\hat{\delta}_{K^{c}}\right|_{\infty}\leq\left\Vert \hat{\Sigma}_{K^{c}K}\hat{\Sigma}_{KK}^{-1}\right\Vert _{\infty}+\left|X_{K^{c}}^{T}\frac{\varepsilon}{n\lambda}\right|_{\infty}+\left\Vert \hat{\Sigma}_{K^{c}K}\hat{\Sigma}_{KK}^{-1}\right\Vert _{\infty}\left|X_{K}^{T}\frac{\varepsilon}{n\lambda}\right|_{\infty}.
\]

By Lemma \ref{lem:A3}, $\left\Vert \hat{\Sigma}_{K^{c}K}\hat{\Sigma}_{KK}^{-1}\right\Vert _{\infty}\leq1-\frac{\phi}{2}$
with probability at least $1-c^{'}\exp\left(-\frac{b\log p}{k^{3}}\right)$;
as a result, 
\begin{eqnarray*}
\left|\hat{\delta}_{K^{c}}\right|_{\infty} & \leq & 1-\frac{\phi}{2}+\left|X_{K^{c}}^{T}\frac{\varepsilon}{n\lambda}\right|_{\infty}+\left\Vert \hat{\Sigma}_{K^{c}K}\hat{\Sigma}_{KK}^{-1}\right\Vert _{\infty}\left|X_{K}^{T}\frac{\varepsilon}{n\lambda}\right|_{\infty}\\
 & \leq & 1-\frac{\phi}{2}+\left(2-\frac{\phi}{2}\right)\left|\hat{X}^{T}\frac{\varepsilon}{n\lambda}\right|_{\infty}.
\end{eqnarray*}
It remains to show that $\left(2-\frac{\phi}{2}\right)\left|X^{T}\frac{\varepsilon}{n\lambda}\right|_{\infty}\leq\frac{\phi}{4}$
with high probability. This result holds if $\lambda\geq\frac{4\left(2-\frac{\phi}{2}\right)}{\phi}\left|X^{T}\frac{\varepsilon}{n}\right|_{\infty}$.
In particular, Lemma \ref{lem:A1}(a) and a union bound imply that
\[
\mathbb{P}\left(\left|\frac{X^{T}\varepsilon}{n}\right|_{\infty}\geq t\right)\leq2\exp\left(\frac{-nt^{2}}{c_{0}\sigma^{2}\alpha^{2}}+\log p\right).
\]
Thus, under the choice of $\lambda$ in Lemma \ref{lem:A4}, we have
$\left|\hat{\delta}_{K^{c}}\right|_{\infty}\leq1-\frac{\phi}{4}$
with probability at least $1-c_{1}\exp\left(-b\frac{\log p}{k^{3}}\right)$. 
\end{proof}

\end{document}